\documentclass[reqno]{amsart}
\usepackage{relsize}
\usepackage{scalerel}
\usepackage[margin=1in]{geometry}
\usepackage{stackengine,wasysym}
\usepackage{todonotes}
\usepackage{xcolor}
\usepackage{mathrsfs}
\usepackage{dsfont}
\usepackage{mathtools}
\mathtoolsset{showonlyrefs}
\usepackage[hyperfootnotes=false,colorlinks=true,linkcolor = blue, urlcolor  = blue, citecolor = blue]{hyperref}
\usepackage[sort,nocompress]{cite}
\usepackage{float}
\usepackage{amsmath, amsthm, amssymb}
\usepackage{times}
\usepackage{color}
\usepackage{comment}
\usepackage[toc,page]{appendix}
\usepackage{verbatim}
\usepackage{enumerate}

\newcommand{\E}{\mathbb{E}}
\newcommand{\pa}{\partial}
\newcommand{\la}{\label}
\newcommand{\fr}{\frac}
\newcommand{\na}{\nabla}
\newcommand{\be}{\begin{equation}}
\newcommand{\ee}{\end{equation}}
\newcommand{\ba}{\begin{array}{l}}
\newcommand{\ea}{\end{array}}

\newcommand{\beg}{\begin}

\renewcommand{\l}{\Lambda}

\usepackage{MnSymbol,wasysym}
\newcommand{\N}{\mathbb N}

\newcommand{\R}{\mathbb R}

\newcommand{\mbF}{\mathbb F}
\def\ZZ{{\mathbb Z}}
\def\RR{{\mathbb R}}
\def\TT{{\mathbb T}}
\def\NN{\mathbb N}
\def\PP{\mathbb P}

\DeclareMathOperator*{\esssup}{ess\,sup}

\theoremstyle{plain}

\newtheorem{Thm}{Theorem}[section]

\newtheorem{prop}[Thm]{Proposition}

\theoremstyle{definition}
\newtheorem{defi}[Thm]{Definition}

\newtheorem{setting}{Setting}

\theoremstyle{remark}
\newtheorem{rem}{Remark}

\numberwithin{equation}{section}

\title[The Stochastic Nernst-Planck-Navier-Stokes System]
{On the Long-time Dynamics and Ergodicity of the Stochastic Nernst-Planck-Navier-Stokes System}

\author[E. Abdo]{Elie Abdo}
\address[E. Abdo]
{	Department of Mathematics \\
     University of California  \\
	Santa Barbara, CA 93106-3080, USA.} \email{elieabdo@ucsb.edu}

\author[R. Hu]{Ruimeng Hu}
\address[R. Hu]
{Department of Mathematics \\
Department of Statistics and Applied Probability \\
     University of California  \\
	Santa Barbara, CA 93106-3080, USA.} \email{rhu@ucsb.edu}

\author[Q. Lin]{Quyuan Lin}
\address[Q. Lin]
{	School of Mathematical and Statistical Sciences \\
Clemson University\\
Clemson, SC 29634, USA.} \email{quyuanl@clemson.edu}
\date{\today}

\begin{document}
\maketitle

\begin{abstract}
 {We consider an electrodiffusion model that describes the intricate interplay of multiple ionic species with a two-dimensional, incompressible, viscous fluid subjected to stochastic additive noise. This system involves nonlocal nonlinear drift-diffusion Nernst-Planck equations for ionic species and stochastic Navier-Stokes equations for fluid motion under the influence of electric and time-independent forces. Under the selective boundary conditions imposed on the concentrations, we establish the existence and uniqueness of global pathwise solutions to this system on smooth bounded domains. Our study also investigates long-time ionic concentration dynamics and explores Feller properties of the associated Markovian semigroup. In the context of equal diffusive species and under appropriate conditions, we demonstrate the existence of invariant ergodic measures supported on $H^2$. We then enhance the ergodicity results on periodic tori and obtain smooth invariant measures under a constraint on the initial spatial averages of the concentrations. The uniqueness of the invariant measures on periodic boxes and smooth bounded domains is further established when the noise forces sufficient modes, and the diffusivities of the species are large. Finally, in the case of two ionic species with equal diffusivities and valences of $1$ and $-1$, we study the rate of convergence of the Markov transition kernels to the invariant measure and obtain unconditional, unique exponential ergodicity for the model.}  
\end{abstract}



\section{Introduction}

Electrodiffusion in fluids is a physical phenomenon that describes the transport of ions driven by three main processes: advection by the fluid velocity, diffusion by the gradient of the ionic concentrations, and migration by the gradient of the electric field. Studies of electrodiffusion phenomena have been extensively addressed in different branches of science, bringing forth outstanding applications in the real world. In \cite{tan2016computational}, the relation between the dendrite formation on the anode surface of batteries and the transport of ions near the dendrite nucleation site is investigated, aiming at improving the performance and lifetime of batteries. Interpretations of the electrodiffusion-related processes that occur in neurophysiology lead to a more accurate understanding of the nervous system and, in particular, of neural tissues \cite{jasielec2021electrodiffusion}. 
{
For more related applications in neuroscience, see \cite{nicholson2000diffusion,qian1989electro,pods2013electrodiffusion,mori2009numerical,lopreore2008computational,koch2004biophysics,cole1965electrodiffusion,savtchenko2017electrodiffusion}.
}
The mass transport from a landfill site to its neighboring environment is determined by the electrodiffusion of contaminants, and considerable efforts have been dedicated to understanding the diffusion mechanism and, consequently, improving the engineering design of liner systems \cite{jungnickel2004coupled}. 
{Further tremendous applications of electrodiffusion occurrences in semiconductors \cite{biler1994debye,gajewski1986basic,mock1983analysis}, water purification, desalination, and ion separations \cite{alkhad2022electrochemical,zhu2020ion,lee2018diffusiophoretic,yang2019review,gao2014high,lee2016membrane}, and ion selective membranes \cite{davidson2016dynamical,goldman1989electrodiffusion} have been widely studied in the literature.}

\subsection{The Stochastic Nernst-Planck-Navier-Stokes System}

In this paper, we consider an electrodiffusion model that describes the nonlinear time evolution of $N$ ionic concentrations in a two-dimensional incompressible viscous fluid, forced by the electric field induced by the ions, some time-independent body forces, and an additive stochastic noise process. For each $i \in \left\{1, \dots, N \right\}$, the corresponding ionic concentration $c_i$ evolves according to a deterministic Nernst-Planck equation
\be \la{1}
\partial_t c_i + u \cdot \nabla c_i = D_i \mathrm{div} (\nabla c_i + z_ic_i\nabla \Phi),
\ee
where $D_i$ and $z_i$ are, respectively, the diffusivities and valences of the ionic species. The potential 
$\Phi:= \frac{e}{K_{\beta}T_k}\Psi$ 
is the normalization
of the electric potential $\Psi$ created by the total charge density 
\noeqref{eqn:charge-density}
\be \label{eqn:charge-density}
\rho = \sum\limits_{i=1}^{n} z_ic_i,
\ee 
and  obeys  the semi-linear Poisson equation 
\noeqref{eqn:poisson}
\be \label{eqn:poisson}
-\epsilon \Delta \Phi = \rho.
\ee 
Here, $e$ is the elementary charge, $K_{\beta}$ is the Boltzmann constant, $T_k$ denotes temperature, and
$ 
\epsilon := \frac{\mathcal{E}K_{\beta}T_k}{e^2} = c_0 \left(\sum\limits_{i=1}^{m} z_i^2\right)\lambda_D^2,
$
where the constant $\mathcal{E} > 0$ represents the dielectric permittivity of the solvent, the constant $c_0 > 0$ is a reference bulk concentration of ions,  and the constant 
$
\lambda_D := \sqrt{\frac{\mathcal{E}K_{\beta}T_k}{c_0e^2 \sum\limits_{i=1}^{m}z_i^2}}
$
is the Debye screening length. The velocity $u$ of the fluid {satisfies} the stochastic incompressible Navier-Stokes equation
\noeqref{eqn:NSE}
\be \label{eqn:NSE}
d u + u \cdot \nabla u dt - \nu \Delta u dt+ \nabla p dt
= - K_{\beta}T_k \rho \nabla \Phi dt+ f dt+ gdW,
\ee
{and} obeys the divergence-free condition 
\be \la{2}
\nabla \cdot u = 0.
\ee 
Here $p$ represents the pressure of the fluid, {$W(t)$} is a standard Wiener process, and $\nu$ denotes the kinematic viscosity. The forces $f$ and $g$ are assumed to be time-independent and divergence-free.

As discussed in \cite{constantin2019nernst}, there are mainly two types of boundary conditions for the NPNS system, the blocking boundary conditions and the selective boundary conditions.
Blocking boundary conditions refer to ions that are not allowed to cross the boundary of the domain, in which case they have a vanishing boundary normal flux. Selective boundary conditions refer to ions that may cross some parts of the boundary while being blocked from crossing other parts. In our case, we consider uniform selective boundary conditions where the term ``uniform" characterizes the space-time independent constant values of the potential and the first $M$ ionic species on the boundary. Specifically, 
we study the model \eqref{1}--\eqref{2} on a bounded domain $\mathcal{O}\subset \mathbb R^2$ with a smooth boundary, equipped with the following mixed boundary {(uniform selective boundary)} conditions
\begin{equation}\la{3}
\begin{split}
    &u|_{\pa \mathcal{O}} = 0,
    \,\,\Phi|_{\pa \mathcal{O}} = \gamma, \, \,c_i|_{\pa \mathcal{O}} = \gamma_i \,\text{ for }\,i= 1, \dots, M, \\
    &\left(\na c_i + z_i c_i \na\Phi\right)|_{\pa \mathcal{O}} \cdot n = 0 \,\text{ for }\, i = M+1, \dots, N,
\end{split}
\end{equation}
 where $\gamma, \gamma_1, \dots, \gamma_m$ are positive constants, and $n$ is the outward unit normal to $\pa \mathcal{O}$; and on the two-dimensional torus $\TT^2 = [0,2\pi]^2$, equipped with periodic boundary conditions. In addition, the force $g$ is assumed to be zero on the boundary, i.e., $g|_{\pa \mathcal{O}} = 0$.
We denote the initial data by 
\be \la{4}
u(x,0) = u_0, \; c_i(x,0) = c_i(0).
\ee The initial boundary value problem described by equations \eqref{1}--\eqref{4} is called the stochastic Nernst-Planck-Navier-Stokes system and abbreviated by S-NPNS throughout the paper. Furthermore, we take the physical constants $\epsilon, \nu, K_{\beta}$, and $T_k$ to be $1$ for the sake of simplicity. 

\begin{rem}
  For $i=M+1,\dots, N$, the spatial mean of each concentration $c_i$ is conserved in time and amounts to the initial spatial average 
$
\bar{c}_i(0) := \frac{1}{|\mathcal{O}|}\int_{\mathcal{O}} c_i(x,0) dx,$ a key property of the model that is frequently exploited in the analysis of its features. This fact follows from integrating the ionic concentration equation over $\mathcal{O}$ and making use of the boundary condition \eqref{3} for $i=M+1,\dots, N$, together with the divergence-free condition \eqref{2}. Based on this observation, we use the notation $\bar{c}_i$ throughout the paper to denote the constant average of $c_i$ over $\mathcal{O}$ for any time $t\geq 0$. 
\end{rem}

Different mathematical challenges arise, not only from the nonlinear and nonlocal aspects of the model, but also from the boundary effects, the number of ionic species, and the values of their valences and diffusivities.  Generally, the following four settings are explored in various contexts and arranged below by their level of generality:
\begin{enumerate}[\bf A:]\itemsep-0.1em
    \item\label{setting A} $N$ ionic species with arbitrary valences and diffusivities; 
    \item\label{setting B} $N$ ionic species with the same diffusivities but arbitrary valences;
    \item\label{setting C} $N$ ionic species with the same diffusivities and the same absolute values of the valences;
    \item\label{setting D} Two ionic species with the same diffusivities and valences $1$ and $-1$.
\end{enumerate}

\subsection{Literature for the Deterministic Case} The deterministic unforced Nernst-Planck-Navier-Stokes (NPNS) system has been widely studied over the last decade in the presence and absence of physical boundaries. In \cite{ryham2009existence}, the existence, uniqueness, and long-time behavior of solutions were obtained under Setting \ref{setting D} for $L^2$ large and small initial data on 2D and 3D smooth bounded domains respectively, provided that the concentrations have blocking (no flux) boundary conditions and the electric potential vanishes on the boundary. Those aforementioned results were also established in \cite{schmuck2009analysis} for homogeneous Neumann boundary conditions imposed on the potential.  
In \cite{constantin2019nernst}, the authors considered the NPNS model in the most general setting \ref{setting A} on 2D bounded smooth domains with selective boundary conditions and obtained the global regularity of solutions for $W^{2,p}$ initial data and proved their convergence to stable steady states. 
Under a regularity condition imposed on the velocity, global regular solutions were obtained in \cite{constantin2021nernst} on 3D smooth bounded domains for selective boundary conditions under Setting \ref{setting C} where the ions have valences $z_i \in \left\{1, - 1\right\}$ and Setting \ref{setting D}  for two species with opposite valences. 
Regarding Boltzmann states' stability, \cite{constantin2022nernst} illustrated nonlinear stability in both 2D and 3D bounded domains under certain boundary conditions, and instabilities have been examined in simplified models through mathematical and numerical approaches \cite{rubinstein2000electro,zaltzman2007electro}. Furthermore, empirical evidence of these instabilities occurring under selective boundary conditions was reported in \cite{rubinstein2008direct}. The existence of a global unique smooth solution was established in \cite{abdolong} on $d$-dimensional periodic boxes in Setting \ref{setting A} and the exponential stability of solutions was shown when $d=2$. 
{In addition, the analyticity of the solutions was established in \cite{abdo2022space}.
}

\subsection{Main Results} Under different assumptions imposed on the size of the initial data, the size of the stochastic and deterministic forcing, the parameters of the problem, the boundary data, or the geometry of the domain, we address the following four main questions:
\begin{enumerate}
    \item[(I)] The global well-posedness of the stochastic Nernst-Planck-Navier-Stokes (S-NPNS) system;
    \item[(II)] The long-time behavior of the ionic concentrations in $L^p$ spaces;
    \item[(III)] The existence, smoothness, and uniqueness of invariant ergodic measures for the Markov transition kernels associated with the model;
    \item[(IV)] The rate of convergence of the Markov kernels to the unique invariant measure. 
\end{enumerate}

\smallskip
\noindent{\bf{Global Well-Posedness of the Model.}} The first main result of this paper concerns the existence and uniqueness of global pathwise solutions on bounded domains with mixed boundary conditions (Theorem \ref{glos}). In contrast with existing results in the presence of boundaries where $H^2$ Sobolev initial regularity is imposed on the initial data, we present an iterative scheme that yields unique weak-strong solutions for $L^2$ initial velocity and concentrations, in spite of the additive stochastic forcing. The constructed iteration gives rise to a locally unique solution on a short time interval $[0,T_0]$. Via a blow-up criterion, we extend the local solution from $[0, T_0]$ globally to any time interval $[0.T]$, provided that the initial ionic concentrations are nonnegative. Although we restrict ourselves to the case of a positive constant boundary potential, the approximating scheme presented in Section~\ref{sec:existence} also works in the case of spatially dependent Dirichlet boundary conditions imposed on $\Phi$. This boundary restriction is needed throughout the paper to investigate the asymptotic behavior of the concentrations and ergodicity of the model.


\smallskip

\noindent{\bf{Long-Time Dynamics of the Ionic Concentrations.}} 
Our second set of main results addresses the rate of convergence of the ionic concentrations $c_1, \dots, c_N$ to their steady states (which are constants $\gamma_i$ for $i=1,\dots,M$ and $\bar c_i$ for $i=M+1,\dots,N$) when the relation 
\be \la{44}
\sum\limits_{i=1}^{M} z_i \gamma_i +  \sum\limits_{i=M+1}^{N} z_i \bar{c}_i(0) = 0.
\ee 
is imposed. Under \eqref{44}, we prove that the $L^2$ norm of each concentration $c_i$ decays exponentially fast in time to $\gamma_i$ for $i \in \left\{1, \dots, M\right\}$ and $\bar{c}_i$ for $i \in \left\{M+1, \dots, N\right\}$ under the assumption that the initial and boundary data are sufficiently small in $L^2$ (Theorem \ref{longt}). Furthermore, given an even integrability exponent $p$, we present a new proof by induction by which we obtain the exponential decay in the spatial $L^p$ norm with a rate depending on that in the $L^{p-2}$ norm for any large initial concentration in $L^p$. A bootstrapping argument yields consequently the exponential convergence in time of all ionic concentrations to constant values in $L^p$ spaces for a small initial datum in $L^2$ and small boundary values $\gamma_1, \dots, \gamma_{M}$ (Theorem \ref{t4}). 

The relation \eqref{44} is motivated by a result of the deterministic unforced Nernst-Planck-Navier-Stokes system \cite{constantin2019nernst} $(f=g=0)$ under selective boundary conditions obeyed by the ionic concentrations and Dirichlet conditions obeyed by the potential $\Phi|_{\partial \mathcal O} = \gamma(x)$ with $\gamma(x)$ depending on space. It is proved therein that the ionic concentrations $c_1, \dots, c_N$ and potential $\Phi$ converge in time to steady states $c_1^*, \dots, c_N^*$ and $\Phi^*$ respectively with 
\begin{align}
-\epsilon \Delta \Phi^* = \sum\limits_{i=1}^{N} z_i c_i^* \quad &\text{with} \quad  \Phi^*|_{\partial \mathcal O} = \gamma, \\ 
c_i^* = \gamma_i e^{z_i\gamma} e^{-z_i\Phi^*} \text{ for } i \in \left\{1, \dots, M \right\} \quad &\text{and} \quad 
c_i^* = \fr{\int_{\mathcal{O}} c_i(x,0)dx}{\int_{\mathcal{O}} e^{-z_i \Phi^*} dx} e^{-z_i \Phi^*} \text{ for } i \in \left\{M+1, \dots, N\right\}.
\end{align} 
Our setting that the potential $\Phi$ restricted to the boundary is a constant $\gamma$ independent of the spatial variable (cf. \eqref{3}) forces the limiting potential $\Phi^*$ to be  $\gamma$ everywhere in $\mathcal{O}$ and its Laplacian to vanish, thus naturally gives rise to \eqref{44}.

\smallskip

\noindent{\bf{Ergodicity on Bounded Domains.}} The Markov transition functions associated with the initial boundary value problem \eqref{1}--\eqref{4} are well-defined instantaneously in time on the largest space where the uniqueness of probabilistically strong solutions is guaranteed which, in our case, turns out to be the subspace of square-integrable vector fields $(v, \xi_1, \dots, \xi_N)$ where $v$ is divergence-free with Dirichlet boundary conditions, $\xi_1, \dots, \xi_M$ are nonnegative and amounts to $\gamma_1, \dots, \gamma_M$ respectively on $\pa \mathcal{O}$, and $\gamma_{M+1}, \dots, \gamma_N$ are nonnegative with blocking boundary conditions. 

In the third set of results, we first define the corresponding Markovian semigroup and investigate its Feller properties in the most general setting \ref{setting A} via continuous dependency estimates (Theorem \ref{feller}). In contrast with the 2D Navier-Stokes equations where $L^2$ cancellations law reduces the influence of the nonlinearities, a few challenges arise from the analysis of the electromigration effects but are handled by the dissipative structure of the Nernst-Planck equations and the elliptic regularity obeyed by the electric potential. 

Having the Feller continuity in hand, we dive into the question of whether or not invariant ergodic measures exist (Theorem \ref{kbprocedure}).
We consider the stochastic dynamics of $N$ ionic concentrations with Dirichlet boundary conditions and seek topologies where one has cancellation identities for the high regular nonlinearities of the problem. Due to the incompressibility of the fluid, the cancellation law
\be 
\int_{\mathcal{O}} (\rho \na \Phi) \cdot u dx + \int_{\mathcal{O}} (u \cdot \na \rho) \Phi dx = 0
\ee holds and motivates the coupling of the deterministic $H^{-1}$ evolution of the charge density $\rho$ with the stochastic $L^2$ evolution of the velocity $u$. Due to the nonnegativity of the ionic concentrations, the electromigration effects are fully dissipated when the species have equal diffusivities, yielding quadratic moment bounds, linear in time, for the $H^1$ norm of $u$ and $L^2$ norm of $\rho$.  {We then address the $L^2$ evolution of each ionic concentration $c_i$ and derive inequalities that bound the logarithmic Sobolev moments of the energies $\|\na c_i\|_{L^2}$ by those of $\|c_i - \gamma_i\|_{L^2}$, reducing consequently the regularity problem from $H^1$ to $L^2$.}  
In order to obtain good control of these $L^2$ energy norms of the ionic concentrations, further implicit cancellations are required and many challenges come into play.
{This issue can be effectively tackled in a situation  when each $i$th ionic species with valence $z_i$ is accompanied by another $j$th species with valence $z_j = -z_i$.} 
In this scenario, we couple the deterministic spatial $L^2$ evolutions of the difference $\rho_{ij} = c_i - c_j$ and sum $\sigma_{ij} = c_i + c_j$ and observe that the migration process determined by 
\be 
\int_{\mathcal{O}} \na \cdot (\sigma_{ij} \na \Phi) \cdot \rho_{ij} dx + \int_{\mathcal{O}}
 \na \cdot (\rho_{ij} \na \Phi) \cdot \sigma_{ij} dx
 \ee can be greatly simplified after expansion and integration by parts. We thus obtain exponentially decaying-in-time moment bounds (with a rate depending only on the diffusivities) for both $\rho_{ij}$ and $\sigma_{ij}$ in $L^2$ that are controlled by the exponential moment of the charge density $\rho$. We show that this latter expectation grows exponentially in time with a growth rate depending only on the deterministic and stochastic forcing $f$ and $g$ and the boundary data $\gamma_1, \dots, \gamma_N$. If these aforementioned parameters do not exceed the diffusivities of the species, the time growth gets beaten by the time decay, giving rise to appropriate moment bounds for $\|c_i - \gamma_i\|_{L^2}$. Consequently, the existence of invariant measures in this specific case is guaranteed by the Krylov Bogoliubov averaging procedure \cite{da2014stochastic}). We point out that this scenario is completely new and has not been treated previously in the literature neither from a deterministic nor from a stochastic point of view.
{Alternatively, when all species exhibit equal absolute valences $|z_1| = \dots = |z_N|$, it becomes possible to establish uniform quadratic moment bounds for $\|\nabla c_i\|_{L^2}$. In this scenario, the coupling of the $L^2$ evolutions of the density $\rho$ and the sum of the concentrations $\tilde{\rho} = c_1 + \dots + c_N$ results in cancellations in the nonlinear electromigration by which the sum
 \be
\int_{\mathcal{O}} \na \cdot (\tilde{\rho} \na \Phi) \rho dx + \int_{\mathcal{O}} \na \cdot (\rho \na \Phi) \tilde{\rho} dx
 \ee increases the dissipation of energy and yields an exponential decay in time of $\|\rho\|_{L^2}$. Consequently, this decay result in the desired bounds for $\|\nabla c_i\|_{L^2}$. This outcome allows us to deduce the existence of invariant measures for the Markov transition kernels without imposing any size conditions on the forces or the boundary data.} 
 
  We then investigate higher regularity properties of the invariant measure and show that it is supported on the Sobolev $H^2$ space (Theorem \ref{boundeddomainh2regularity}) by establishing the logarithmic moment bound 
 \be 
\beg{aligned}
\E \int_{0}^{T} \log \left(1 + \|u\|_{H^2}^2 + \sum\limits_{i=1}^{N} \|c_i - \gamma_i\|_{H^2}^2 \right) dt 
\le C_0\left(\|u_0\|_{H^1}, \|c_i - \gamma_i\|_{H^1}, \|g\|_{H^1} \right) 
+ C_1(\|f\|_{L^2}, \|g\|_{H^1}) T
\end{aligned}
 \ee that is at most linear in time. Such estimates are obtained via integration by parts, which turns out to be an obstacle to upgrading the $H^2$ regularity due to the boundary effects resulting from the ionic concentrations. 

Last but not least, we address the unique ergodicity in the stochastic S-NPNS system on $\mathcal{O}$ under Setting \ref{setting C} (Theorem \ref{uniquemeasure}). The answer to this question relies on the effects of the stochastic perturbation on the dynamics of the model. 
We make use of the asymptotic coupling techniques that have been widely studied in the literature and adapted to many nonlinear partial differential equations (see \cite{bricmont2002exponential,glatt2017unique,hairer2002exponential,hairer2011asymptotic,hairer2011theory,kuksin2001coupling,kuksin2002coupling,mattingly2003recent,weinan2001gibbsian} and references therein). Namely, we construct a copy of the S-NPNS system with a feedback control function, chosen in such a way that the long-time dynamics are fully determined by the low frequencies of the solutions. This approach usually requires the number of modes forcing the noise to be sufficiently large. As the Nernst-Planck equations are deterministic, we additionally need the diffusivities of the species to be large enough. A few difficulties arise from the nonlinear structure of the model and are dealt with using the uniform boundedness of the concentrations spatially in $H^1$ and timely in $L^2$ in this specific setting \ref{setting C}.

\smallskip
\noindent{\bf{Ergodicity on the Torus.}} Our next major result is the existence of smooth invariant measures for the transition kernels associated with the periodic S-NPNS system on the two-dimensional torus $\TT^2 = [0, 2\pi^2]$ with periodic boundary conditions in the most general setting \ref{setting A} where the ionic species have different diffusivities and valences (Theorem \ref{maintheoremonthetorus}). 
By simultaneously studying the stochastic evolution of the velocity $L^2$ norm, the potential $H^1$ norm, and the entropy
\be 
\mathcal{E} = \sum\limits_{i=1}^{N} \int_{\TT^2} \left(c_i \log \left(\fr{c_i}{\bar{c}_i} \right) - c_i + \bar{c}_i\right) dx,
\ee we derive quadratic moment bounds for the $L^2$ norms of the charge density and velocity gradient that are linear in time and exponential moment bounds for small constant multiples $\mu$ of $\|\rho\|_{L^2}^2$ with a time exponential growth depending on that parameter $\mu$. We then seek concentrations $L^2$  moment bounds whose long-time dynamics are controlled partially by the dissipation and partially by the electromigration effects arising from the potential gradient $\na \Phi$. To this end, we investigate the elliptic regularity obeyed by $\Phi$ and establish a new elliptic-interpolation inequality
\be 
\|\na \Phi\|_{L^4}^4 \le C\|\rho\|_{L^1}^2 \|\rho\|_{L^2}^2
\ee based on Fourier series techniques, a duality argument, the Hausdorff-Young inequality, the Marcinkiewicz interpolation theorem, and $L^p$ interpolation estimates (Proposition \ref{ellipticprop}). This good control of $\na \Phi$ allows us to obtain moment bounds of the form
\be 
\E \left(\sum\limits_{i=1}^{N} \|c_i - \bar{c}_i\|_{L^2}^2  \right)^p \le \left(\sum\limits_{i=1}^{N} \|c_i(0) - \bar{c}_i\|_{L^2}^2  \right)^p e^{-p \min \left\{D_1, \dots, D_N \right\} t} \E e^{C_2 \mu (\bar{c}_i) \left( \|\rho\|_{L^2}^2 + C_3\right)},
\ee where $\mu (\bar{c}_i)$ depends only on the averages $\bar{c}_i$. Due to the conservation of the spatial means of the concentrations, we can choose $\mu$ to be sufficiently small and obtain uniform-in-time moment bounds for the $L^2$ norms of concentrations and consequently of their gradients in $L_t^2 L_x^2$. In contrast with the case of a bounded domain with a smooth boundary where the ergodicity holds under some restrictions on the values of the diffusivities and valences or the size of the body forces, stochastic noise, and boundary data, we obtain ergodic invariant measures for the periodic S-NPNS system for large forcing and arbitrary parameters provided that the initial spatial averages of the ionic species are small. Moreover, any such invariant measure is smooth, a fact that follows from linear logarithmic moment bounds
\be 
\beg{aligned}
&\E \int_{0}^{T} \log \left(1 + \|u\|_{H^k}^2 + \sum\limits_{k=1}^{N} \|c_i - \bar{c}_i\|_{H^k}^2 \right) dt
\\&\quad\quad\quad\quad\le C_4 (\|u_0\|_{H^k}, \sum\limits_{i=1}^{N} \|c_i{(0)} - \bar{c}_i\|_{H^k}, \|g\|_{H^1}) + C_5(\|g\|_{H^k}, \|f\|_{H^k}) T
\end{aligned}
\ee that holds for any positive integer $k$. These estimates are based on fractional product estimates and expectations bounds on the product stochastic processes $\|c_i - \bar{c}_i\|_{L^2}^2\|\na c_i\|_{L^2}^2$ and $\|u\|_{L^2}^2 \|\na u\|_{L^2}^2$ in $L^2(0,T)$.  Due to periodicity, integration by parts applies, constraining the challenges to the nonlinear aspects.

\smallskip
\noindent{\bf{Exponential Ergodicity.}} Lastly, we consider two ionic species with equal diffusivities and valences $1$ and $-1$ (Setting \ref{setting D}) and study the exponential ergodicity of the resulting model (Theorem \ref{maintheoremtwospecies}). The two-species model has a special structure by which the nonlinear sum
\be 
\int_{\TT^2} (c_1 - c_2) \na  \cdot ((c_1 + c_2) \na \Phi) dx + \int_{\TT^2} (c_1 + c_2) \na \cdot ((c_1 - c_2) \na \Phi)dx  
\ee reduces to $-\|(c_1-c_2)\sqrt{c_1 + c_2} \|_{L^2}^2$, providing a structural understanding of the spatial $L^2$ evolutions of both $c_1$ and $c_2$ and yielding  an energy estimate of the form 
\be 
U(X_t) + C_6 \int_{0}^{t} S(X_s) ds
\le U(X_0) + C_7(f,g)t + C_8 (g,u)_{L^2} dW
\ee where $X_t = (u, c_1 - \bar{c}_1,  c_2 - \bar{c}_2)$, $U(\cdot) = \|\cdot\|_{L^2}^2$, and $S(\cdot) = \|\cdot\|_{H^1}^2$.
Moreover, we construct a stochastic process $Y_t = (U, C_1 - \bar{C}_1, C_2 - \bar{C}_2)$ such that the dissipativity bound 
\be 
q(X_t, Y_t) \le q(X_0, Y_0) e^{-C_9t + \int_{0}^{t} S(X_s) ds}
\ee holds. This construction is not trivial and does not hold for the N-species S-NPNS model due to the absence of crucial cancellation laws. This aforementioned pairing of the stochastic processes $X_t$ and $Y_t$ satisfies the generalized coupling framework established in \cite{butkovsky2020generalized}, from which we infer the exponential convergence of the transition kernels to the invariant measure in a suitable probability metric and obtain the exponential ergodicity of two-species S-NPNS model. 

\subsection{Organization of the Paper}
This paper is organized as follows. In  Section \ref{pre}, we introduce the basic functional spaces, operators, and notations that are frequently used throughout the whole manuscript. In Section \ref{sec:existence}, we construct probabilistic strong unique solutions to the S-NPNS model \eqref{1}--\eqref{4} for $L^2$ initial velocity and $L^2$ nonnegative initial ionic concentrations. Section \ref{lon} is dedicated to the exponential stability of the concentrations in $L^p$ spaces under a smallness size condition imposed on the $L^2$ norm of the initial concentrations. In Section \ref{fel}, we define the Markov semigroup associated with the stochastic S-NPNS system and obtain its Feller continuity. In Section \ref{erg}, we construct smooth unique ergodic invariant measures on bounded smooth domains for constant Dirichlet boundary data obeyed by the concentrations and under different conditions imposed on the size of the parameters, the forcing, or the boundary values. Section \ref{per} deals with the unique ergodicity of the periodic S-NPNS system based on novel elliptic-interpolation estimates derived in Appendix \ref{Marcin}. Finally, we consider the two-species model in Section \ref{two} and prove its exponential ergodicity based on the generalized coupling approach summarized in Appendix \ref{ExpErgFramework}.

\section{Preliminaries} \label{pre}

Let $\mathcal{O} \subset \R^2$ be a bounded domain with a smooth boundary. Throughout the paper, $C$ denotes a positive universal constant, and may change from step to step. For a letter $\mathcal{L}$, $\mathcal{L}(a,b,c,...)$ denotes a positive constant depending on $a$, $b$, $c, \dots$.

\smallskip

\noindent{\bf Functional Settings.}
For $1 \le p \le \infty$, we denote by $L^p(\mathcal{O})$ the Lebesgue spaces of measurable functions $f$ from $\mathcal{O}$ to $\R$ (or $\RR^2)$ such that 
\be 
\|f\|_{L^p} = \left(\int_{\mathcal{O}} \|f\|^p dx\right)^{1/p} <\infty, \;\text{if } p \in [1, \infty) \quad \text{and} \quad 
\|f\|_{L^{\infty}} = {{\esssup}}_{\mathcal{O}}  |f| < \infty, \; \text{if } p = \infty.
\ee
The $L^2$ inner product is denoted by $(\cdot,\cdot)_{L^2}$.

For $k \in \NN$, we denote by $H^k(\mathcal{O})$ the classical Sobolev space of measurable functions $f$ from $\mathcal{O}$ to $\R$ (or $\RR^2)$ with weak derivatives of order $k$ such that  
$ 
\|f\|_{H^k}^2 = \sum\limits_{|\alpha| \le k} \|D^{\alpha}f\|_{L^2}^2 < \infty.
$  The space $H_{0}^{1}(\mathcal{O})$ refers to the subspace of $H^1(\mathcal{O})$ consisting of functions with homogeneous Dirichlet boundary conditions.

For a Banach space $(X, \|\cdot\|_{X})$ and $p\in [1,\infty]$, we consider the Lebesgue spaces $ L^p(0,T; X)$ of functions $f$  from $X$ to $\R$ (or $\RR^2)$ satisfying 
$
\int_{0}^{T} \|f\|_{X}^p dt  <\infty
$ with the usual convention when $p = \infty$. The corresponding norm will be denoted by $\|\cdot\|_{L^p(0, T; X)}$ or abbreviated as $\|\cdot\|_{L_t^p X}$. 

\smallskip

\noindent{\bf The Stokes Operator.} Let $H$ be
\be 
H = \left\{v = (v_1, v_2)\in L^2(\mathcal{O}): \na \cdot v = 0, \, v \cdot n|_{\pa \mathcal{O}} = 0 \right\}
\ee 
where $n$ is the outward unit normal to $\pa \mathcal{O}$, and denote by $\mathcal{P}: L^2(\mathcal{O}) \rightarrow H$ the Leray Hodge projection onto $H$. We define the Stokes operator, denoted by $A$, on $H \cap H_0^1(\mathcal{O}) \cap H^2(\mathcal{O})$ as $A := -\mathcal{P} \Delta$. Denote the eigenvalues of $A$ by $\mu_j$ with $j\in \mathbb N$, and the corresponding eigenfunctions by $\phi_j$. 
{
By the standard spectral theorem (see, for example \cite{constantin1988navier}), one has $0 < \mu_1 \le ... \le \mu_j \le ... \rightarrow \infty$.
}
The fractional powers of the Stokes operator, denoted by $A^s$, are defined by 
\be \la{stoeig}
A^s v= \sum_{j=1}^{\infty} \mu_j^s (v, \phi_j)_{L^2} \phi_j,\quad
\text{with domain}\quad 
\mathcal{D}(A^s) = \left\{v \in H : \|A^{s} v\|_{L^2}^2 := \sum\limits_{j \in \N} \mu_j^{2s}(v, \phi_j)_{L^2}^2 < \infty \right\}.
\ee 
We define the bilinear form $B$ by
$
B(u,v) = \mathcal{P}(u \cdot \na v)
$ for any $u, v \in H \cap H_0^1$.  

\smallskip

\noindent{\bf Periodic Fractional Powers of the Laplacian.}  Let $\TT^2 = [0,2\pi]^2$ be the two dimensional torus.
For $s\in\R$, the periodic fractional Laplacian $\l^s$ applied to a mean zero function $f \in L^2(\TT^2)$ is a Fourier multiplier with symbol $|k|^s$. That is, for $f$ with Fourier series representation
\be
f(x) = \sum\limits_{k \in \mathbb{Z}^2 \setminus \left\{0\right\}} f_k e^{ik \cdot x},
\text{ and obeying }
\sum\limits_{k \in \mathbb{Z}^2 \setminus \left\{0\right\}} |k|^{2s} |f_k|^2 < \infty,
\ee
we have
\be
\l^s f(x) = \sum\limits_{k \in \mathbb{Z}^2 \setminus \left\{0\right\}} |k|^s f_k e^{ik \cdot x}.
\ee 

\smallskip

\noindent{\bf Stochastic Settings.} We denote by $(\Omega, \mathcal{F}, \mbF, \PP)$ a filtered probability space and $\mathbb{F} =\left\{ \mathcal{F}_s\right\}_{s \ge 0}$ be a filtration on $(\Omega, \mathcal{F}, \PP)$, supporting independent real-valued, standard Brownian motions, denoted by $\{W_k\}_{k \ge 0}$. The stochastic noise term $gdW$ appearing in the S-NPNS system is 
interpreted as
\be 
g dW = \sum\limits_{k = 1}^{\infty} g_k(x) dW_k(t),
\ee  
where the components $g_k$ are assumed to be in $\mathcal{D}(A^{\fr{1}{2}})$. For an integer $k \ge 0$, we denote 
\be 
\|g\|_{H^k}^2 = \sum\limits_{l=1}^{\infty} \| g_l \|_{H^k}^2,
\text{ for any } g \in \mathcal D(A^{\frac k2}). 
\ee

\section{Global Well-posedness of the S-NPNS System}\label{sec:existence}

\subsection{Local Well-posedness}

We consider the S-NPNS system
\noeqref{s212}
\begin{subequations}\label{sys:SNPNS}
    \begin{align}
       &d u + (u \cdot \na u - \Delta u + \na p) dt = (-\rho \na \Phi + f) dt + g dW, \la{s21}
       \\
       &\pa_t c_i + u \cdot \na c_i - D_i \Delta c_i = z_iD_i \na \cdot (c_i \na \Phi), \text{ for } i = 1, \dots, N, \la{s210}
       \\
       &-\Delta \Phi = \rho, \la{s211}
       \\
       &\na \cdot u = 0 \la{s212}
    \end{align}
\end{subequations}
in a bounded domain $\mathcal{O}$ with a smooth boundary $\pa \mathcal{O}$, equipped with the boundary conditions \eqref{3} and initial conditions \eqref{4}. The body forces $f\in H$ and $g\in \mathcal D(A^{\frac12})$ are assumed to be time-independent and divergence-free. By projecting equation \eqref{s21} onto the space of divergence-free vectors, we observe that \eqref{s21} is equivalent to  
\be \la{s22}
d u + (B(u,u)  + Au) dt = (-\mathcal{P}(\rho \na \Phi) + f) dt + g dW.
\ee

\beg{Thm} \la{loc} 

Let  $u_0 \in H$ and $c_i(0) \in L^2$ for all $i \in \left\{1, \dots, N\right\}$. There exists a time $T_0 > 0$ depending only on the size of the initial data in $L^2$ and the parameters of the problem, such that the initial boundary value problem determined by the system \eqref{sys:SNPNS} and the boundary conditions \eqref{3} has a unique weak solution $(u, c_1, \dots, c_N)$ on $[0, T_0]$ obeying  
\be 
u \in L^{\infty} (0, T_0; H) \cap L^2(0,T_0; \mathcal{D}(A^{\frac{1}{2}})),
\ee  
and
\be 
c_i \in L^{\infty}(0,T_0; L^2(\mathcal{O})) \cap L^2(0, T_0; H^1(\mathcal{O})),
\ee 
{almost surely} for any $i \in \left\{1, \dots, N \right\}$.
\end{Thm}

\begin{proof}
We divide the proof into four steps.
\smallskip

\noindent{\bf{Step 1. The iterative system.}}
 Let $u^{(0)} = c_1^{(0)} = \dots = c_N^{(0)} = 0$. We consider the iterative system 
 \noeqref{sys:finite-SNPNS-1}
  \noeqref{sys:finite-SNPNS-2}
   \noeqref{sys:finite-SNPNS-3}
    \noeqref{sys:finite-SNPNS-4}
 \begin{subequations}\label{sys:finite-SNPNS}
  \begin{align}
    &d u^{(m)} + (B(u^{(m)},u^{(m)})  + Au^{(m)}) dt = (-\mathcal{P}(\rho^{(m-1)} \na \Phi^{(m-1)}) + f) dt + g dW, \label{sys:finite-SNPNS-1}
    \\
    &\pa_t c_i^{(m)} + u^{(m)} \cdot \na c_i^{(m)} - D_i \Delta c_i^{(m)} = z_iD_i \na \cdot (c_i^{(m)} \na \Phi^{(m-1)}), \quad i = 1, \dots, N, \label{sys:finite-SNPNS-2}
    \\
    &-\Delta \Phi^{(m)} = \rho^{(m)} = \sum\limits_{i=1}^{N} z_ic_i^{(m)}, \label{sys:finite-SNPNS-3}
    \\
    &\na \cdot u^{(m)} = 0, \label{sys:finite-SNPNS-4}
 \end{align}   
 \end{subequations}
for each integer $m \ge 1$, with initial conditions
\be 
u^{(m)} (x,0) = u(x,0), \; c_i^{(m)} (x,0) = c_i(x,0) \label{IC-finite}
\ee 
for $i = 1, ..., N$, and boundary conditions 
\begin{equation}\label{BC-finite}
\begin{split}
    &u^{(m)}|_{\pa \mathcal{O}}  = 0, \quad \Phi^{(m)}|_{\pa \mathcal{O}} = \gamma,
    \\
    &c_i^{(m)}|_{\pa \mathcal{O}} = \gamma_i \quad \text{for} \quad i= 1, \dots, M,
    \\
    &\left(\nabla c_i^{(m)} + z_i c_i^{(m)} \na \Phi^{(m-1)}\right)|_{\pa \mathcal{O}} \cdot n = 0 \quad \text{for} \quad  i = M+1, \dots, N.
\end{split}   
\end{equation}
Denote by
\be \la{cov}
G(x,t,\omega) = \int_{0}^{t} e^{(t-s)A} g(x) dW \quad \text{and} \quad v^{(m)}(x,t,\omega) = (u^{(m)} - G)(x,t,\omega).
\ee  
Then we can rewrite system \eqref{sys:finite-SNPNS} as the following deterministic system 
\begin{subequations}\label{sys:finite-deter}
    \begin{align}
&\pa_t v^{(m)} + B(v^{(m)}+G, v^{(m)}+G) - { Av^{(m)}} = -\mathcal{P}(\rho^{(m-1)} \na \Phi^{(m-1)})+ f, \la{th11}
\\
&\pa_t c_i^{(m)}+ (v^{(m)} + G) \cdot \na c_i^{(m)} - D_i \Delta c_i^{(m)} =  z_i D_i \na \cdot (c_i^{(m)} \na \Phi^{(m-1)}), \quad i = 1, \dots, N, \la{th12}
\\
&-\Delta \Phi^{(m)} = \rho^{(m)} = \sum\limits_{i=1}^{N} z_i c_i^{(m)}, \la{th13}
\\
&\na \cdot v^{(m)} = 0, \la{th14}
    \end{align}
\end{subequations}
with the same initial and boundary conditions \eqref{IC-finite} and \eqref{BC-finite}, except the conditions for $u^{(m)}$ is replaced by
\[
v^{(m)} (x,0) = u(x,0), \quad v^{(m)}|_{\pa \mathcal{O}}  = 0.
\]
The homogeneous Dirichlet boundary conditions obeyed by $v^{(m)}$ arise from the vanishing of both $u^{(m)}$ and $G$ on $\pa \mathcal{O}$.

{
Fix $m \ge 1$ and suppose $v^{(m-1)}, c_1^{(m-1)}, \dots, c_N^{(m-1)}$ are given. As $v^{(m)}$ is determined in terms of $c_1^{(m-1)}$, $\dots$, $c_N^{(m-1)}$, it is evident that \eqref{th12} is linear in $c_i^{(m)}$ for any $i \in \left\{1, \dots, N\right\}.$}
By making the change of variable
\be \la{th20}
\tilde{c_i}^{(m)} = c_i^{(m)} e^{z_i \Phi^{(m-1)}},
\ee 
for $i= M+1, \dots, N$,  we observe that $\tilde{c_i}^{(m)}$ also obeys a linear equation 
equivalent to \eqref{th12}, with homogeneous Neumann boundary conditions. 
{
Consequently, we can view the equations obeyed by the concentration approximants as linear equations equipped with either Dirichlet or Neumann boundary conditions, which allows us to justify the solvability of the approximating model in hand.} Indeed, for $m=1$, the Navier-Stokes equation \eqref{th11} with the divergence-free condition \eqref{th14} has global solutions in $L^{\infty}(0,T; H)$ and $L^2(0,T; \mathcal{D}(A^{\fr{1}{2}}))$ for any $T>0$. In view of the observation associated with \eqref{th20}, the linear equation \eqref{th12} obeyed by the concentration approximant $c_i^{(1)}$ also has global solutions in  $L^{\infty}(0,T; L^2(\mathcal{O}))$ and $L^2(0,T; H^1(\mathcal{O}))$ for any $T>0$.  Suppose the $(m-1)$-th solution exists and obeys the same aforementioned regularity. Given the $(m-1)$-th regular electric forces, the $m$-th Navier-Stokes system has global smooth solutions, and so do the linear parabolic $m$-th Nernst-Planck equations. This iterative argument allows us to conclude that for each fixed integer $m \ge 1$, the system \eqref{sys:finite-deter} has global regular solutions. Next, we derive \textit{a priori} bounds which is uniform in $m$. 
\smallskip

\noindent \textbf{Step 2. \textit{A priori} $L^2$ estimate of $c_i^{(m)}$.}
First, we consider an fixed index $i \in \left\{M+1, ..., N \right\}$, and take the $L^2$ inner product of the equation \eqref{th12} obeyed by the approximant $c_i^{(m)}$ with $c_i^{(m)}$. The nonlinear term $((v^{(m)} + G) \cdot \na c_i^{(m)}, c_i^{(m)})_{L^2}$ vanishes due to the divergence-free condition and Dirichlet boundary condition obeyed by both $v^{(m)}$ and $G$. 
This gives rise to the energy equality
\be \la{th21}
\beg{aligned} 
&\fr{1}{2} \fr{d}{dt} \|c_i^{(m)}\|_{L^2}^2 + D_i \|\na c_i^{(m)}\|_{L^2}^2 - \int_{\pa \mathcal{O}} D_i c_i^{(m)} \na c_i^{(m)} \cdot n d\sigma(x)
\\&\quad\quad= - \int_{\mathcal{O}} D_i z_i c_i^{(m)} \na \Phi^{(m-1)} \cdot \na c_i^{(m)} dx
+ \int_{\pa \mathcal{O}} D_i z_i c_i^{(m)} \na \Phi^{(m-1)} c_i^{(m)} \cdot n d\sigma(x)
\end{aligned}
\ee 
after an integration by parts. Here $d\sigma$ denotes the surface measure. Making use of the boundary conditions \eqref{BC-finite} brings the following cancellation
\be 
\int_{\pa \mathcal{O}} D_i c_i^{(m)} \na c_i^{(m)} \cdot n d\sigma(x) + \int_{\pa \mathcal{O}} D_i z_i c_i^{(m)} \na \Phi^{(m-1)} c_i^{(m)} \cdot n d\sigma(x) = 0.
\ee
We estimate the nonlinear forcing term 
\be \la{th22}
\beg{aligned}
&\left|\int_{\mathcal{O}} D_i z_i c_i^{(m)} \na \Phi^{(m-1)} \cdot \na c_i^{(m)} dx \right|
\le C\|c_i^{(m)}\|_{L^4} \|\na \Phi^{(m-1)}\|_{L^4} \|\na c_i^{(m)}\|_{L^2}
\\&\quad\quad\le C_{\gamma}\left(\|c_i^{(m)}\|_{L^2} + \|c_i^{(m)}\|_{L^2}^{\fr{1}{2}} \|\na c_i^{(m)}\|_{L^2}^{\fr{1}{2}} \right) \left(\|\rho^{(m-1)}\|_{L^2} + 1 \right)\|\na c_i^{(m)}\|_{L^2}
\\&\quad\quad\le \fr{D_i}{2} \|\na c_i^{(m)}\|_{L^2}^2 
+ C_{\gamma}\|c_i^{(m)}\|_{L^2}^2 \left(\|\rho^{(m-1)}\|_{L^2}^4 + 1 \right)
\\&\quad\quad\le \fr{D_i}{2} \|\na c_i^{(m)}\|_{L^2}^2  + C\left[\left(\sum\limits_{j=1}^{N} \|c_j^{(m-1)}\|_{L^2}^2\right)^2 + 1 \right]\|c_i^{(m)}\|_{L^2}^2 
\end{aligned}
\ee 
by making use of the Ladyzhenskaya inequality, the elliptic regularity obeyed by the solution $\Phi^{(m-1)}$ to the Poisson equation \eqref{th13}, and Young's inequality for products. We sum the energy equalities \eqref{th21} over all indices $i \in \left\{M+1, ..., N\right\}$ and obtain the differential inequality 
\be \la{th23}
\beg{aligned}
&\fr{d}{dt} \left(\sum\limits_{i=M+1}^{N} \|c_i^{(m)}\|_{L^2}^2 \right) + \sum\limits_{i=M+1}^{N} D_i \|\na c_i^{(m)}\|_{L^2}^2 
\le C\left[\left(\sum\limits_{j=1}^{N} \|c_j^{(m-1)}\|_{L^2}^2\right)^2 + 1 \right] \sum\limits_{i=M+1}^{N} \|c_i^{(m)}\|_{L^2}^2
\end{aligned}
\ee 
by appealing to \eqref{th22}. As for the ionic concentrations with inhomogeneous constant Dirichlet boundary conditions, we fix an integer $i \in \left\{1, \dots, M\right\}$ and take the scalar product in $L^2$ of the equation \eqref{th12} obeyed by $c_i^{(m)}$ with $c_i^{(m)} - \gamma_i$. We obtain the energy evolution
\be \la{th24}
\beg{aligned} 
\fr{1}{2} \fr{d}{dt} \|c_i^{(m)} - \gamma_i\|_{L^2}^2 + D_i \|\na c_i^{(m)}\|_{L^2}^2 
= - \int_{\mathcal{O}} D_i z_i c_i^{(m)} \na \Phi^{(m-1)} \cdot \na c_i^{(m)} dx,
\end{aligned}
\ee 
which boils down to 
\be \la{th25}
 \fr{d}{dt} \|c_i^{(m)} - \gamma_i\|_{L^2}^2 + D_i \|\na c_i^{(m)}\|_{L^2}^2 
\le C\left[\left(\sum\limits_{j=1}^{N} \|c_j^{(m-1)}\|_{L^2}^2\right)^2 + 1 \right]\|c_i^{(m)}\|_{L^2}^2 
\ee 
due to the estimate \eqref{th22} and the independency of the constants $\gamma_i$ on both the spatial and time variables. Adding these latter $M$ inequalities for $i \in \left\{1, \dots, M \right\}$, we end up with 
\be 
\fr{d}{dt} \left(\sum\limits_{i=1}^{M} \|c_i^{(m)} - \gamma_i\|_{L^2}^2\right) + \sum\limits_{i=1}^{M} D_i \|\na c_i^{(m)}\|_{L^2}^2 
\le C\left[\left(\sum\limits_{j=1}^{N} \|c_j^{(m-1)}\|_{L^2}^2\right)^2 + 1 \right]\sum\limits_{i=1}^{M} \|c_i^{(m)}\|_{L^2}^2.
\ee 
Therefore, we deduce that
\be \la{th26}
\beg{aligned}
&\fr{d}{dt} \left(\sum\limits_{i=1}^{M} \|c_i^{(m)} - \gamma_i\|_{L^2}^2 + \sum\limits_{i=M+1}^{N} \|c_i^{(m)}\|_{L^2}^2 \right) 
+ \sum\limits_{i=1}^{N} D_i \|\na c_i^{(m)}\|_{L^2}^2 
\\&\le C\left[\left(\sum\limits_{i=1}^{N} \|c_i^{(m-1)}\|_{L^2}^2\right)^2 + 1 \right] \sum\limits_{i=1}^{N} \|c_i^{(m)}\|_{L^2}^2
\\&\le C \left[\left(\sum\limits_{i=1}^{M} \|c_i^{(m-1)} - \gamma_i\|_{L^2}^2 + \sum\limits_{i=M+1}^{N} \|c_i^{(m-1)}\|_{L^2}^2 \right)^2 + 1\right] \left(\sum\limits_{i=1}^{M} \|c_i^{(m)} - \gamma_i\|_{L^2}^2 + \sum\limits_{i=M+1}^{N} \|c_i^{(m)}\|_{L^2}^2 \right) 
\\&\quad\quad\quad\quad+ C\left[\left(\sum\limits_{i=1}^{M} \|c_i^{(m-1)} - \gamma_i\|_{L^2}^2 + \sum\limits_{i=M+1}^{N} \|c_i^{(m-1)}\|_{L^2}^2 \right)^2 + 1 \right],
\end{aligned}
\ee 
where $C$ is a positive constant depending on $\gamma$ and $\gamma_i$ for $i \in \left\{1, \dots, M\right\}$, the parameters of the problem, and some universal constants. 
We define the time-dependent sequence $a_m(t)$ by 
\be  
a_m(t) = \sum\limits_{i=1}^{M} \|c_i^{(m)} - \gamma_i\|_{L^2}^2 + \sum\limits_{i=M+1}^{N} \|c_i^{(m)}\|_{L^2}^2, 
\ee 
and note that $a_m$ satisfies the ODE
\be 
\fr{d}{dt} a_m(t) \le C(a_{m-1}^2 +1)a_m + Ca_{m-1}^2 + C,
\ee 
from which we obtain the bound
\be 
a_m(t) \le \left(a_m(0) + C\int_{0}^{t} a_{m-1}^2 dt + Ct \right) \exp \left\{C\int_{0}^{t} (a_{m-1}^2 +1) ds \right\}.
\ee 
Taking the supremum over the time interval $[0,T]$ produces
\be 
\beg{aligned}
A_m(t) 
&\le \left(a_m(0) + CA_{m-1}^2T + CT\right) \exp \left\{C(A_{m-1}^2 +1)T \right\} 
\le \exp \left\{a_m(0) + CA_{m-1}^2T + CT \right\},
\end{aligned}
\ee 
where 
$
A_m(T) = \sup\limits_{0 \le t \le T} a_m(t).
$
Since $a_m(0)$ does not depend on $m$ and obeys $a_m(0) = a_1(0) = A_1(0)$, we infer that
\be 
A_m(T) \le \exp \left\{A_1(0) + C_0A_{m-1}^2T + C_0T \right\}
\ee 
for any $T > 0$. Here $C_0> 0$ depends on the boundary values, parameters of the problem, and universal constants. An induction argument gives the uniform-in-$m$ bound
$
A_m(T) \le e^{3A_1(0)},
$ 
provided that 
\be 
0 < T \le T_0:= \min\left\{\fr{A_1(0)}{C_0}, \fr{A_1(0)}{C_0 A_0^2}, \fr{A_1(0)}{C_0e^{6A_1(0)}} \right\}.
\ee 
Indeed, 
\be 
A_1(T) \le \exp \left\{A_1(0) + C_0A_0^2T + C_0T \right\}
\le \exp \left\{3A_1(0)\right\}
\ee 
for any $T \in [0, T_0]$. If $A_m(T) \le \exp \left\{3A_1(0)\right\}$, then 
\be 
A_{m+1}(T) \le \exp \left\{A_1(0) + C_0e^{6A_1(0)}T + C_0T \right\}
\le \exp \left\{3A_1(0)\right\}
\ee for any time $T \in [0,T_0].$ 
Integrating \eqref{th26} in time from $0$ to $T_0$, we have the local-in-time integrability of the spatial $H^1$ norm of the concentration approximants
\be 
\int_{0}^{T_0} \sum\limits_{i=1}^{N} \|\na c_i^{(m)}(t)\|_{L^2}^2 dt \le \Gamma_0,
\ee
where $\Gamma_0$ is a positive constant depending on the initial data and the parameters of the model.

Consequently, the family of approximants $\left\{c_i^{(m)} \right\}_{m=1}^{\infty}$ is uniformly bounded in the Lebesgue spaces $L^{\infty}(0, T_0; L^2(\mathcal{O}))$ and $L^2(0,T_0; H^1(\mathcal{O}))$ a.s. for all $i \in \left\{1, \dots, N \right\}$.

\medskip

\noindent \textbf{Step 3. \textit{A priori} $L^2$ estimate of $v^{(m)}$.}
We take the scalar product in $L^2$ of the velocity equation in \eqref{th11} obeyed by the approximants $v^{(m)}$ with $v^{(m)}$ and obtain the energy equality
\be \la{s250}
\beg{aligned}
&\fr{1}{2} \fr{d}{dt} \|v^{(m)}\|_{L^2}^2
+ \|A^{\fr{1}{2}} v^{(m)}\|_{L^2}^2
\\&\quad\quad=  - (B(v^{(m)} + G, v^{(m)} + G), v^{(m)})_{L^2}
+ (f, v^{(m)})_{L^2}
- (\mathcal{P}(\rho^{(m-1)}\na \Phi^{(m-1)}), v^{(m)})_{L^2}.
\end{aligned}
\ee 
In view of the $L^2$ cancellation law
\be 
(B(v^{(m)} + G,v^{(m)} + G), v^{(m)} + G)_{L^2} = 0,
\ee 
the self-adjoitness of the Leray projector $\mathcal{P}$, and the divergence-free condition obeyed by both $v^{(m)}$ and $G$, we can rewrite the nonlinear term in $v^{(m)}$ as 
\be 
\beg{aligned}
&- (B(v^{(m)} + G, v^{(m)} + G), v^{(m)})_{L^2}
= (B(v^{(m)} + G, v^{(m)} + G), G)_{L^2}
\\&= ((v^{(m)} + G) \cdot \na (v^{(m)} + G), \mathcal{P} G)_{L^2}
= ((v^{(m)} + G) \cdot \na (v^{(m)} + G), G)_{L^2}
\\&= - ((v^{(m)} + G) \cdot \na G, v^{(m)} + G)_{L^2},
\end{aligned}
\ee 
and estimate using the Cauchy-Schwarz inequality, the Ladyzhenskaya interpolation inequality, and Young's inequality for products as follows, 
\be 
\beg{aligned}
&| (B(v^{(m)} + G, v^{(m)} + G), v^{(m)})_{L^2}|
\le \|v^{(m)} + G\|_{L^4}^2 \|\na G\|_{L^2}
\\&\le C\|v^{(m)} + G\|_{L^2}\|\na (v^{(m)} + G)\|_{L^2} \|\na G\|_{L^2}
\\&\le \fr{1}{8} \|A^{\fr{1}{2}} v^{(m)}\|_{L^2}^{2}
+ C\|\na G\|_{L^2}^2 \left(\|v^{(m)}\|_{L^2}^2 + \|G\|_{L^2}^2  \right) + C\|G\|_{L^2}^2 \|\na G\|_{L^2}.
\end{aligned}
\ee  
As for the nonlinear term in $\rho^{(m-1)}$, we have 
\be \la{s26}
\beg{aligned}
&| (\mathcal{P}(\rho^{(m-1)} \na \Phi^{(m-1)}), v^{(m)})_{L^2}|
= | (\rho^{(m-1)} \na (\Phi^{(m-1)} - \gamma), v^{(m)})_{L^2}|
\\&\le \|\rho^{(m-1)}\|_{L^2} \|\na (\Phi^{(m-1)} - \gamma)\|_{L^4} \|v^{(m)}\|_{L^4}
\\&\le \|\rho^{(m-1)}\|_{L^2}^2 \|\na v^{(m)}\|_{L^2}
\le \fr{1}{8} \|A^{\fr{1}{2}} v^{(m)}\|_{L^2}^2 
+ \|\rho^{(m-1)}\|_{L^2}^4.
\end{aligned}
\ee 
Here, we have applied H\"older's inequality with exponents $2,4,4$, the Sobolev inequality, and took advantage of the elliptic regularity of solutions to the Poisson equation \eqref{th13}. Therefore, the  equation \eqref{s250} yields to the differential inequality
\be \la{th27}
\beg{aligned}
&\fr{d}{dt} \|v^{(m)}\|_{L^2}^2
+ \|A^{\fr{1}{2}}v^{(m)}\|_{L^2}^2
\\&\quad\quad\le C\|\na G\|_{L^2}^2 \|v^{(m)}\|_{L^2}^2 
+ C\left(\sum\limits_{i=1}^{N} \|c_i^{(m-1)}\|_{L^2}^4\right)  
+ C\|G\|_{L^2}^2 \|A^{\fr{1}{2}} G\|_{L^2}
+ C\|G\|_{L^2}^2 \|A^{\fr{1}{2}} G\|_{L^2}^2 
+ C\|f\|_{L^2}^2.
\end{aligned}
\ee 
From \eqref{th27} and the regularity of the concentrations approximants on the time interval $[0,T_0]$, we infer that the velocity approximants $v^{(m)}$ lie in the spaces $L^{\infty}(0, T_0; L^2(\mathcal{O}))$ and $L^2(0, T_0; H^1(\mathcal{O}))$ a.s.. Moreover, the bounds on $v^{(m)}$ are uniform in $m$.
\medskip

\noindent\textbf{Step 4. Local existence and uniqueness of solutions.}
Finally, we apply the Aubin-Lions lemma and obtain a subsequence of the family {$(v^{(m)}, c_1^{(m)}, \dots, c_N^{(m)})$ that converges to a weak solution $(v, c_1, \dots, c_N)$ of the system \eqref{sys:finite-deter} on the time interval $[0,T_0]$. Define $u=v+G$. As $G$ is smooth, we infer that $(u, c_1, \dots, c_N)$ is a weak solution of system \eqref{sys:SNPNS} on the time interval $[0,T_0]$. 
}
As for uniqueness, the proof follows along the lines of Proposition \ref{cont} and will be omitted to avoid redundancy. 
\end{proof}


\subsection{Extension to the Global Solution}

Starting with nonnegative initial concentrations, we present the following proposition regarding the maintenance of the nonnegativity of the concentrations of each ionic species at all times.

\beg{prop} \la{prop1} 
Let $(u, c_1, \dots, c_N)$ be the unique weak solution of the problem \eqref{sys:SNPNS} with boundary conditions \eqref{3} on the time interval $[0,T]$. If the initial concentration $c_i(0)$ is nonnegative, then $c_i(x,t) \ge 0$ for a.e. $x \in \mathcal{O}$ and $t \in [0,T]$.
\end{prop}

The proof of Proposition \ref{prop1} can be found in \cite[Section 5]{constantin2019nernst} and is based on the regularity property
\be \la{prop12}
\int_{0}^{T} \|\na \Phi(t)\|_{L^{\infty}}^2 dt < \infty
\ee obeyed by the potential $\Phi$. Indeed, the integrability condition \eqref{prop12} holds for weak solutions on $[0,T]$ due to the elliptic regularity $\|\na \Phi(t)\|_{L^{\infty}} \le C\|\rho\|_{L^4}$ gained from the Poisson equation \eqref{s211}. 

The nonnegativity of the ionic concentrations allows us to obtain uniform-in-time bounds on any time interval, as demonstrated in the following proposition.

\beg{prop} \la{ext} 
 Let $T>0$ be arbitrary. Suppose $(u, c_1, \dots, c_N)$ is the unique weak solution of the problem \eqref{sys:SNPNS} with boundary conditions \eqref{3} on the time interval $[0,T]$. If $c_i(x,t) \ge 0$ for a.e. $x \in \mathcal{O}$ and $t \in [0,T]$, then there exists a positive continuous increasing  function in $T$, denoted by $K=K(T)$, which depends on the size of the initial data in $L^2$ norm, the boundary data, the noise $g$, the body forces $f$, and the parameters of the problem,
such that 
\be 
\sup\limits_{0 \le t \le T} \left[\|u(t)\|_{L^2}^2 + \sum\limits_{i=1}^{N} \|c_i(t)\|_{L^2}^2 \right] + \int_{0}^{T} \left[\|\na u(t)\|_{L^2}^2 + \sum\limits_{i=1}^{N} \|\na c_i(t)\|_{L^2}^2  \right] dt\le K.
\ee 
\end{prop}

    The proof of Proposition \ref{ext} is analogue to \cite{constantin2019nernst} and is therefore omitted.
   Based on Proposition \ref{ext}, we are able to obtain the following theorem concerning the global existence of a unique weak solution. 

\beg{Thm} \la{glos} Let $T>0$ be arbitrary, and suppose that $u_0 \in H$ and $c_i(0) \in L^2$ for all $i \in \left\{1, \dots, N\right\}$ such that $c_i(0) \ge 0$. Then the initial boundary value problem determined by the system  \eqref{sys:SNPNS} and the boundary conditions \eqref{3} has a unique weak solution $(u, c_1, \dots, c_N)$ on $[0,T]$ obeying  
\be 
u \in L^{\infty} (0, T; H) \cap L^2(0,T; \mathcal{D}(A^{\frac{1}{2}}))
\ee  and
\be 
c_i \in L^{\infty}(0,T; L^2(\mathcal{O})) \cap L^2(0, T; H^1(\mathcal{O})) 
\ee for any $i \in \left\{1, \dots, N \right\}$ almost surely .
\end{Thm}

\begin{proof}
 The existence of a local weak solution on a time interval $[0, T_0]$ is guaranteed by Theorem \ref{loc}. Since the initial concentrations are assumed to be nonnegative, it follows from Proposition \ref{prop1} that the concentrations stay nonnegative for all times $t \in [0,T_0]$. Having this property in hand, we can apply Proposition \ref{ext} to conclude that the weak solution is uniformly bounded at the time $T_0$, the fact that allows us to repeat the argument of Theorem \ref{loc} and extend the local solution from $[0,T_0]$ into $[0, T_1]$ for some time $T_1 > T_0$. Due to Propositions \ref{prop1} and \ref{ext}, we deduce that the solution behaves nicely at $T_1$. {We keep repeating the same 
 argument and obtain a sequence of times $T_n$ such that $\left\{T_n\right\}_{n=1}^{\infty}$ is increasing and the system \eqref{sys:SNPNS} has a weak solution on $[0,T_n]$ with the property that 
 \be \la{sequence}
\sup\limits_{0 \le t \le T_n} \left[\|u(t)\|_{L^2}^2 + \sum\limits_{i=1}^{N} \|c_i(t)\|_{L^2}^2 \right] + \int_{0}^{T_n} \left[\|\na u(t)\|_{L^2}^2 + \sum\limits_{i=1}^{N} \|\na c_i(t)\|_{L^2}^2  \right] dt\le K_n
 \ee where $K_n$ is also increasing in time. Suppose that $T_n$ converges to some $T'$ with $T'< T$. As the sequences involved in \eqref{sequence} are increasing and bounded, they converge and consequently, it holds that
 \be 
\sup\limits_{0 \le t < T'} \left[\|u(t)\|_{L^2}^2 + \sum\limits_{i=1}^{N} \|c_i(t)\|_{L^2}^2 \right] + \int_{0}^{T'} \left[\|\na u(t)\|_{L^2}^2 + \sum\limits_{i=1}^{N} \|\na c_i(t)\|_{L^2}^2  \right] dt\le K
 \ee where $K$ depends on $T'$. This latter uniform boundedness allows us to deduce that the solution can be uniquely continued after $T'$ until it reaches the desired arbitrary time $T$. }
\end{proof}

\section{\texorpdfstring{$L^p$}{} Regularity and Asymptotic Behavior of the Ionic Concentrations} \label{lon}

In this section, we address the long-time behavior and the $L^p$ regularity of the ionic concentrations solving the system \eqref{sys:SNPNS} in $\mathcal{O}$ with the mixed boundary conditions \eqref{3}. We first state and prove the following theorem concerning the $L^2$ long-time behavior of the global unique weak solution constructed in the previous section. 

\beg{Thm} \la{longt} 

Let $u_0 \in H$, and $c_i(0) \in L^2$ be nonnegative for each $i \in  \left\{1, \dots, N\right\}$.  Suppose that the relation \eqref{44}
holds. 
 There exists a positive constant $\varepsilon$ depending only on the diffusivities and valences such that if 
\be  \la{SC}
\sum\limits_{i=1}^{M} \left[\|c_i(0) - \gamma_i\|_{L^2}^2 + \gamma_i\right]
+ \sum\limits_{i=M+1}^{N} \left[\|c_i(0) - \bar{c}_i\|_{L^2}^2 
+ \bar{c}_i\right] < \varepsilon ,
\ee
then there exists a positive constant $c>0$ depending only on the size of the domain $\mathcal{O}$, such that the unique global weak solution satisfies, for any time $t \ge 0$
\be \la{longt3}
\beg{aligned}
&\sum\limits_{i=1}^{M} \|c_i(t) - \gamma_i\|_{L^2}^2 
+ \sum\limits_{i=M+1}^{N} \|c_i(t)- \bar{c}_i\|_{L^2}^2
\\&\quad\quad\quad\quad\le \left[\sum\limits_{i=1}^{M} \|c_i(0) - \gamma_i\|_{L^2}^2 
+ \sum\limits_{i=M+1}^{N} \|c_i(0)- \bar{c}_i\|_{L^2}^2 \right]e^{-{c}\min\left\{D_1, \dots, D_N\right\} t},
\end{aligned}
\ee and 
\be \la{longt4}
\int_{0}^{t} \sum\limits_{i=1}^{N} D_i \|\na c_i(s)\|_{L^2}^2 ds \le \sum\limits_{i=1}^{M} \|c_i(0) - \gamma_i\|_{L^2}^2 
+ \sum\limits_{i=M+1}^{N} \|c_i(0)- \bar{c}_i\|_{L^2}^2,  \;almost \; surely.
\ee 
\end{Thm} 

\begin{proof} 

The proof is divided into two main steps.
\smallskip

\noindent{\bf{Step 1. Potential $L^{\infty}$ bounds.}} The potential $\Phi$ solving the semi-linear Poisson equation \eqref{s211} obeys 
\be 
\|\na \Phi\|_{L^{\infty}}  
= \|\na (\Phi - \gamma)\|_{L^{\infty}}
\le C \|\rho\|_{L^4},
\ee 
due to elliptic regularity estimates and the homogeneous Dirichlet boundary conditions obeyed by $\Phi - \gamma$. Letting
\be 
\rho^* := \sum\limits_{i=1}^{M} z_i \gamma_i + \sum\limits_{i=M+1}^{N} z_i \bar{c}_i,
\ee and recalling the assumption $\rho^* = 0$ stated in \eqref{44}, we have
\be 
\|\na \Phi\|_{L^{\infty}}  
\le C\|\rho - \rho^*\|_{L^4}
\le C\sum\limits_{i=1}^{M} |z_i| \|c_i - \gamma_i\|_{L^4}
+ C \sum\limits_{i=M+1}^{N} |z_i| \|c_i - \bar{c}_i\|_{L^4}
\le C\sum\limits_{i=1}^{N} |z_i| \|\na c_i\|_{L^2},
\ee 
where the last bound follows from the Sobolev inequality and the Poincar\'e inequality applied to the boundary vanishing functions $c_i - 
\gamma_i$ for $i \in \left\{1, \dots, M\right\}$ and the mean-free functions $c_i - \bar{c}_i$ for $i \in \left\{M+1, \dots, N\right\}.$ 
\smallskip

\noindent{\bf{Step 2. Ionic concentrations $L^2$ bounds.}} We fix an index $i \in \left\{1, \dots, M\right\}$ and take the $L^2$ inner product of the equation \eqref{s210} obeyed by $c_i$ with $c_i - \gamma_i$. We obtain the energy equality
\be \la{longt1}
\fr{1}{2} \fr{d}{dt} \|c_i - \gamma_i\|_{L^2}^2 
+ D_i \|\na c_i\|_{L^2}^2
= -D_iz_i \int_{\mathcal{O}} c_i  \na \Phi \cdot \na (c_i - \gamma_i) dx, 
\ee 
after making use of the cancellations $\pa_t \gamma_i = \Delta \gamma_i = 0$ and integration by parts. Next we fix an index $i \in \left\{M+1, \dots, N\right\}$ and study the time evolution of $c_i - \bar{c}_i$ in $L^2$. In view of the blocking boundary condition obeyed by $c_i$, we have the boundary cancellation 
\be 
\int_{\pa \mathcal{O}} (\na c_i + z_ic_i \na \Phi) \cdot n d\sigma(x) = 0,
\ee 
which yields the differential equality
\be \la{longt2}
\fr{1}{2} \fr{d}{dt} \|c_i - \bar{c}_i\|_{L^2}^2
+ D_i \|\na c_i\|_{L^2}^2
= - z_iD_i \int_{\mathcal{O}} c_i \na \Phi \cdot \na (c_i - \bar{c}_i) dx, 
\ee 
after integration by parts. Adding \eqref{longt1} and \eqref{longt2} and applying H\"older's inequality gives 
\be 
\fr{1}{2}\fr{d}{dt} \left[\sum\limits_{i=1}^{M} \|c_i - \gamma_i\|_{L^2}^2 + \sum\limits_{i=M+1}^{N} \|c_i - \bar{c}_i\|_{L^2}^2 \right]
+ \sum\limits_{i=1}^{N} D_i \|\na c_i\|_{L^2}^2 \le \sum\limits_{i=1}^{N} |z_i|D_i \|c_i\|_{L^2}\|\na c_i\|_{L^2} \|\na \Phi\|_{L^{\infty}},
\ee
which implies that
\begin{align}
&\fr{1}{2}\fr{d}{dt} \left[\sum\limits_{i=1}^{M} \|c_i - \gamma_i\|_{L^2}^2 + \sum\limits_{i=M+1}^{N} \|c_i - \bar{c}_i\|_{L^2}^2 \right]
+ \sum\limits_{i=1}^{N} D_i \|\na c_i\|_{L^2}^2
\\&\quad\quad\quad\quad\le C\sum\limits_{i=1}^{N} \sum\limits_{j=1}^{N} |z_i||z_j| D_i \|c_i\|_{L^2} \|\na c_i\|_{L^2}\|\na c_j\|_{L^2}
\le C(z_i, D_i, N)\sum\limits_{i=1}^{N} \|c_i\|_{L^2}^2 \|\na c_i\|_{L^2}^2 
+ \sum\limits_{i=1}^{N} \fr{D_i}{2} \|\na c_i\|_{L^2}^2,
\end{align}
due to the potential bounds derived in \textbf{Step 1}. Here $C(z_i, D_i, N)$ is a constant depending on the maximum value of the valences, the number of ionic species, and the minimum and maximum values of the diffusivities. Consequently, we infer that 
\be \la{dd}
\beg{aligned}
&\fr{d}{dt} \left[\sum\limits_{i=1}^{M} \|c_i - \gamma_i\|_{L^2}^2 + \sum\limits_{i=M+1}^{N} \|c_i - \bar{c}_i\|_{L^2}^2 \right]
+ \fr{1}{2} \sum\limits_{i=1}^{N} D_i \|\na c_i\|_{L^2}^2
\le \sum\limits_{i=1}^{N}  \|\na c_i\|_{L^2}^2  \left[2C(z_i, D_i, N) \|c_i\|_{L^2}^2 - \fr{D_i}{2}\right]
\\&\quad\quad\le \sum\limits_{i=1}^{M}  \|\na c_i\|_{L^2}^2  \left[4C(z_i, D_i, N) \|c_i - \gamma_i\|_{L^2}^2 - \fr{D_i}{2}\right]
+\sum\limits_{i=M+1}^{N}  \|\na c_i\|_{L^2}^2  \left[4C(z_i, D_i, N)\|c_i - \bar{c}_i\|_{L^2}^2 - \fr{D_i}{2}\right]
\\&\quad\quad\quad\quad\quad\quad+ \sum\limits_{i=1}^{M}  4C(z_i, D_i, N)\gamma_i^2 |\mathcal{O}| \|\na c_i\|_{L^2}^2  
+\sum\limits_{i=M+1}^{N}  4C(z_i, D_i, N)\bar{c}_i^2 |\mathcal{O}| \|\na c_i\|_{L^2}^2. 
\end{aligned}
\ee 
Supposing that
\be 
4C(z_i, D_i, N) \gamma_i^2 |\mathcal{O}| \le \fr{D_i}{8}
\text{ for } i \in \left\{1, \dots, M\right\}, \quad \text{and} \quad 
4C(z_i, D_i, N) \bar{c}_i^2 |\mathcal{O}| \le \fr{D_i}{8}, \text{ for } i \in \left\{M+1, \dots, N\right\},
\ee
the last two terms on the right-hand side of \eqref{dd} get absorbed by the dissipation on the left-hand side of \eqref{dd}, yielding 
\be \la{dd1}
\beg{aligned}
&\fr{d}{dt} \left[\sum\limits_{i=1}^{M} \|c_i - \gamma_i\|_{L^2}^2 + \sum\limits_{i=M+1}^{N} \|c_i - \bar{c}_i\|_{L^2}^2 \right]
+ \fr{1}{4} \sum\limits_{i=1}^{N} D_i \|\na c_i\|_{L^2}^2
\\&\quad\quad\quad\quad\le \sum\limits_{i=1}^{M}  \|\na c_i\|_{L^2}^2  \left[4C(z_i, D_i, N)\|c_i - \gamma_i\|_{L^2}^2 - \fr{D_i}{2}\right]
+\sum\limits_{i=M+1}^{N}  \|\na c_i\|_{L^2}^2  \left[4C(z_i, D_i, N) \|c_i - \bar{c}_i\|_{L^2}^2 - \fr{D_i}{2}\right].
\end{aligned}
\ee
Moreover, if the initial concentrations satisfy the bounds
\be 
\sum\limits_{i=1}^{M} \|c_i(0) - \gamma_i\|_{L^2}^2 + \sum\limits_{i=M+1}^{N} \|c_i(0) - \bar{c}_i\|_{L^2}^2 < \fr{\min\left\{D_1, \dots, D_N\right\}}{8C(z_i, D_i, N)},
\ee 
then by a continuity argument, we conclude that 
\be 
\sum\limits_{i=1}^{M} \|c_i(t) - \gamma_i\|_{L^2}^2 + \sum\limits_{i=M+1}^{N} \|c_i(t) - \bar{c}_i\|_{L^2}^2 < \fr{\min\left\{D_1, \dots, D_N\right\}}{8C(z_i, D_i, N)},
\ee 
for all times $t \ge 0$.Therefore, the differential inequality 
\be 
\fr{d}{dt} \left[\sum\limits_{i=1}^{M} \|c_i - \gamma_i\|_{L^2}^2 + \sum\limits_{i=M+1}^{N} \|c_i - \bar{c}_i\|_{L^2}^2 \right]
+ \fr{1}{4} \sum\limits_{i=1}^{N} D_i \|\na c_i\|_{L^2}^2
\le 0
\ee 
holds for all $t \ge 0.$ Applications of the Poincar\'e inequality give the desired decay in $L^2$ described by \eqref{longt3}. Integrating in time from $0$ to $t$, we obtain the $L^2$ gradient estimate \eqref{longt4}. This ends the proof of Theorem \ref{longt}. 
\end{proof}

{The next theorem addresses the $L^p$ regularity of ionic concentrations and their long-term behavior in $L^p$ norms.}

\beg{Thm} \la{t4} 
Let $p \in (2, \infty)$ be an even integer. Let $u_0\in H$ and $c_i(0) \in L^p$ be nonnegative for each $i \in  \left\{1, \dots, N\right\}$. Suppose the initial ionic concentrations satisfy the $L^2$ smallness condition \eqref{SC} imposed in Theorem \ref{longt}. Then there exists a positive constant $C_p$ depending on the size of the initial data in $L^{p}$, the parameters of the problem, and $p$, and a positive constant $c_p$ depending on $p$ and the parameters of the problem, such that the estimate
\be
\sum\limits_{i=1}^{M} \|c_i(t) - \gamma_i\|_{L^{p}} +\sum\limits_{i=M+1}^{N} \|c_i(t) - \bar{c}_i\|_{L^{p}} 
\le C_p e^{-c_p t}
\ee 
holds for all times $t \ge 0$ and almost surely. 
\end{Thm}

\begin{proof}
{The proof is divided into two main steps distinguishing the cases of ionic concentrations with Dirichlet boundary conditions and blocking boundary conditions separately.
}

\smallskip
\noindent {\bf Step 1. Ionic concentrations with Dirichlet boundary conditions.}
We fix an index $i \in \left\{1, \dots, M\right\}$ and multiply the ionic concentration equation \eqref{s210} obeyed by the corresponding $c_i$ by $(c_i - \gamma_i)^{p-1}$. Since the nonlinear term in $u$ vanishes, we obtain the energy evolution
\be \la{longt6}
\fr{1}{p} \fr{d}{dt} \|c_i - \gamma_i\|_{L^p}^p
- D_i \int_{\mathcal{O}}  (c_i - \gamma_i)^{p-1} \Delta (c_i - \gamma_i) dx
= z_i D_i \int_{\mathcal{O}} \na \cdot (c_i \na \Phi) (c_i - \gamma_i)^{p-1} dx.
\ee 
Since $c_i$ amounts to $\gamma_i$ on $\pa \mathcal{O}$, we can integrate by parts the diffusion term as follows,
\be 
\beg{aligned}
&- D_i \int_{\mathcal{O}}  (c_i - \gamma_i)^{p-1} \Delta (c_i - \gamma_i) dx
= D_i \int_{\mathcal{O}}  \na (c_i - \gamma_i)^{p-1} \cdot \na (c_i - \gamma_i) dx
\\&\quad\quad= D_i (p-1) \int_{\mathcal{O}}  (c_i - \gamma_i)^{p-2} \na c_i \cdot \na c_i dx
= D_i (p-1) \|(c_i - \gamma_i)^{\fr{p-2}{2}} \na c_i\|_{L^2}^2,
\end{aligned}
\ee 
where the last equality uses the positivity of $(c_i - \gamma_i)^{p-2}$ that follows from the evenness of the integer $p$. A similar argument allows us to rewrite the migration term as 
\be 
\beg{aligned}
z_i D_i \int_{\mathcal{O}} \na \cdot (c_i \na \Phi) (c_i - \gamma_i)^{p-1} dx.
&= - z_iD_i (p-1) \int_{\mathcal{O}} (c_i - \gamma_i) (c_i - \gamma_i)^{p-2} \na \Phi \cdot \na c_i dx
\\&\quad\quad- z_iD_i \gamma_i (p-1) \int_{\mathcal{O}} (c_i - \gamma_i)^{p-2} \na \Phi \cdot \na c_i dx,
\end{aligned}
\ee 
which can be estimated as
\be \la{longt7}
\beg{aligned}
&|z_i D_i \int_{\mathcal{O}} \na \cdot (c_i \na \Phi) (c_i - \gamma_i)^{p-1} dx| 
\\&\quad\quad\le |z_i|D_i (p-1) \|\na \Phi\|_{L^{\infty}} \|(c_i - \gamma_i)^{\fr{p-2}{2} + 1}\|_{L^2} \|(c_i - \gamma_i)^{\fr{p-2}{2}} \na c_i\|_{L^2}
\\&\quad\quad\quad\quad+ |z_i|D_i \gamma_i(p-1) \|\na \Phi\|_{L^p} \|(c_i - \gamma_i)^{\fr{p-2}{2}}\|_{L^{\fr{2p}{p-2}}} \|(c_i - \gamma_i)^{\fr{p-2}{2}} \na c_i\|_{L^2}
\\&\quad\quad\le |z_i|D_i (p-1) \|(c_i - \gamma_i)^{\fr{p-2}{2}} \na c_i\|_{L^2} \|\na \Phi\|_{L^{\infty}} \|c_i - \gamma_i\|_{L^p}^{\fr{p}{2}} 
\\&\quad\quad\quad\quad+  C|z_i|D_i\gamma_i (p-1) \|(c_i - \gamma_i)^{\fr{p-2}{2}} \na c_i\|_{L^2} \|\rho\|_{L^{2}}  \|c_i - \gamma_i\|_{L^p}^{\fr{p-2}{2}},
\end{aligned}
\ee
via applications of H\"older's inequality, the Sobolev inequality, and elliptic regularity estimates. Putting \eqref{longt6}--\eqref{longt7} together and applying Young's inequality for products, we end up with the differential inequality
\be \la{longt8}
\beg{aligned}
&\fr{1}{p} \fr{d}{dt} \|c_i - \gamma_i\|_{L^p}^p + \fr{D_i(p-1)}{2} \|(c_i - \gamma_i)^{\fr{p-2}{2}} \na c_i\|_{L^2}^2 
\\&\quad\quad\le C|z_i|^2 D_i(p-1) \left[\|\na \Phi\|_{L^{\infty}}^2  \|c_i - \gamma_i\|_{L^p}^p + C\gamma_i^2\|\rho\|_{L^2}^2 \|c_i - \gamma_i\|_{L^p}^{p-2} \right].
\end{aligned} 
\ee 
Now we note that the dissipation amounts to 
\be 
\beg{aligned}
\|(c_i - \gamma_i)^{\fr{p-2}{2}} \na c_i\|_{L^2}^2
&= ((c_i - \gamma_i)^{\fr{p-2}{2}} \na (c_i - \gamma_i), (c_i - \gamma_i)^{\fr{p-2}{2}} \na (c_i - \gamma_i))_{L^2}
\\&= \fr{4}{p^2}(\na (c_i - \gamma_i)^{\fr{p}{2}}, \na (c_i - \gamma_i)^{\fr{p}{2}})_{L^2} = \fr{4}{p^2} \|\na (c_i - \gamma_i)^{\fr{p}{2}}\|_{L^2}^2,
\end{aligned}
\ee 
which, after using the Poincar\'e inequality applied to the boundary vanishing function $(c_i - \gamma_i)^{\fr{p}{2}}$, boils down to
\be \la{poin}
\|(c_i - \gamma_i)^{\fr{p-2}{2}} \na c_i\|_{L^2}^2
\ge \fr{4c_1}{p^2} \|(c_i - \gamma_i)^{\fr{p}{2}}\|_{L^2}^2
= \fr{4c_1}{p^2} \|c_i - \gamma_i\|_{L^p}^p,
\ee 
where $c_1$ is the constant from the Poincar\'e inequality.
Therefore, the energy inequality \eqref{longt8} implies that
\be
\beg{aligned}
&\fr{1}{p} \fr{d}{dt} \|c_i - \gamma_i\|_{L^p}^p + \fr{4c_1D_i(p-1)}{2p^2} \|c_i - \gamma_i\|_{L^p}^p
\\&\quad\quad\le C|z_i|^2 D_i(p-1) \left[\|\na \Phi\|_{L^{\infty}}^2  \|c_i - \gamma_i\|_{L^p}^p + C\gamma_i^2 \|\rho\|_{L^2}^2 \|c_i - \gamma_i\|_{L^p}^{p-2} \right].
\end{aligned} 
\ee
Dividing both sides by $\|c_i - \gamma_i\|_{L^p}^{p-2}$, this latter inequality reduces to 
\be
\fr{1}{2} \fr{d}{dt} \|c_i - \gamma_i\|_{L^p}^2 + \fr{4 c_1 D_i(p-1)}{2p^2} \|c_i - \gamma_i\|_{L^p}^{2}
\le C|z_i|^2 D_i(p-1) \left[\|\na \Phi\|_{L^{\infty}}^2  \|c_i - \gamma_i\|_{L^p}^{2} + C\gamma_i^2 \|\rho\|_{L^2}^2\right], 
\ee
from which we deduce that
\be
\beg{aligned}
\fr{d}{dt} \|c_i - \gamma_i\|_{L^p}^2 + r(t)  \|c_i - \gamma_i\|_{L^p}^{2}
\le C|z_i|^2 \gamma_i^2 D_i(p-1) \|\rho\|_{L^2}^2,
\end{aligned} 
\ee 
and $r(t)$ is defined by 
\be 
r(t):= \min \left\{\fr{4c_1 D_i(p-1)}{p^2}, \fr{c\min\left\{D_1, \dots, D_N \right\}}{2} \right\} - C|z_i|^2 D_i(p-1) \|\na \Phi\|_{L^{\infty}}^2,
\ee 
(where c is the constant in \eqref{longt3}). We multiply by the integrating factor $e^{\int_{0}^{t} r(s) ds}$ and integrate in time from $0$ to $t$. In view of the $L^2$ decaying estimates established in Theorem \ref{longt}, we have 
\be 
\beg{aligned}
&|z_i|^2 \gamma_i^2 D_i(p-1) \int_{0}^{t} e^{\int_{0}^{s} r(\zeta) d\zeta }\|\rho(s)\|_{L^2}^2 ds
\\&\quad\quad\le C|z_i|^4 \gamma_i^2 D_i(p-1) \int_{0}^{t} e^{\int_{0}^{s} r(\zeta) d\zeta } \left(\sum\limits_{j=1}^{M} \|c_i - \gamma_i \|_{L^2}^2 + \sum\limits_{j=M+1}^{N} \|c_i - \bar{c}_i \|_{L^2}^2 \right) ds
\\&\quad\quad\le C_0{|z_i|^4} \gamma_i^2  D_i(p-1) \int_{0}^{t} e^{\fr{c \min \left\{D_1, \dots, D_N\right\}}{2}s} e^{-c\min \left\{D_1, \dots, D_N\right\}s} ds 
\le C_1(\|c_i(0)\|_{L^2}, z_i, D_i, \gamma_i) (p-1),
\end{aligned}
\ee 
where $C_1(\|c_i(0)\|_{L^2}, z_i, D_i, \gamma_i)$ is a positive constant depending only on the $L^2$ norm of the initial data, the valences and diffusivities of the ionic species, and the Dirichlet boundary data obeyed by $c_i$. Consequently, the $i$th concentration $L^p$ estimate
\be 
\|c_i - \gamma_i\|_{L^p}^2 \le e^{- \int_{0}^{t} r(s) ds} \left[\|c_i(0) - \gamma_i\|_{L^p}^2 + C_1(\|c_i(0)\|_{L^2}, z_i, D_i, \gamma_i)  (p-1) \right]
\ee 
holds for all times $t \ge 0$. Making use of \eqref{longt4} yields
\be 
\beg{aligned}
\int_{0}^{t} r(s) ds 
&\geq \min \left\{\fr{4c_1D_i(p-1)}{p^2}, \fr{c\min\left\{D_1, \dots, D_N\right\}}{2} \right\}t 
- Cz_i^2 D_i (p-1) \int_{0}^{t} \sum\limits_{j=1}^{N} z_j^2\|\na c_j(s)\|_{L^2}^2  ds
\\&\ge \min \left\{\fr{4c_1 D_i(p-1)}{p^2}, \fr{c\min\left\{D_1, \dots, D_N \right\}}{2} \right\}t 
- C_2(\|c_i(0)\|_{L^2}, z_i, D_i, \gamma_i)  (p-1), 
\end{aligned}
\ee 
for some positive constant $C_2(\|c_i(0)\|_{L^2}, z_i, D_i, \gamma_i)$  depending only on the size of the initial concentrations in $L^2$ and the parameters of the problem. Thus, we obtain that, for any $t \ge 0$,
\be 
\|c_i(t) - \gamma_i\|_{L^p}^2 
\le e^{C_2 (p-1)} \left[\|c_i(0) - \gamma_i\|_{L^p}^2 + C_1 (p-1) \right] e^{-\min \left\{\fr{ c_1 D_i(p-1)}{p^2}, \fr{c\min\left\{D_1, \dots, D_N \right\}}{2}\right\} t}.
\ee 

\smallskip
\noindent {\bf Step 2. Ionic concentrations with blocking boundary conditions.} Now we proceed to study the $L^p$ asymptotic behavior of the ionic concentrations having blocking boundary conditions. The situation differs from the case of Dirichlet boundary conditions due to the absence of the simplified Poincar\'e inequality \eqref{poin}. We present a proof by induction on $p$ from which the desired $L^p$ decay follows. Indeed, suppose there is a positive constant $\Gamma_p^1$ depending on the size of the initial data in $L^{p-2}$, the parameters of the problem, and $p$, and a positive constant $\Gamma_p^2$ depending on $p$ and the parameters of the problem, such that, for all times $t \geq 0$,  the estimate
\be \la{ind1}
\|c_i(t)- \bar{c}_i\|_{L^{p-2}} \le \Gamma_p^1 e^{-\Gamma_p^2 t}
\ee 
holds. We aim to show that there is a positive constant $\kappa_p^1$ depending on the size of the initial data in $L^{p}$, the parameters of the problem, and $p$, and a positive constant $\kappa_p^2$ depending on $p$ and the parameters of the problem, such that for all times $t \geq 0$, the following estimate holds
\be \la{ind2}
\|c_i(t) - \bar{c}_i\|_{L^{p}} \le \kappa_p^1 e^{-\kappa_p^2 t}.
\ee 

Fixing an index $i \in \left\{M+1, \dots, N \right\}$, we have
\be \la{longt9}
\beg{aligned}
&\fr{1}{p} \fr{d}{dt} \|c_i - \bar{c}_i\|_{L^p}^p + D_i(p-1)\|(c_i - \bar{c}_i)^{\fr{p-2}{2}} \na c_i\|_{L^2}^2 
\\&\quad\quad= -z_iD_i (p-1) \int_{\mathcal{O}} (c_i - \bar{c}_i)^{p-1} \na \Phi \cdot \na c_i  dx
- z_iD_i \bar{c}_i (p-1) \int_{\mathcal{O}} (c_i - \bar{c}_i)^{p-2} \na \Phi \cdot \na c_i dx,
\end{aligned} 
\ee
where the boundary cancellation
\be 
\int_{\pa \mathcal{O}} (\na c_i + z_i c_i \na \Phi)(c_i - \bar{c}_i)^{p-1}\cdot n d\sigma(x) =0,
\ee 
is exploited. We estimate
\be 
\beg{aligned}
\left|z_iD_i (p-1) \int_{\mathcal{O}} (c_i - \bar{c}_i)^{p-1} \na \Phi \cdot \na c_i  dx \right|
&\le |z_i|D_i(p-1) \|\na \Phi\|_{L^{\infty}} \|(c_i - \bar{c}_i)^{\fr{p-2}{2}} \na c_i\|_{L^2} \|(c_i - \bar{c}_i)^{\fr{p}{2}}\|_{L^2}
\\&= |z_i|D_i(p-1) \|\na \Phi\|_{L^{\infty}} \|(c_i - \bar{c}_i)^{\fr{p-2}{2}} \na c_i\|_{L^2} \|c_i - \bar{c}_i\|_{L^p}^{\fr{p}{2}},
\end{aligned}
\ee 
and 
\be 
\beg{aligned}
\left| z_iD_i \bar{c}_i (p-1) \int_{\mathcal{O}} (c_i - \bar{c}_i)^{p-2} \na \Phi \cdot \na c_i dx \right|
&\le |z_i|D_i \bar{c}_i (p-1) \|\na \Phi\|_{L^{\infty}} \|(c_i - \bar{c}_i)^{\fr{p-2}{2}} \na c_i\|_{L^2}\|(c_i - \bar{c}_i)^{\fr{p-2}{2}}\|_{L^2}
\\&=  |z_i|D_i \bar{c}_i (p-1) \|\na \Phi\|_{L^{\infty}} \|(c_i - \bar{c}_i)^{\fr{p-2}{2}} \na c_i\|_{L^2}\|c_i - \bar{c}_i\|_{L^{p-2}}^{\fr{p-2}{2}},
\end{aligned}  
\ee 
through applications of H\"older's inequality. Taking advantage of the dissipation governing the $L^p$ evolution produces
\be \la{longt10}
\beg{aligned}
&\fr{1}{p} \fr{d}{dt} \|c_i - \bar{c}_i\|_{L^p}^p + \fr{D_i(p-1)}{2}\|(c_i - \bar{c}_i)^{\fr{p-2}{2}} \na c_i\|_{L^2}^2 
\\&\quad\quad\le C|z_i|^2 D_i (p-1) \|\na \Phi\|_{L^{\infty}}^2 \left(\|c_i - \bar{c}_i\|_{L^p}^p + \bar{c}_i^2 \|c_i - \bar{c}_i\|_{L^{p-2}}^{p-2} \right),
\end{aligned}
\ee 
by applying Young's inequality. In view of the Poincar\'e inequality, we have
\be 
\left\|\left(c_i - \bar{c}_i \right)^{\fr{p}{2}} - \fr{1}{|\mathcal{O}|} \int_{\mathcal{O}} \left(c_i - \bar{c}_i \right)^{\fr{p}{2}} dx \right\|_{L^2}^2
\le C\|\na (c_i - \bar{c}_i)^{\fr{p}{2}}\|_{L^2}^2
= \fr{Cp^2}{4} \|(c_i - \bar{c}_i)^{\fr{p-2}{2}} \na c_i\|_{L^2}^2,
\ee 
which, after a straightforward application of the reverse triangle inequality, yields
\be 
\|\left(c_i - \bar{c}_i \right)^{\fr{p}{2}}\|_{L^2}^2
\le Cp^2 \|(c_i - \bar{c}_i)^{\fr{p-2}{2}} \na c_i\|_{L^2}^2
+ C\left(\fr{1}{|\mathcal{O}|} \int_{\mathcal{O}} \left(c_i - \bar{c}_i \right)^{\fr{p}{2}} dx \right)^2.
\ee 
Here $C$ is a constant depending on the size of the domain $\mathcal{O}$. From \eqref{longt10}, we deduce that 
\be 
\beg{aligned}
&\fr{1}{p} \fr{d}{dt} \|c_i - \bar{c}_i\|_{L^p}^p 
+ \fr{D_i(p-1)}{2Cp^2} \|c_i - \bar{c}_i\|_{L^p}^p 
\\&\quad\quad\quad\quad\le C|z_i|^2 D_i(p-1) \|\na \Phi\|_{L^{\infty}}^2 \left[  \|c_i - \bar{c}_i\|_{L^p}^p + \bar{c}_i^2  \|c_i - \bar{c}_i\|_{L^{p-2}}^{p-2} \right]
+ \fr{1}{p^2} \left(\fr{1}{|\mathcal{O}|} \int_{\mathcal{O}} \left(c_i - \bar{c}_i \right)^{\fr{p}{2}} dx \right)^2.
\end{aligned}
\ee 
Since $p \ge 4$, the exponent $\fr{2(p-2)}{p}$ is greater than or equal to 1, thus the average on the right-hand side of the latter differential inequality bounds as
\be 
\fr{1}{p^2} \left(\fr{1}{|\mathcal{O}|} \int_{\mathcal{O}} \left(c_i - \bar{c}_i \right)^{\fr{p}{2}} dx \right)^2
\le  C(p, \mathcal{O}) \left(\int_{\mathcal{O}}  \left(c_i - \bar{c}_i \right)^{p-2}dx  \right)^{\fr{p}{p-2}}, 
\ee 
where $C(p,\mathcal{O})$ is a constant depending on $p$ and the diameter of $\mathcal{O}$. This gives rise to 
\be 
\fr{d}{dt} \|c_i - \bar{c}_i\|_{L^p}^p + r_1(t) \|c_i - \bar{c}_i\|_{L^p}^p 
\le C(\mathcal{O}, z_i, D_i, p) \left[\bar{c}_i^2\|\na \Phi\|_{L^{\infty}}^2 \|c_i - \bar{c}_i\|_{L^{p-2}}^{p-2}
+\|c_i - \bar{c}_i\|_{L^{p-2}}^p\right],
\ee 
and
\be 
r_1(t) = \min\left\{\fr{D_i (p-1)}{2Cp}, (p-2)\Gamma_p^2, \fr{p\Gamma_p^2}{2} \right\} - C|z_i|^2 D_i p(p-1) \|\na \Phi\|_{L^{\infty}}^2.
\ee 
We multiply by the integrating factor $e^{\int_{0}^{t} r_1(s) ds}$ and integrate in time from $0$ to $t$. We use the induction hypothesis \eqref{ind1} to obtain good control of the $L^{p-2}$ terms. In fact, we have 
\be  
\beg{aligned}
&\left|\int_{0}^{t} e^{\int_{0}^{s} r_1(\zeta) d\zeta} \|\na \Phi\|_{L^{\infty}}^2 \|c_i - \bar{c}_i\|_{L^{p-2}}^{p-2} ds \right|
\\&\quad\quad\le C(\Gamma_p^1)^{p-2}\int_{0}^{t} e^{(p-2)\Gamma_p^2s} e^{-(p-2)\Gamma_p^2s} \left(\sum\limits_{j=1}^{N} z_j^2 \|\na c_j(s)\|_{L^2}^2\right) ds
\le C(\Gamma_p^1)^{p-2} \int_{0}^{t} \sum\limits_{j=1}^{N} \|\na c_j(s)\|_{L^2}^2 ds 
\le C(\Gamma_p^1)^{p-2},
\end{aligned}
\ee  
where the last inequality follows from \eqref{longt4}, and 
\be 
\left|\int_{0}^{t} e^{\int_{0}^{s} r_1(\zeta) d\zeta} \|c_i - \bar{c}_i\|_{L^{p-2}}^p ds\right|
\le C(\Gamma_p^1)^p\int_{0}^{t} e^{\fr{p\Gamma_p^2}{2}s} e^{-p\Gamma_p^2s} ds
\le Cp^{-1}(\Gamma_p^1)^p (\Gamma_p^2)^{-1}.
\ee 
As a consequence of the boundedness of these time integrals, we conclude that
\be 
\|c_i(t) - \bar{c}_i\|_{L^p}^p
\le e^{- \int_0^t r_1(s)ds} \left[\|c_i(0) - \bar{c}_i\|_{L^p}^p + C(\mathcal{O}, z_i, D_i, p, \|c_i(0)\|_{L^{p-2}})\right],
\ee where $ C(\mathcal{O}, z_i, D_i, p, \|c_i(0)\|_{L^{p-2}})$ is a positive constant depending on the diameter of $\mathcal{O}$, $z_i, D_i, p$, and the $L^{p-2}$ norm of the ionic concentrations. We use again the global integrability estimate \eqref{longt4} to obtain uniform-in-time control of $\int_{0}^{t}\|\na \Phi\|_{L^{\infty}}^2 dt$ and deduce that \eqref{ind2} holds. This ends the proof of Theorem \ref{t4}.
\end{proof}

\section{S-NPNS Semigroup and Feller Properties} \label{fel}

This section is dedicated to Feller properties of the Markovian semigroup associated with the S-NPNS system \eqref{sys:SNPNS}. To this end, we denote by $\mathcal{H}$ the $N+1$ product space 
\be \label{Hspace}
\mathcal{H} := H \times L^2 \times \dots L^2
\ee  of vectors $\omega = (v, \xi_1, \dots, \xi_N)$ where $v \in H$ and $\xi_i \in L^2$ for all $i \in \left\{1, \dots, N \right\}$, equipped with the natural norm 
$ 
\|\omega\|_{\mathcal{H}}^2 = \|v\|_{L^2}^2 + \sum\limits_{i=1}^{N} \|\xi_i\|_{L^2}^2.
$ 
For given data $\gamma$ and $\gamma_1, \dots, \gamma_M$  with valences $z_1, \dots, z_N$, we consider the space $\tilde{\mathcal{H}}$ consisting of vectors $(v, \xi_1, \dots, \xi_N) \in \mathcal{H}$ such that the scalar functions $\xi_1, \dots, \xi_N$ are nonnegative a.e. and satisfy 
\be  
\sum\limits_{i=1}^{M} z_i \xi_i|_{\pa \mathcal{O}} 
+ \sum\limits_{i=M+1}^{N} \fr{z_i}{|\mathcal{O}|} \int_{\mathcal{O}} \xi_i(x) dx = 0,
\ee {and such that the vector $(v, \xi_1, \dots, \xi_N)$ and scalar 
\be 
\Psi := (-\Delta_D)^{-1} \left(\sum\limits_{i=1}^{N} z_i\xi_i \right) + \gamma
\ee obey the boundary conditions
\be
\begin{split}
    &v|_{\pa \mathcal{O}} = 0, \quad (\xi_1,\dots,\xi_M)|_{\pa \mathcal{O}} = (\gamma_1, \dots, \gamma_M), \quad \Psi|_{\partial \mathcal O}=\gamma,
    \\
    &\left((\na \xi_{M+1} + z_{M+1} \xi_{M+1} \na \Psi)|_{\pa \mathcal{O}} \cdot n, \dots, (\na \xi_{N} + z_N \xi_{N} \na \Psi)|_{\pa \mathcal{O}} \cdot n \right)=(0, \dots, 0).
\end{split}
\ee Here $\Delta_D$ is the 2D Laplacian with homogeneous Dirichlet boundary conditions.
} 

For a positive time $t \ge 0$ and a Borel set $A \in \mathcal{B}(\tilde{\mathcal{H}})$, we define the Markov transition kernels associated with the S-NPNS system \eqref{sys:SNPNS} with the mixed boundary conditions \eqref{3} by 
\be \label{transitionkernels}
P_t(\omega_0, A) := \PP (\omega (t, \omega_0) \in A),
\ee
where $\omega(t, \omega_0)$ denotes the solution $\omega = (u, c_1, \dots, c_N)$ to the problem \eqref{sys:SNPNS} with boundary conditions \eqref{3} and initial datum $\omega_0 = (u_0, c_1(0), \dots, c_N(0))$.
Let $\mathcal{M}_b(\tilde{\mathcal{H}})$ be the collection of bounded real-valued Borel measurable functions on $\tilde{\mathcal{H}}$. For each $t \ge 0$ and $\phi \in \mathcal{M}_b(\tilde{\mathcal{H}})$, we define the Markovian semigroup, denoted by $\tilde P_t$, by 
\be \label{markovsemigroup}
\tilde P_t \phi (\omega_0) := \mathbb{E}\phi(\omega(t, \omega_0)) = \int_{\tilde{\mathcal{H}}} \phi (\omega) P_t (\omega_0, d\omega).
\ee 

Let $C_b(\tilde{\mathcal{H}})$ be the space of continuous bounded real-valued functions on $\tilde{\mathcal{H}}$. The semigroup $\left\{\tilde P_t\right\}_{t \ge 0}$ obeys the following property:

\beg{Thm} \la{feller} 

The semigroup $\left\{\tilde P_t\right\}_{t \ge 0}$ is Markov Feller on the space $C_b(\tilde{\mathcal{H}})$. That is, if $\phi \in C_b(\tilde{\mathcal{H}})$ then $\tilde P_t \phi \in C_b(\tilde{\mathcal{H}})$.

\end{Thm}

In order to prove Theorem \ref{feller}, we need the following preliminary proposition:

\beg{prop} \la{cont} 
Let $\omega_0^1 = (u_0^1, c_1^1(0), \dots, c_N^1(0))$ and $\omega_0^2 = (u_0^2, c_1^2(0), \dots, c_N^2(0))$ be in $\tilde{\mathcal{H}}$. Then the solutions $\omega^1(t) = (u^1(t), c_1^1(t), \dots, c_N^1(t))$ and $\omega^2(t) = (u^2(t), c_1^2(t), \dots, c_N^2(t))$ to the S-NPNS system with respective initial data $\omega_0^1$ and $\omega_0^2$ satisfy
\be \la{con13}
\|\omega^1(t) - \omega^2(t)\|_{\mathcal{H}}^2 \le e^{\kappa(t)} \|\omega_0^1 - \omega_0^2\|_{\mathcal{H}}^2,
\ee 
where
\be \la{con14}
\kappa(t) = C\int_{0}^{t} \left\{\sum\limits_{i=1}^{N} \left(\|c_i^1(s)\|_{L^2}^2 + \|c_i^1(s)\|_{L^2}^4 \right) + \sum\limits_{i=1}^{N} \| c_i^2(s)\|_{L^4}^4 + \|\na \Phi^2(s)\|_{L^{\infty}}^2    + \|\na u^2(s)\|_{L^2}^2\right\} ds.
\ee 
\end{prop}

We will first prove Theorem \ref{feller}, assuming Proposition \ref{cont} holds.

\begin{proof}[Proof of Theorem \ref{feller}]
Let $\phi \in C_b(\tilde{\mathcal{H}})$. Suppose $\omega_0^n = (v^n, \xi_1^n, \dots, \xi_N^n)$ is a sequence in $\tilde{\mathcal{H}}$ that converges to $\omega_0 = (v, \xi_1, \dots, \xi_N)$ in the norm of $\mathcal{H}$, that is 
$
\|v^n - v\|_{L^2} + \sum\limits_{i=1}^{N} \|\xi_i^n - \xi\|_{L^2} \rightarrow 0,
$ 
as $n \rightarrow \infty$. We denote by $\omega(t, \omega_0^n)$ and $\omega(t, \omega_0)$ the solutions to the S-NPNS system at time $t$ corresponding to the initial data $\omega_0^n$ and $\omega_0$, respectively. As a consequence of the Lipschitz continuity estimate derived in Proposition \ref{cont}, we have 
\be 
\|\omega(t, \omega_0^n) - \omega(t, \omega_0)\|_{\mathcal{H}}^2
\le e^{C(K_n(t)+ K(t))} \|\omega_0^n - \omega_0\|_{\mathcal{H}}^2,
\ee 
with
\be 
K_n(t) = \int_{0}^{t} \left\{\sum\limits_{i=1}^{N} \left(\|c_i(s, \omega_0^n)\|_{L^2}^2 + \|c_i(s, \omega_0^n)\|_{L^2}^4 \right) \right\}ds,
\ee 
and 
\be 
K(t) = \int_{0}^{t} \left\{\sum\limits_{i=1}^{N} \| c_i(s, \omega_0)\|_{L^4}^4 + \|\na \Phi(s, \omega_0)\|_{L^{\infty}}^2    + \|\na u(s, \omega_0)\|_{L^2}^2\right\} ds. 
\ee 
In view of the regularity of weak solutions obtained in Theorem \ref{glos}, we have  $K(t) < \infty$. Moreover, $K_n(t)$ is uniformly bounded by some constant depending on $t$, the parameters of the problem, and the size of the sequence $\omega_0^n$ in $L^2$. Since the $L^2$ norm of $\omega_0^n$ is convergent, we deduce that this latter sequence is bounded in $L^2$, from which we obtain the uniform boundedness in $n$ of $K_n(t)$ at each instant $t \ge 0$. Thus, 
\be 
\limsup\limits_{n \to \infty} K_n(t) < \infty,
\ee 
for all $t \ge 0$. It follows that
\be 
\limsup\limits_{n \to \infty} \|\omega(t, \omega_0^n) - \omega(t, \omega_0)\|_{\mathcal{H}}^2 = 0,
\ee which implies that
\be 
\lim\limits_{n \to \infty} \|\omega(t, \omega_0^n) - \omega(t, \omega_0)\|_{\mathcal{H}}^2 = 0.
\ee Therefore, it holds that 
\be 
\lim\limits_{n \to \infty} \mathbb{E}\phi (\omega(t, w_0^n)) =  \mathbb{E}\phi (\omega(t, w_0)), 
\ee 
due to the continuity of $\phi$ and the Dominated Convergence Theorem. This ends the proof of Theorem \ref{feller}. 
\end{proof}

We now present the proof of Proposition \ref{cont}.

\begin{proof}[Proof of Proposition \ref{cont}]
We let $u = u^1 - u^2$, $c_i = c_i^1 - c_i^2$ for $i \in \left\{1, \dots, N\right\}$, $\rho = \rho^1 - \rho^2$, and $\Phi = \Phi^1 - \Phi^2$. These differences satisfy the system of deterministic equations
\noeqref{con3}
\noeqref{con4}
\begin{subequations}
    \begin{align}
        \la{con1}
&\pa_t u + Au = - \mathcal{P}(u^1 \cdot \na u) - \mathcal{P}(u \cdot \na u^2) - \mathcal{P}(\rho^1 \na \Phi) - \mathcal{P}(\rho \na \Phi^2),
\\
\la{con2}
&\pa_t c_i - D_i \Delta c_i 
= - u^1 \cdot \na c_i - u \cdot \na c_i^2
+ D_iz_i \na \cdot (c_i^1 \na \Phi)
+ D_i z_i \na \cdot (c_i \na \Phi^2),
\\
\la{con3}
&- \Delta \Phi = \rho,
\\
\la{con4}
&\na \cdot u = 0,
    \end{align}
\end{subequations}
with boundary conditions
\begin{align} \la{con5}
&u|_{\pa \mathcal{O}} = 0,\quad \Phi|_{\pa \mathcal{O}} = 0,\\
c_i|_{\pa \mathcal{O}} = 0, \;\text{ for } i \in \left\{1, \dots, M\right\},  \quad &\text{and} \quad 
\left(\na c_i + z_i (c_i^1 \na \Phi + c_i \na \Phi^2)\right)|_{\pa \mathcal{O}} \cdot n = 0, \text{ for } i \in \left\{M+1, \dots, N \right\}.\la{con78}
\end{align}
Taking the scalar product of the velocity equation \eqref{con1} in $L^2$ with $u$, we obtain
\be \la{con9}
\fr{1}{2} \fr{d}{dt} \|u\|_{L^2}^2 
+ \|A^{\fr{1}{2}} u\|_{L^2}^2
= - \int_{\mathcal{O}} (u \cdot \na u^2) \cdot u dx
- \int_{\mathcal{O}} (\rho^1  \na \Phi) \cdot u dx
- \int_{\mathcal{O}} (\rho \na \Phi^2) \cdot u dx,
\ee 
where the divergence-free condition of $u$ and the self-adjointness of the Leray projector $\mathcal{P}$ are exploited. 
Taking the $L^2$ inner product of the ionic concentration equation \eqref{con2} with $c_i$ and making use of the homogeneous Dirichlet boundary conditions when the index $i \in \left\{1, \dots, M\right\}$ and the blocking boundary conditions  when $i \in \left\{M+1, \dots, N\right\}$ \eqref{con78}, we deduce that each $c_i$ evolves according to the energy equality
\be \la{con10}
\fr{1}{2} \fr{d}{dt} \|c_i\|_{L^2}^2 
- D_i \|\na c_i\|_{L^2}^2
= - \int_{\mathcal{O}} u \cdot \na c_i^2 c_i dx
- D_iz_i \int_{\mathcal{O}} c_i^1 \na \Phi \cdot \na c_i dx
- D_iz_i \int_{\mathcal{O}} c_i \na \Phi^2 \cdot \na c_i dx,
\ee 
after integrating by parts. 
We estimate 
\be \la{con11}
\left|\int_{\mathcal{O}} (u \cdot \na u^2) \cdot u dx \right|
\le \|u\|_{L^4}^2 \|\na u^2\|_{L^2}
\le C\|u\|_{L^2} \|\na u\|_{L^2} \|\na u^2\|_{L^2}
\le \fr{1}{4} \|\na u\|_{L^2}^2 + C\|\na u^2\|_{L^2}^2\|u\|_{L^2}^2,
\ee 
by applying Ladyzhenskaya's interpolation inequality to the boundary vanishing velocity vector field $u$. By making use of the elliptic regularity obeyed by the potential $\Phi$, we have 
\be 
\left| \int_{\mathcal{O}} (\rho^1  \na \Phi) \cdot u dx \right|
\le \|\rho^1\|_{L^2} \|\na \Phi\|_{L^{\infty}} \|u\|_{L^2} 
\le \sum\limits_{i=1}^{N} \fr{D_i}{8} \|\na c_i\|_{L^2}^2 + C\|\rho^1\|_{L^2}^2 \|u\|_{L^2}^2.
\ee 
The Poincar\'e inequality applied to $u$ yields the bound
\be
\left| \int_{\mathcal{O}} (\rho \na \Phi^2) \cdot u dx \right|
\le \|\rho\|_{L^2} \|\na \Phi^2\|_{L^{\infty}} \|u\|_{L^2}
\le \fr{1}{8} \|\na u\|_{L^2}^2 + C\|\na \Phi^2\|_{L^{\infty}}^2 \sum\limits_{i=1}^{N} \|c_i\|_{L^2}^2.
\ee 
Due to the divergence-free property and homogeneous Dirichlet boundary conditions satisfied by $u$, we have
\be 
\beg{aligned}
\left|\int_{\mathcal{O}} u \cdot \na c_i^2 c_i dx \right|
& = \left|\int_{\mathcal{O}} u \cdot \na c_i c_i^2 dx \right| \le \|u\|_{L^4} \|\na c_i\|_{L^2} \|c_i^2\|_{L^4}
\le C\|u\|_{L^2}^{\fr{1}{2}}\|\na u\|_{L^2}^{\fr{1}{2}} \|\na c_i\|_{L^2} \|c_i^2\|_{L^4}
\\&\le \fr{1}{8N} \|\na u\|_{L^2}^2 + \fr{D_i}{8} \|\na c_i\|_{L^2}^2 + C\|c_i^2\|_{L^4}^4 \|u\|_{L^2}^2,
\end{aligned}
\ee 
after integration by parts and interpolation. We bound 
\be 
\beg{aligned}
\left| D_iz_i \int_{\mathcal{O}} c_i^1 \na \Phi \cdot \na c_i dx \right|
&\le C\|c_i^1\|_{L^2} \|\na \Phi\|_{L^{\infty}} \|\na c_i\|_{L^2}
\le C\|c_i^1\|_{L^2} \|\rho\|_{L^4} \|\na c_i\|_{L^2} 
\\&\le C\|c_i^1\|_{L^2} \left(\|\rho\|_{L^2} + \|\rho\|_{L^2}^{\fr{1}{2}} \|\na \rho\|_{L^2}^{\fr{1}{2}}\right) \|\na c_i\|_{L^2}
\\&\le \sum\limits_{j=1}^{N} \fr{D_j}{8N} \|\na c_j\|_{L^2}^2 
+ C\left( \|c_i^1\|_{L^2}^2 + \|c_i^1\|_{L^2}^4\right) \sum\limits_{j=1}^{N} \|c_j\|_{L^2}^2,
\end{aligned}
\ee
by using elliptic regularity estimates and interpolation inequalities again. Finally, a straightforward application of H\"older's and Young's inequalities gives
\be \la{con12}
\left| D_iz_i \int_{\mathcal{O}} c_i \na \Phi^2 \cdot \na c_i dx \right|
\le \fr{D_i}{8} \|\na c_i\|_{L^2}^2 
+ C\|\na \Phi^2\|_{L^{\infty}}^2 \|c_i\|_{L^2}^2.
\ee 
Summing the equations \eqref{con10} over all indices $i \in \left\{1, \dots, N \right\}$, adding the resulting energy equalities to the evolution equation \eqref{con9} obeyed by $u$, and using the estimates \eqref{con11}--\eqref{con12}, we end up with the differential inequality 
\be 
\fr{d}{dt} \left[\|u\|_{L^2}^2 + \sum\limits_{i=1}^{N} \|c_i\|_{L^2}^2 \right]
\le \kappa(t) \left[\|u\|_{L^2}^2 + \sum\limits_{i=1}^{N} \|c_i\|_{L^2}^2 \right],
\ee 
where $\kappa$ is given by \eqref{con14}. We then apply Gronwall's inequality and obtain the desired Lipschitz estimate \eqref{con13}, completing the proof of Proposition \ref{cont}.

\end{proof}

\section{Unique Ergodicity in the Case of Equal Diffusivities and Dirichlet Boundary Conditions} \label{erg}

In this section, we investigate the existence, uniqueness, and smoothness of invariant measures for the Markov transition kernels associated with the S-NPNS system under the assumptions that all ionic species have equal diffusivities and their concentrations have Dirichlet boundary conditions, that is 
\be \la{dcon}
D_1 = D_2 = \dots = D_N = D,
\ee and 
\be \la{upgradedcond1}
c_i|_{\pa \mathcal{O}} = \gamma_i
\ee for all indices $i \in \left\{1, \dots, N\right\}.$ In this setting, and under the following condition on the boundary data (which is equivalent to \eqref{44})
\be \la{upgradedcond}
\sum\limits_{i=1}^{N} z_i \gamma_i = 0,
\ee we can track the evolution of the charge density $\rho$ when coupled with the Navier-Stokes equation. 

\begin{defi}
  Let $Pr(\tilde{\mathcal H})$ be the set of Borel probability measures on $\tilde{\mathcal H}$. An element $\mu\in Pr(\tilde{\mathcal H})$ is called an invariant measure for the Feller Markov semigroup $\tilde{P}_t$ associated to the problem \eqref{sys:SNPNS} with boundary conditions \eqref{3} if 
    \begin{equation}
        \int_{\tilde{\mathcal H}} \phi(\omega_0) d\mu(\omega_0) = \int_{\tilde{\mathcal H}} \tilde{P}_t \phi(\omega_0) d\mu(\omega_0)
    \end{equation}
    for all $t\geq 0$ and any $\phi\in  C_b(\tilde{\mathcal H})$. If $\mu$ is an extremal point of the set containing all such invariant measures, then $\mu$ is said to be an ergodic invariant measure.
\end{defi}

In this section, we will work in two different settings.
{
\begin{setting}\label{setting1}
    Consider $N$ ionic species satisfying \eqref{dcon}, \eqref{upgradedcond1}, and \eqref{upgradedcond}.
    For each $i$-th ionic species with valence $z_i$, there exists a $j$-th ionic species with valence $z_j= -z_i$. Moreover, $\gamma_1, \dots, \gamma_N$ are sufficiently small, and either $g$ and $f$ are small enough or $D$ is large enough. 
\end{setting}

\begin{setting}\label{setting2}
    Consider $N$ ionic species satisfying \eqref{dcon}, \eqref{upgradedcond1}, and \eqref{upgradedcond}. The absolute value of all valences are equal to each other ($|z_i|=z>0$ for any $i \in \left\{1, \dots, N\right\}$).
\end{setting}
}

\subsection{Existence of Ergodic Invariant Measures}
{
Below is the main theorem of this section which concerns the existence of ergodic invariant measures for the S-NPNS model. Its proof depends on several propositions that will be stated and proved later.
}
\beg{Thm} \label{kbprocedure} 
Let $f \in H$ and $g \in H$ be time-independent. Suppose that either Setting \ref{setting1} or Setting \ref{setting2} holds.
Then there exists an ergodic invariant probability measure $\mu$ for the Markov semigroup \eqref{markovsemigroup} associated with the S-NPNS problem \eqref{sys:SNPNS} with boundary conditions 
\be 
u|_{\pa \mathcal{O}} = 0, \quad \Phi|_{\pa \mathcal{O}} = \gamma, \quad (c_1,\dots,c_N)|_{\pa \mathcal{O}}  = (\gamma_1, \dots,  \gamma_N).
\ee  In other words, 
\be 
\int_{\tilde{\mathcal{H}}} \phi(\omega) d\mu(\omega) 
= \int_{\tilde{\mathcal{H}}} \tilde{P}_t \phi(\omega) d\mu(\omega)
\ee for any $\phi \in C_b(\tilde{\mathcal{H}})$, where $\tilde{\mathcal{H}}$ is the space of vectors $(v, \xi_1, \dots, \xi_N) \in H \times L^2 \times \dots \times L^2$ such that $v|_{\pa \mathcal{O}} = 0$, $\xi|_{\pa \mathcal{O}} = \gamma_i$ for $i =1, \dots, N$, and  $\xi_1, \dots, \xi_N$ are nonnegative almost everywhere. 
\end{Thm}

\begin{proof} The proof is divided into two main steps.

\smallskip
\noindent{\bf{Step 1. The set of invariant measures $\mathcal{I}$ is nonempty.}}
   For each $T>0$, we define the time-average probability measure
   \be 
\mu_T (\cdot) = \fr{1}{T} \int_0^{T} P_t(\omega_0, \cdot) dt
   \ee where $P_t$ are the transition kernels defined by \eqref{transitionkernels}. The family $\left\{\mu_T \right\}_{T>0}$ is tight. Indeed, if $R>0$ and $B_R$ is a ball of radius $R$ in $\mathcal{D}(A^{\fr{1}{2}}) \times H^1 \times \dots \times H^1$ (which is compact in $H \times L^2 \times \dots \times L^2$). By the Chebyshev inequality and the moment bounds derived in Propositions \ref{prop7}, \ref{prop8} and \ref{prop9}, we have
   \be 
   \beg{aligned}
&\sup\limits_{T>0} \mu_T(B_R^c) 
= \sup\limits_{T>0}  \fr{1}{T} \int_{0}^{T} \PP(\|\omega(t, (0, \gamma_1, \dots, \gamma_N))\|_{H^1} \ge R)
\\&\quad\quad\le \fr{1}{\log (1+R^2)} \sup\limits_{T>0} \fr{1}{T} \int_{0}^{T} \E \log (1 + \|\omega(t, (0, \gamma_1, \dots, \gamma_N))\|_{H^1}^2) dt
\\&\quad\quad\le \fr{1}{\log (1+R^2)} \sup\limits_{T>0} \fr{1}{T} \int_{0}^{T} \E \log (1 + \|u (t, (0, \gamma_1, \dots, \gamma_N))\|_{H^1}^2 + \sum\limits_{i=1}^N \|c_i(t, (0, \gamma_1, \dots, \gamma_N))\|_{H^1}^2) dt
\\&\quad\quad\le \fr{C}{\log (1+R^2)} 
\end{aligned}
   \ee for some constant $C$ depending only on the parameters of the problem and the forcing terms $f$ and $g$. 
   Letting $R \rightarrow \infty$ yields the tightness of the family  $\left\{\mu_T \right\}_{T>0}$. By Prokhorov's theorem and the Krylov Bogoliubov averaging procedure, we deduce that $\left\{\mu_T \right\}_{T>0}$ has a subsequence that converges to an invariant measure $\mu$ for $\tilde{P}_t$. 
   
\smallskip
\noindent{\bf{Step 2. Existence of an ergodic invariant measure.}} In view of Step 1, $\mathcal{I}$ is nonempty. From the definition of $\tilde{P}_t$, we deduce that $\mathcal{I}$ is convex. Since $\tilde{P}_t$ is Feller, it follows that $\mathcal{I}$ is closed. Finally, the moment bounds derived in Propositions \ref{prop7}, \ref{prop8}, and \ref{prop9} allow us to conclude that $\mathcal{I}$ is tight and thus compact. By the Krein-Millman theorem, $\mathcal{I}$ has an extreme point which turns out to be an ergodic invariant measure. We refer the reader to \cite{da2014stochastic} for a more detailed elaboration of this argument. 
\end{proof}

{We note that Theorem \ref{kbprocedure} holds in two distinct settings, Setting \ref{setting1} and Setting \ref{setting2}. Below, we present two sequences of propositions that provide different tools and ingredients needed for each setting separately.
}

\subsubsection{Moment bounds under Setting \ref{setting1}} {In this subsection, we address two propositions essential for the establishment of Theorem \ref{kbprocedure} under Setting \ref{setting1}. The second proposition is based on several auxiliary lemmas and results.}

\beg{prop} \label{prop7} Let $u_0 \in H$ and $c_i(0) \in L^2$ for all $i \in \left\{1, \dots, N\right\}$, and suppose the ionic concentrations are nonnegative. Under conditions \eqref{dcon}, \eqref{upgradedcond1}, and \eqref{upgradedcond}, it holds that, for all $t \ge 0$,
\be \la{ito1}
\beg{aligned}
&\E \left[\|u(t)\|_{L^2}^2 + \|\na \Phi(t)\|_{L^2}^2 \right]
+ \E \int_{0}^{t} \left[\|\na u(s)\|_{L^2}^2 + D\|\rho(s)\|_{L^2}^2 + D\sum\limits_{i=1}^{N} \|z_i\sqrt{c_i} \na \Phi (s)\|_{L^2}^2 \right]ds
\\&\quad\quad\quad\quad\le \|u_0\|_{L^2}^2 + \|\na \Phi_0\|_{L^2}^2 + \left(C\|f\|_{L^2}^2 + \|g\|_{L^2}^2 \right)t, 
\end{aligned}
\ee  and 
\be \la{ito2}
\beg{aligned}
&\E \left[\|u(t)\|_{L^2}^2 + \|\na \Phi(t)\|_{L^2}^2 \right]^2
\\&\quad\quad+ \E \int_{0}^{t} \left(\|u(s)\|_{L^2}^2 + \|\na \Phi(s)\|_{L^2}^2 \ \right) \left(\|\na u(s)\|_{L^2}^2 + D\|\rho(s)\|_{L^2}^2 + D\sum\limits_{i=1}^{N} \|z_i\sqrt{c_i} \na \Phi (s)\|_{L^2}^2 \right)ds
\\&\quad\quad\quad\quad\quad\quad\le  \left[\|u_0\|_{L^2}^2 + \|\na \Phi_0\|_{L^2}^2 \right]^2 
+C\left(\|g\|_{L^2}^4 + \|f\|_{L^2}^4 \right)t.
\end{aligned}
\ee 
\end{prop}

\begin{proof}
    Multiplying the ionic concentration equations by $z_i$ gives
    \be 
\pa_t(z_ic_i) + u \cdot \na(z_ic_i) - D\Delta (z_ic_i)
= Dz_i^2 \na \cdot (c_i \na \Phi),
    \ee 
for all $i \in \left\{1, \dots, N\right\}$. Summing over all indices $i \in \left\{1, \dots, N\right\}$ yields the equation 
\be \la{densityevolution}
\pa_t \rho + u \cdot \na \rho - D\Delta \rho 
= D \sum\limits_{i=1}^{N} \na \cdot (z_i^2 c_i \na \Phi),
\ee 
which describes the time evolution of the charge density $\rho$. Note that, here, the diffusion term $-D\Delta \rho$ shows up as a consequence of all ionic species having equal diffusivities. Multiplying \eqref{densityevolution} by $\Phi - \gamma$,  integrating in space over $\mathcal{O}$, integrating by parts and using the Dirichlet boundary data obeyed by $c_i$ and the vanishing condition \eqref{upgradedcond}, we have
\be 
    \beg{aligned}
-D \int_{\mathcal{O}} \Delta \rho (\Phi - \gamma) dx
&= D\int_{\mathcal{O}} \na \rho \cdot \na \Phi dx
= D\int_{\mathcal{O}} \na \left(\rho - \sum\limits_{i=1}^{N} z_i \gamma_i\right) \cdot \na \Phi dx
\\
&= -D\int_{\mathcal{O}} \left(\rho - \sum\limits_{i=1}^{N} z_i \gamma_i\right) \cdot \Delta \Phi dx
= D\int_{\mathcal{O}} \rho^2 dx.
\end{aligned}
\ee
Due to the divergence-free property of the velocity $u$, it follows that
    \be 
\int_{\mathcal{O}} (u \cdot \na \rho) (\Phi - \gamma) dx 
= - \int_{\mathcal{O}} (u \cdot \na \Phi) \rho dx.
    \ee 
Another integration by parts allows us to deduce the relation
    \be 
D \sum\limits_{i=1}^{N} \int_{\mathcal{O}} \na \cdot (z_i^2 c_i \na \Phi) (\Phi - \gamma) dx
= - D \sum\limits_{i=1}^{N} \int_{\mathcal{O}} z_i^2 c_i \na \Phi \cdot \na \Phi dx 
= - D \sum\limits_{i=1}^{N} \|z_i \sqrt{c_i} \na \Phi\|_{L^2}^2,
    \ee 
where the nonnegativity of the ionic concentrations is used. Thus, the following deterministic energy equality
    \be \la{energy1}
\fr{d}{dt} \|\na \Phi\|_{L^2}^2 + 2D\|\rho\|_{L^2}^2 + 2D\sum\limits_{i=1}^{N} \|z_i \sqrt{c_i} \na \Phi\|_{L^2}^2
= 2 \int_{\mathcal{O}} \rho \na \Phi \cdot u dx
    \ee 
holds. As for the stochastic evolution of the velocity $u$, we apply It\^o's lemma and obtain 
    \be \la{energy2}
d \|u\|_{L^2}^2 
+ 2\|\na u\|_{L^2}^2 dt
= -2(\rho \na \Phi, u)_{L^2} dt
+ 2(f,u)_{L^2} dt
+ \|g\|_{L^2}^2 dt
+ 2(g,u)_{L^2} dW.
    \ee 
    Combining \eqref{energy1} and \eqref{energy2} together, we observe that the nonlinear terms cancel each other, which results in 
    \be \la{ito3}
    \beg{aligned}
&d \left[\|u\|_{L^2}^2 + \|\na \Phi\|_{L^2}^2 \right]
+ 2\left[\|\na u\|_{L^2}^2 + D\|\rho\|_{L^2}^2 + D\sum\limits_{i=1}^{N} \|z_i \sqrt{c_i} \na \Phi\|_{L^2}^2 \right] dt
\\&\quad\quad\quad\quad= \|g\|_{L^2}^2 dt
+ 2(f,u)_{L^2}dt
+ 2(g,u)_{L^2} dW.
\end{aligned}
 \ee
We control the forcing term $(f,u)_{L^2}$ as follows,
\be 
2(f,u)_{L^2}
\le 2\|f\|_{L^2} \|u\|_{L^2}
\le C\|f\|_{L^2} \|\na u\|_{L^2}
\le \|\na u\|_{L^2}^2 + C\|f\|_{L^2}^2,
\ee 
which leads to the stochastic inequality
\be \la{expm1}
\beg{aligned}
&d \left[\|u\|_{L^2}^2 + \|\na \Phi\|_{L^2}^2 \right]
+ \left[\|\na u\|_{L^2}^2 + D\|\rho\|_{L^2}^2 + D\sum\limits_{i=1}^{N} \|z_i \sqrt{c_i} \na \Phi\|_{L^2}^2 \right] dt
\\&\quad\quad\quad\quad
\le \|g\|_{L^2}^2 dt
+ C\|f\|_{L^2}^2 dt
+ 2(g,u)_{L^2} dW.
\end{aligned}
 \ee
We integrate in time, take expectations on both sides, and obtain \eqref{ito1}.

We now proceed to prove \eqref{ito2}. We define the energies
\be \la{ED}
\mathcal{E} = \|u\|_{L^2}^2 + \|\na \Phi\|_{L^2}^2,
\quad \text{and} \quad 
\mathcal{D} = \|\na u\|_{L^2}^2 + D\|\rho\|_{L^2}^2 + D\sum\limits_{i=1}^{N} \|z_i \sqrt{c_i} \na \Phi\|_{L^2}^2,
\ee 
and rewrite \eqref{ito3} as
\be 
d \mathcal{E} 
+ 2\mathcal{D} dt
= \|g\|_{L^2}^2 dt
+ 2(f,u)_{L^2} dt
+ 2(g,u)_{L^2} dW.
\ee Applying It\^o's lemma to the stochastic process $X = \mathcal{E}^2$ gives
\be 
d \mathcal{E}^2 
+ 4\mathcal{E}\mathcal{D} dt
= 2\mathcal{E} \|g\|_{L^2}^2 dt
+ 2\mathcal{E} (f,u)_{L^2} dt
+ 4(g,u)_{L^2}^2 dt
+ 4 \mathcal{E} (g,u)_{L^2} dW.
\ee 
We then estimate 
\be 
2\mathcal{E} \|g\|_{L^2}^2
= 2\mathcal{E}^{\fr{1}{2}} \mathcal{E}^{\fr{1}{2}} \|g\|_{L^2}^2
\le C\mathcal{E}^{\fr{1}{2}} \mathcal{D}^{\fr{1}{2}} \|g\|_{L^2}^2
\le \mathcal{E} \mathcal{D} + C\|g\|_{L^2}^4,
\ee 
\be 
|2\mathcal{E} (f,u)_{L^2}|
\le 2\mathcal{E} \|u\|_{L^2}\|f\|_{L^2}
\le C\mathcal{E}^{\fr{3}{2}} \|f\|_{L^2}
\le  C\mathcal{E}^{\fr{3}{4}} \mathcal{D}^{\fr{3}{4}}\|f\|_{L^2}
\le \mathcal{E} \mathcal{D} + C\|f\|_{L^2}^4,
\ee
and
\be 
4(g,u)_{L^2}^2
\le 4\|g\|_{L^2}^2\|u\|_{L^2}^2
\le 4\|g\|_{L^2}^2 \mathcal{E}^{\fr{1}{2}} \mathcal{D}^{\fr{1}{2}}
\le \mathcal{E}\mathcal{D} + C\|g\|_{L^2}^4,
\ee 
where we have used the Poincar\'e inequality $\mathcal{E} \le C\mathcal{D}$ due to the vanishing of $u$ on the boundary and the elliptic regularity estimate $\|\na (\Phi-\gamma)\|_{L^2} \le C\|\rho\|_{L^2}$. Consequently, we obtain the stochastic differential inequality
\be 
d \mathcal{E}^2 
+ \mathcal{E}\mathcal{D} dt
\le C\|g\|_{L^2}^4 dt
+ C\|f\|_{L^2}^4  dt
+ 4 \mathcal{E} (g,u)_{L^2} dW,
\ee from which we deduce \eqref{ito2} after integrating in time from $0$ to $t$ and applying the expectation $\E$.
\end{proof}

\begin{prop}\label{prop8} Let $u_0 \in H$ and $c_i(0) \in L^2$ for all $i \in \left\{1, \dots, N\right\}$, and suppose the ionic concentrations are nonnegative. Under Setting \ref{setting1} it holds that
\be 
\beg{aligned}
\E \int_{0}^{T} \log (1 + \|\na c_i\|_{L^2}^2) ds
\le R_1(\|u_0\|_{L^2}, \|c_i(0) - \gamma_i\|_{L^2}, g)  + R_2(g, f)T,
\end{aligned}
\ee for $T \ge 0$, where $R_1$ is a positive constant depending only on  $\|u_0\|_{L^2}, \|c_i(0) - \gamma_i\|_{L^2}, g$, the parameters of the problems, and some universal constants, with the property that $R_1 = 0$ when $u_0  = 0$ and $c_i(0) = \gamma_i$, and $R_2$ is a positive constant depending only on $f, g$, the parameters of the problem and some universal constants. 
\end{prop}

\begin{proof} The proof follows from the calculation
\be 
\beg{aligned}
\E \int_{0}^{T} \log (1 + \|\na c_i\|_{L^2}^2) ds
&= \E \int_{0}^{T} \log \left(\fr{1 + \|\na c_i(s)\|_{L^2}^2}{1 + \|c_i (s)- \gamma_i\|_{L^2}^2}\right) ds 
+ \E\int_{0}^{T} \log \left(1  + \|c_i(s) - \gamma_i\|_{L^2}^2 \right) ds
\\&\le \E\int_{0}^{T} \fr{\|\na c_i(s)\|_{L^2}^2}{1 + \|c_i(s) - \gamma_i\|_{L^2}^2} ds + \E \int_{0}^{T} \log \left(1  + \|c_i(s) - \gamma_i\|_{L^2}^2 \right) ds,
\end{aligned}
\ee 
and Lemma \ref{momentbound1}, Corollary \ref{firstcor}, below.
\end{proof}

\beg{lem} \la{momentbound1}
Let $u_0 \in H$ and $c_i(0) \in L^2$ for all $i \in \left\{1, \dots, N\right\}$, and suppose the ionic concentrations are nonnegative. Under conditions \eqref{dcon}-- \eqref{upgradedcond}, it holds that 
\be \la{logcon}
\beg{aligned}
&\E \log \left(1 + \|c_i(t) - \gamma_i\|_{L^2}^2 \right)
+ D\E \int_{0}^{t} \fr{\|\na c_i(s)\|_{L^2}^2}{1+ \|c_i(s) - \gamma_i\|_{L^2}^2} ds
\\&\quad\quad\le \log \left(1 + \|c_i(0) - \gamma_i\|_{L^2}^2 \right)
+ CD|z_i|^2 \mathcal{E}(0)\left(|z_i|^2 \mathcal{E}(0)+ \gamma_i^2 \right)
\\&\quad\quad\quad\quad\quad\quad+ CD|z_i|^2 {\left(\|g\|_{L^2}^2 + \|f\|_{L^2}^2\right)}\left(|z_i|^2 {\left(\|g\|_{L^2}^2 + \|f\|_{L^2}^2\right)} + \gamma_i^2 \right)t
\end{aligned}
\ee for all $t \ge 0$ and $i \in \left\{1, \dots, N\right\}.$ Here,  
$
\mathcal{E}(0) = \|u_0\|_{L^2}^2 + \|\na \Phi_0\|_{L^2}^2.
$
\end{lem}

\begin{proof}
    We multiply the ionic concentration equations by $c_i - \gamma_i$, integrate over $\mathcal{O}$, and obtain the energy equality
    \be 
\fr{1}{2} \fr{d}{dt} \|c_i - \gamma_i\|_{L^2}^2 
+ D\|\na c_i\|_{L^2}^2 
= - Dz_i \int_{\mathcal{O}} (c_i - \gamma_i) \na \Phi \cdot \na c_i
- Dz_i \gamma_i \int_{\mathcal{O}} \na \Phi \cdot \na c_i.
\ee 
 By the Ladyzhenskaya's interpolation inequality, the H\"older inequality, and Young's inequality with exponents $4$ and $4/3$, we estimate
\be 
\beg{aligned}
&\left|Dz_i \int_{\mathcal{O}} (c_i - \gamma_i) \na \Phi \cdot \na c_i \right|
\le D|z_i| \|c_i - \gamma_i\|_{L^4} \|\na \Phi\|_{L^4} \|\na c_i\|_{L^2}
\\&\quad\quad\le CD|z_i| \|c_i - \gamma_i\|_{L^2}^{\fr{1}{2}} \|\na c_i\|_{L^2}^{\fr{1}{2}} \|\na \Phi\|_{L^2}^{\fr{1}{2}}\|\rho\|_{L^2}^{\fr{1}{2}} \|\na c_i\|_{L^2}
\le \fr{D}{4} \|\na c_i\|_{L^2}^2
+ CD|z_i|^4 \|\na \Phi\|_{L^2}^2 \|\rho\|_{L^2}^2\|c_i - \gamma_i\|_{L^2}^2.
\end{aligned}
\ee Due to the Cauchy-Schwarz inequality and the elliptic regularity satisfied by the electric potential $\Phi$, it follows that
\be 
\left| Dz_i \gamma_i \int_{\mathcal{O}} \na \Phi \cdot \na c_i\right|
\le D|z_i||\gamma_i| \|\na \Phi\|_{L^2} \|\na c_i\|_{L^2}
\le CD|z_i| \gamma_i \|\rho\|_{L^2} \|\na c_i\|_{L^2}
\le \fr{D}{4} \|\na c_i\|_{L^2}^2
+ CD|z_i|^2 \gamma_i^2 \|\rho\|_{L^2}^2.
\ee 
This yields the energy inequality
\be\label{prop5eq1}
\fr{d}{dt} \|c_i - \gamma_i\|_{L^2}^2
+ D\|\na c_i\|_{L^2}^2
\le CD|z_i|^4 \|\na \Phi\|_{L^2}^2 \|\rho\|_{L^2}^2 \|c_i - \gamma_i\|_{L^2}^2 
+ CD|z_i|^2 \gamma_i^2 \|\rho\|_{L^2}^2.
\ee Letting 
$
\mathcal{X} = \log \left(1 + \|c_i - \gamma_i\|_{L^2}^2 \right),
$
we have 
\be 
\fr{d}{dt} \mathcal{X}
+ \fr{D\|\na c_i\|_{L^2}^2}{1 + \|c_i - \gamma_i\|_{L^2}^2}
\le \fr{CD|z_i|^4 \|\na \Phi\|_{L^2}^2 \|\rho\|_{L^2}^2 \|c_i - \gamma_i\|_{L^2}^2 }{1 + \|c_i - \gamma_i\|_{L^2}^2} + \fr{CD|z_i|^2 \gamma_i^2 \|\rho\|_{L^2}^2}{1 + \|c_i - \gamma_i\|_{L^2}^2},
\ee and thus
\be 
\fr{d}{dt} \mathcal{X}
+ \fr{D\|\na c_i\|_{L^2}^2}{1 + \|c_i - \gamma_i\|_{L^2}^2}
\le CD|z_i|^4 \|\na \Phi\|_{L^2}^2 \|\rho\|_{L^2}^2 + CD|z_i|^2 \gamma_i^2 \|\rho\|_{L^2}^2.
\ee Finally, we integrate in time from $0$ to $t$, take expectations, use the bounds \eqref{ito1} and \eqref{ito2}, and obtain \eqref{logcon}.
\end{proof}

\beg{lem} Let $\eta \in \left(0, \fr{1}{4\|(-\Delta)^{-\fr{1}{2}} g\|_{L^2}^2}\right)$.  Let $u_0 \in H$ and $c_i(0) \in L^2$ for all $i \in \left\{1, \dots, N\right\}$, and suppose the ionic concentrations are nonnegative. Under conditions \eqref{dcon}--\eqref{upgradedcond}, there exists a positive universal constant $C_1$ such that  
\be \la{expm2}
\E \exp \left\{\fr{\eta D}{2} \int_{0}^{t} \|\rho(s)\|_{L^2}^2 ds \right\}
\le \exp \left\{C_1 \eta \left(\|u_0\|_{L^2}^2 + \|\na \Phi_0\|_{L^2}^2 + \|g\|_{L^2}^2 t + \|f\|_{L^2}^2t \right) \right\}
\ee holds for all $t \ge 0$.
\end{lem}

\begin{proof}
From \eqref{expm1}, we have 
\be 
d \mathcal{E} 
+ \mathcal{D} dt
\le \|g\|_{L^2}^2 dt
+ C\|f\|_{L^2}^2 dt
+ 2(g,u)_{L^2} dW
\ee where $\mathcal{E}$ and $\mathcal{D}$ are given by \eqref{ED}. We integrate the above inequality in time from $0$ to $t$, multiply by $\eta$, and obtain 
\be 
\eta \mathcal{E}(t) 
+ \fr{\eta}{2} \int_{0}^{t} \mathcal{D}(s) ds
\le \eta \mathcal{E}(0) +\eta \|g\|_{L^2}^2t + C\eta \|f\|_{L^2}^2 t
+ 2\eta \int_{0}^{t} (g,u)_{L^2} dW
- \fr{\eta}{2} \int_{0}^{t} \mathcal{D}(s) ds,
\ee which yields
\be 
\E \exp \left\{\fr{\eta D}{2} \int_{0}^{t} \|\rho(s)\|_{L^2} ds \right\}
\le \exp \left\{\eta \mathcal{E}(0) +\eta \|g\|_{L^2}^2t + C\eta \|f\|_{L^2}^2 t \right\} \E \exp \left\{ 2\eta \int_{0}^{t} (g,u)_{L^2} dW
- \fr{\eta}{2} \int_{0}^{t} \mathcal{D}(s) ds \right\}.
\ee 
In view of the exponential martingale identity \cite{schilling2014brownian}
\be 
\E \exp \left\{  \int_{0}^{t} 2\eta (g,u)_{L^2} dW
- \fr{1}{2} \int_{0}^{t} 4\eta^2 (g,u)_{L^2}^2 ds \right\} = 1,
\ee 
 and the estimate
\be 
2\eta^2 (g,u)_{L^2}^2 
\le 2\eta^2 \|(-\Delta)^{-\fr{1}{2}} g\|_{L^2}^2 \|\na u\|_{L^2}^2
\le 2\eta^2  \|(-\Delta)^{-\fr{1}{2}} g\|_{L^2}^2  \mathcal{D},
\ee we obtain \eqref{expm2}, provided that
$
2\eta^2 \|(-\Delta)^{-\fr{1}{2}} g\|_{L^2}^2 \le \fr{\eta}{2},
$ 
which is equivalent to 
$
0 < \eta \le \fr{1}{4 \|(-\Delta)^{-\fr{1}{2}} g\|_{L^2}^2}.
$ 
\end{proof}

\beg{lem} \la{valpair} Let $u_0 \in H$ and $c_i(0) \in L^2$ for all $i \in \left\{1, \dots, N\right\}$, and suppose the ionic concentrations are nonnegative. Furthermore, suppose that the valences of the ionic species obey $z_1 = -z_2 = z_3 = -z_4 = \dots = \pm z_N$. We denote their absolute value by $|z|$. Under conditions \eqref{dcon}--\eqref{upgradedcond}, there exist positive universal constants $c$ and $C_2$ such that 
\be \la{expm8}
\beg{aligned}
\|c_i(t) - \gamma_i\|_{L^2}^2
 \le \Gamma_0 \exp \left\{-D(c - C_2z^2\mathcal{M}) t\right\} \exp \left\{C_2
D|z|^2 \int_{0}^{t} \|\rho(s)\|_{L^2}^2 ds \right\} 
\end{aligned}, 
\ee for all $i \in \left\{1, \dots, N\right\}$ and $t \ge 0$, where
\be \la{inda}
\Gamma_0 = \sum\limits_{j=1}^{N-1} \left(\|c_j(0) - \gamma_j\|_{L^2}^2 + \|c_{j+1}(0) - \gamma_{j+1}\|_{L^2}^2 \right),
\ee and $\mathcal{M}$ is a positive constant depending only on the maximum value of the boundary data $\gamma_1, \dots, \gamma_N$.
\end{lem}

\begin{proof}
    For $j \in \left\{1, \dots, N\right\}$, we define
    \be 
\rho_j = c_{j+1} - c_j, \quad  \tilde{\rho}_j = \gamma_{j+1} - \gamma_j \quad \text{and}
\quad
\sigma_{j} = c_{j+1} + c_j, \quad \tilde{\sigma}_j = \gamma_{j+1} + \gamma_j.
    \ee 
The difference $\rho_j - \tilde{\rho}_j$ and $\sigma_j - \tilde{\sigma}_j$ obey the nonlinear nonlocal equations
    \be \la{dif1}
\pa_t (\rho_j - \tilde{\rho}_j) + u \cdot \na \rho_j - D\Delta \rho_j  = Dz_{j+1} \na \cdot (\sigma_j \na \Phi),
    \ee 
    \be \la{dif2}
\pa_t (\sigma_j - \tilde{\sigma}_j) + u \cdot \na \sigma_j - D\Delta \sigma_j
= Dz_{j+1} \na \cdot (\rho_j \na \Phi).
    \ee 
    We take the scalar product in $L^2$ of \eqref{dif1} and \eqref{dif2} with $\rho_j - \tilde{\rho}_j$ and $\sigma_j - \tilde{\sigma}_j$ respectively, add the resulting energy equalities, and obtain 
    \be 
    \beg{aligned}
&\fr{1}{2} \fr{d}{dt} \left(\|\rho_j - \tilde{\rho}_j\|_{L^2}^2 + \|\sigma_j - \tilde{\sigma}_j\|_{L^2}^2 \right)
+ D \left(\|\na \rho_j\|_{L^2}^2 + \|\na \sigma_j\|_{L^2}^2
\right) 
\\&\quad\quad= - Dz_{j+1} \int_{\mathcal{O}} \sigma_j \na \Phi \cdot \na (\rho_j - \tilde{\rho}_j) dx
+ Dz_{j+1} \int_{\mathcal{O}} \na \cdot (\rho_j \na \Phi) (\sigma_j - \tilde{\sigma}_j) dx
\\&\quad\quad=  - Dz_{j+1} \int_{\mathcal{O}} \sigma_j \na \Phi \cdot \na (\rho_j - \tilde{\rho}_j) dx
\\&\quad\quad\quad\quad+ Dz_{j+1} \int_{\mathcal{O}} (\na \rho_j \cdot \na \Phi)(\sigma_j - \tilde{\sigma}_j) dx
+ Dz_{j+1} \int_{\mathcal{O}} (\rho_j \Delta \Phi) (\sigma_j - \tilde{\sigma}_j) dx
\\&\quad\quad=  - Dz_{j+1} \int_{\mathcal{O}} \sigma_j \na \Phi \cdot \na \rho_j dx 
+ Dz_{j+1} \int_{\mathcal{O}} (\na \rho_j \cdot \na \Phi) \sigma_j dx
\\&\quad\quad\quad\quad- Dz_{j+1} \int_{\mathcal{O}} (\na \rho_j \cdot \na \Phi)\tilde{\sigma}_j dx
- Dz_{j+1} \int_{\mathcal{O}} \rho_j \rho (\sigma_j - \tilde{\sigma}_j) dx,
    \end{aligned}
    \ee which reduces to
    \be 
    \beg{aligned}
&\fr{1}{2} \fr{d}{dt} \left(\|\rho_j - \tilde{\rho}_j\|_{L^2}^2 + \|\sigma_j - \tilde{\sigma}_j\|_{L^2}^2 \right)
+ D \left(\|\na \rho_j\|_{L^2}^2 + \|\na \sigma_j\|_{L^2}^2
\right) 
\\&\quad\quad=- Dz_{j+1} \int_{\mathcal{O}} (\na \rho_j \cdot \na \Phi)\tilde{\sigma}_j dx
- Dz_{j+1} \int_{\mathcal{O}} (\rho_j - \tilde{\rho}_j) \rho (\sigma_j - \tilde{\sigma}_j) dx
- Dz_{j+1} \int_{\mathcal{O}} \tilde{\rho}_j \rho (\sigma_j - \tilde{\sigma}_j) dx.
    \end{aligned}
    \ee
We then compute the following bounds
\begin{align}
&\left|- Dz_{j+1} \int_{\mathcal{O}} (\na \rho_j \cdot \na \Phi)\tilde{\sigma}_j dx \right|
\le D|z||\tilde{\sigma}_j| \|\na \rho_j\|_{L^2} \|\na \Phi\|_{L^2}\\
&\quad \quad
\le CD|z||\tilde{\sigma}_j| \|\na \rho_j\|_{L^2} \|\rho \|_{L^2}
\le \fr{D}{8} \|\na \rho_j\|_{L^2}^2
+ CD|z|^2 \tilde{\sigma}_j^2 \|\rho\|_{L^2}^2,\\
&\left| - Dz_{j+1} \int_{\mathcal{O}} (\rho_j - \tilde{\rho}_j) \rho (\sigma_j - \tilde{\sigma}_j) dx \right|
\le D|z| \|\rho\|_{L^2} \|\rho_j - \tilde{\rho}_j\|_{L^4} \|\sigma_j - \tilde{\sigma}_j\|_{L^4}
\\&\quad\quad\le CD|z| \|\rho\|_{L^2} \|\rho_j - \tilde{\rho}_j\|_{L^2}^{\fr{1}{2}}\|\na \rho_j\|_{L^2}^{\fr{1}{2}} \|\sigma_j - \tilde{\sigma}_j\|_{L^2}^{\fr{1}{2}}\|\na\sigma_j\|_{L^2}^{\fr{1}{2}}
\\&\quad\quad\le \fr{D}{8} \left(\|\na \rho_j\|_{L^2}^2 + \|\na \sigma_j\|_{L^2}^2 \right) 
+ CD|z|^2 \|\rho\|_{L^2}^2 \left(\|\rho_j - \tilde{\rho}_j\|_{L^2}^2 +  \|\sigma_j - \tilde{\sigma}_j\|_{L^2}^2\right),\\
&\left|Dz_{j+1} \int_{\mathcal{O}} \tilde{\rho}_j \rho (\sigma_j - \tilde{\sigma}_j) dx \right|
\le D|z||\tilde{\rho}_j| \|\rho\|_{L^2} \|\sigma_j - \tilde{\sigma}_j\|_{L^2}
\le \fr{1}{2}D|z|^2 |\tilde{\rho}_j| \|\rho\|_{L^2}^2 
+ \fr{1}{2}D|z|^2 |\tilde{\rho}_j| \|\sigma_j - \tilde{\sigma}_j\|_{L^2}^2,
\end{align}
using elliptic regularity estimates, Ladyzshenskaya's interpolation inequality, and Young's inequality for products. Consequently, we deduce the evolution
\be 
\beg{aligned}
&\fr{d}{dt} \left(\|\rho_j - \tilde{\rho}_j\|_{L^2}^2 + \|\sigma_j - \tilde{\sigma}_j\|_{L^2}^2 \right)
+ D \left(\|\na \rho_j\|_{L^2}^2 + \|\na \sigma_j\|_{L^2}^2
\right) 
\\&\quad\quad\le CD|z|^2 \left(|\tilde{\rho}_j| + \tilde{\sigma}_j^2\right) \|\rho\|_{L^2}^2
+ CD|z|^2 |\tilde{\rho}_j| \|\sigma_j - \tilde{\sigma}_j\|_{L^2}^2
+ CD|z|^2 \|\rho\|_{L^2}^2 \left(\|\rho_j - \tilde{\rho}_j\|_{L^2}^2 +  \|\sigma_j - \tilde{\sigma}_j\|_{L^2}^2\right).
\end{aligned}
\ee In view of the parallelogram law
\be 
2\|a\|_{L^2}^2 + 2\|b\|_{L^2}^2 = \|a+b\|_{L^2}^2 + \|a-b\|_{L^2}^2
\ee applied to $a = \rho_j - \tilde{\rho}_j$ and $b = \sigma_j - \tilde{\sigma}_j$, and to $a = \nabla \rho_j$ and $b = \nabla \sigma_j$, we obtain
\be
\beg{aligned}
&\fr{d}{dt} \left(\|c_j - \gamma_j\|_{L^2}^2 + \|c_{j+1} - \gamma_{j+1}\|_{L^2}^2 \right)
+ D \left(\|\na c_j\|_{L^2}^2 + \|\na c_{j+1}\|_{L^2}^2
\right) \\
&\quad\quad\quad\quad
\le CDz^2 \left(|\tilde{\rho}_j| + \tilde{\sigma}_j^2\right) \|\rho\|_{L^2}^2
+ \left(CD|z|^2  \|\rho\|_{L^2}^2  + CD|z|^2 |\tilde{\rho}_j|\right) \left(\|c_j - \gamma_j\|_{L^2}^2 +  \|c_{j+1} - \gamma_{j+1}\|_{L^2}^2\right).
\end{aligned} 
\ee 
We sum over the indices $j \in \left\{1, \dots, N-1\right\}$ and deduce that
\be \la{expm5}
\beg{aligned}
&\fr{d}{dt} \sum\limits_{j=1}^{N-1} \left(\|c_j - \gamma_j\|_{L^2}^2 + \|c_{j+1} - \gamma_{j+1}\|_{L^2}^2 \right)
+ D \sum\limits_{j=1}^{N-1} \left(\|\na c_j\|_{L^2}^2 + \|\na c_{j+1}\|_{L^2}^2
\right) \\
&\quad\quad\quad\quad \le CDz^2\mathcal{M}  \|\rho\|_{L^2}^2
+ \left(CD|z|^2  \|\rho\|_{L^2}^2  + CD|z|^2 \mathcal{M} \right) \sum\limits_{j=1}^{N-1} \left(\|c_j - \gamma_j\|_{L^2}^2 +  \|c_{j+1} - \gamma_{j+1}\|_{L^2}^2\right),
\end{aligned} 
\ee 
where
$\mathcal{M} = \max\limits_{1 \le j \le N-1} \left(|\tilde{\rho}_j| + \tilde{\sigma}_j^2 \right).$
The charge density $\rho$ obeys 
\be 
\|\rho\|_{L^2}^2 = \left\| \sum\limits_{i=1}^{N} z_i (c_i - \gamma_i) \right\|_{L^2}^2 
\le C |z|^2  \sum\limits_{i=1}^{N} \|c_i - \gamma_i\|_{L^2}^2
\le C|z|^2  \sum\limits_{j=1}^{N-1} \left(\|c_j - \gamma_j\|_{L^2}^2 +  \|c_{j+1} - \gamma_{j+1}\|_{L^2}^2\right).
\ee This allows us to rewrite \eqref{expm5} as 
\be 
\fr{d}{dt} \mathcal{A}_N + D\left(c -  Cz^2\mathcal{M} -  Cz^2\|\rho\|_{L^2}^2 \right) \mathcal{A}_N \le 0,
\ee 
after making use of the Poincar\'e inequality, where
$\mathcal{A}_N = \sum\limits_{j=1}^{N-1} \left(\|c_j - \gamma_j\|_{L^2}^2 +  \|c_{j+1} - \gamma_{j+1}\|_{L^2}^2\right),$
and $c$ is the Poincar\'e constant. We multiply the above inequality by the integrating factor
$\exp \left\{D\left(c -  Cz^2\mathcal{M} -  Cz^2\|\rho\|_{L^2}^2 \right) t\right\},$
integrate in time from $0$ to $t$, and conclude that
\be 
\mathcal{A}_N
\le \mathcal{A}_N(0) \exp \left\{-D(c - Cz^2 \mathcal{M}) t\right\} \exp \left\{CD|z|^2 \int_{0}^{t} \|\rho(s)\|_{L^2}^2 ds \right\}.
\ee 
\end{proof}

\beg{rem} Lemma \ref{valpair} holds under different conditions imposed on the valences {and diffusivities}: 
\begin{enumerate}
\item If the number of ionic species $N$ is even, then the proof of \eqref{expm8} works under the valences pairing condition $z_1 = -z_2$, $z_3 = -z_4, \dots, z_{N-1} = - z_N$ {and diffusivities pairing condition $D_1 = D_2, \dots, D_{N-1} = D_N$.} In this latter case, one can define 
\be \la{even12}
\rho_j = c_{2j} - c_{2j-1}, \quad \tilde{\rho}_j = \gamma_{2j} - \gamma_{2j-1},
\quad\text{and}\quad
\sigma_{j} = c_{2j} + c_{2j-1}, \quad \tilde{\sigma}_j = \gamma_{2j} + \gamma_{2j-1},
    \ee 
study the evolution of $\rho_j - \tilde{\rho}_j$ and $\sigma_j - \tilde{\sigma}_j$, sum over all indices $j \in \left\{1, \dots, N/2\right\}$, and obtain good control of the norms $\|c_i - \gamma_i\|_{L^2}$.
\item If the number of ionic species $N$ is odd, and the pairing conditions $z_1 = -z_2$, $z_3 = -z_4, \dots, z_{N-2} = - z_{N-1}$, $z_{N} = - z_k$  and {$D_1 = D_2, \dots, D_{N-2} = D_{N-1}, D_N = D_k$} hold for some $k \in \left\{1, \dots, N-1 \right\}$,  then one can define $\rho_j$, $\tilde{\rho}_j$, $\sigma_{j}$, and $\tilde{\sigma}_j$ for all integers $j \in \left\{1, \dots, \frac{N-1}2 \right\}$ as in \eqref{even12}, and
\be 
\rho_{\frac{N+1}2} = c_{N} - c_{k}, \quad \tilde{\rho}_{\frac{N+1}2} = \gamma_{N} - \gamma_{k} \quad \text{and} \quad
\sigma_{\frac{N+1}2} = c_{N} + c_{k}, \quad \tilde{\sigma}_{\frac{N+1}2} = \gamma_{N} + \gamma_{k},
    \ee 
and repeat the same argument as in Lemma \ref{valpair}.
\end{enumerate}
{More generally,} the result of Lemma \ref{valpair} holds whenever each $i$-th ionic species with valence $z_i$ {and diffusivity $D_i$} can be paired with a $j$-th ionic species with valence $z_j = - z_i$ {and diffusivity $D_j = D_i$}.  
\end{rem}

\beg{cor} \la{firstcor} 
Let $u_0 \in H$ and $c_i(0) \in L^2$ for all $i \in \left\{1, \dots, N\right\}$, and suppose the ionic concentrations are nonnegative. Suppose that Setting \ref{setting1} holds.
Then for $0<\eta < \fr{1}{4\|(-\Delta)^{-\fr{1}{2}} g \|_{L^2}^2} $, there exists a positive constant $C_3$ depending on $\eta$, the initial data, and the parameters of the S-NPNS system, such that for all $i \in \left\{1, \dots, N\right\}$, the estimate below holds
\be 
\int_{0}^{t} \E \log \left(1 + \|c_i- \gamma_i\|_{L^2}^2 \right) ds
\le C_3.
\ee
\end{cor}

\begin{proof} We denote by $|z|$ the maximum value of $|z_1|, \dots, |z_N|$. It is clear that for any $\alpha\in(0,1]$ there exists a constant $k>0$ depending on $\alpha$ such that the inequality $\log(1+x) \leq kx^\alpha$ holds for all $x \ge 0$.
    Consequently, it holds that
    \be \label{logestimate}
\log \left(1 + \|c_i(t)- \gamma_i\|_{L^2}^2 \right) 
\le K\|c_i(t) - \gamma_i\|_{L^2}^{\fr{\eta}{C_2|z|^2}}, 
    \ee for all $t \ge 0$, where $K$ is a positive constant depending on $\eta, |z|$ and some universal constants, and $C_2$ is the constant in estimate \eqref{expm8}. Due to \eqref{expm8}, we have 
    \be 
\|c_i(t) - \gamma_i\|_{L^2}^{\fr{\eta}{C_2 |z|^2}}
\le \Gamma_0^{\fr{\eta}{2C_2|z|^2}} \exp \left\{-\fr{\eta D}{2}(cC_2^{-1}|z|^{-2} - \mathcal{M}) t\right\} \exp \left\{\fr{\eta D}{2} \int_{0}^{t} \|\rho(s)\|_{L^2}^2 ds \right\}, 
    \ee where $\Gamma_0$ is given by \eqref{inda}.
    Now we apply the expectation $\E$ on both sides and use \eqref{expm2} to obtain 
    \be 
    \beg{aligned}
\E \|c_i(t) - \gamma_i\|_{L^2}^{\fr{\eta}{C_2|z|^2}}
&\le \Gamma_0^{\fr{\eta}{2C_2|z|^2}} \exp \left\{-\fr{\eta D}{2}(cC_2^{-1}|z|^{-2} - \mathcal{M}) t\right\} \exp \left\{C_1 \eta \left(\|u_0\|_{L^2}^2 + \|\na \Phi_0\|_{L^2}^2 + \|g\|_{L^2}^2 t + \|f\|_{L^2}^2t \right) \right\}
\\&= \Gamma_{00} \exp \left\{-\fr{c\eta D}{4C_2|z|^2} t\right\} \exp \left\{-\fr{c\eta D}{4C_2|z|^2}t  + \fr{\eta \mathcal{M} D }{2} t + C_1 \eta\|g\|_{L^2}^2 t + C_1 \eta \|f\|_{L^2}^2t \right\},
\end{aligned}
    \ee 
where
$\Gamma_{00} = \Gamma_0^{\fr{\eta}{2C_2|z|^2}} \exp \left\{C_1 \eta \left(\|u_0\|_{L^2}^2 + \|\na \Phi_0\|_{L^2}^2\right) \right\}$ 
    is a constant depending only on the initial data. Finally, if the body forces $f$, noise $g$, and boundary data $\gamma_1, \dots, \gamma_N$ are sufficiently small so that 
    \be 
\fr{\mathcal{M} D }{2}  + C_1 \|g\|_{L^2}^2  + C_1  \|f\|_{L^2}^2 \le \fr{cD}{4C_2|z|^2},
    \ee 
    we have 
    \be 
\E \log \left(1 + \|c_i- \gamma_i\|_{L^2}^2 \right)  \le K \Gamma_{00} \exp \left\{-\fr{c\eta D}{4C_2|z|^2} t\right\},
    \ee which gives
    \be 
\int_{0}^{t} \E \log \left(1 + \|c_i- \gamma_i\|_{L^2}^2 \right) ds
\le \fr{4C_2z^2 K \Gamma_{00} }{c\eta D}, 
    \ee after integrating in time. Thus, we obtain the desired result.
\end{proof}

{
\beg{rem} \label{quadraticremark}
If $g$ is assumed to be sufficiently small, then $\eta$ can be chosen to be large. In particular, one can take $\eta = 2C_2|z|^2$ in \eqref{logestimate} and deduce the quadratic moment estimate
\be \la{quadraticmoment}
\sum\limits_{i=1}^{N} \int_{0}^{t} \E \|c_i(s)  - \gamma_i\|_{L^2}^2 ds \le C_3',
\ee for any $t \ge 0$. Here $C_3'$ is a positive constant depending only on $\|c_i(0) - \gamma_0\|_{L^2}$, the parameters of the problem, and some universal constants. The bound \eqref{quadraticmoment} will be used later to study the regularity of any invariant measure for the S-NPNS problem. 
\end{rem}
}

{
\beg{rem} We observe that diffusivities are influenced by the specific environment in which the phenomenon is investigated. For example, when a substance is exposed to elevated temperatures, its diffusion coefficient tends to increase. Therefore, it is reasonable to expect higher diffusivities when certain parameters of the problem are altered.
\end{rem}
}

\subsubsection{Moment bounds under Setting \ref{setting2}} In this subsection, we present one main proposition required for the proof of Theorem \ref{kbprocedure} under Setting \ref{setting2}.

\begin{prop} \label{prop9} Let $u_0 \in H$ and $c_i(0) \in L^2$ for all $i \in \left\{1, \dots, N\right\}$, and suppose the ionic concentrations are nonnegative. Under Setting \ref{setting2}, it holds that 
\be 
\sum\limits_{i=1}^{N} \|c_i(t) - \gamma_i\|_{L^2}^2
\le C \left(\sum\limits_{j=1}^{N} \|c_j(0) - \gamma_j\|_{L^2}^2 \right) e^{C\sum\limits_{j=1}^{N} \|c_j(0)  - \gamma_j\|_{L^2}^4} e^{-cDt},
\ee and 
\be \label{gradientuniformcontrol}
\beg{aligned}
\sum\limits_{i=1}^{N} \int_{0}^{t}  D \|\na c_i(s)\|_{L^2}^2 ds
\le  C \left(\sum\limits_{j=1}^{N} \|c_j(0) - \gamma_j\|_{L^2}^2 \right) e^{C\sum\limits_{j=1}^{N} \|c_j(0)  - \gamma_j\|_{L^2}^4}
\end{aligned},
\ee for $t \ge 0$, where $C$ is a positive constant depending only on the parameters of the problem. 
\end{prop}

\begin{proof} The proof is divided into two main steps. 

\smallskip
\noindent{\bf{Step 1. Charge density bounds.}} The charge density $\rho$ obeys
\be \la{unitvalence}
\fr{1}{z^2}\pa_t \rho + \fr{1}{z^2} u \cdot \na \rho - \fr{D}{z^2} \Delta \rho 
= D \sum\limits_{i=1}^{N} \na \cdot (c_i \na \Phi)
= D \na \cdot (\tilde \rho \na \Phi), 
\ee 
where 
$
\tilde{\rho} = \sum\limits_{i=1}^{N} c_i.
$ 
Here the condition $|z_i|=z>0$ is used. The sum of the ionic concentrations $\tilde{\rho}$ evolves according to 
\be \la{unitvalence2}
\pa_t{\tilde{\rho}} + u \cdot \na \tilde{\rho} - D\Delta \tilde{\rho} = D \sum\limits_{i=1}^{N} \na \cdot (z_ic_i \na \Phi)
= D \na \cdot (\rho \na \Phi).
\ee We take the scalar product in $L^2$ of the equation \eqref{unitvalence} obeyed by $\rho$ with $\rho - \sum\limits_{i=1}^{N} z_i \gamma_i$ and the equation \eqref{unitvalence2} obeyed by $\tilde{\rho}$ with $\tilde{\rho} -\sum\limits_{i=1}^{N} \gamma_i$. We sum the resulting energy equalities, integrate by parts, and use condition \eqref{upgradedcond} to obtain 
\be 
\beg{aligned}
&\fr{1}{2} \fr{d}{dt} \left(\fr{1}{z^2}\left\|\sum\limits_{i=1}^{N} z_i(c_i - \gamma_i) \right\|_{L^2}^2 + \left\|\sum\limits_{i=1}^{N} (c_i - \gamma_i) \right\|_{L^2}^2  \right)
+ D \left(\fr{1}{z^2}\left\|\sum\limits_{i=1}^{N} z_i \na c_i\right\|_{L^2}^2 + \left\|\sum\limits_{i=1}^{N} \na c_i \right\|_{L^2}^2  \right)
\\&\quad\quad\quad\quad= D\int_{\mathcal{O}} \na \cdot (\tilde{\rho} \na \Phi) \rho dx
- D\int_{\mathcal{O}} \rho \na \Phi \cdot \na \tilde{\rho} dx
\\&\quad\quad\quad\quad= D\int_{\mathcal{O}} (\na \tilde{\rho} \cdot \na \Phi) \rho dx
+ D\int_{\mathcal{O}} \tilde{\rho} \Delta \Phi \rho dx
- D\int_{\mathcal{O}} \rho  \na \Phi \cdot  \na \tilde{\rho} dx,
\end{aligned}
\ee which reduces to 
\be \la{gradcharge}
\fr{1}{2} \fr{d}{dt} \left(\fr{1}{z^2} \left\|\sum\limits_{i=1}^{N} z_i(c_i - \gamma_i) \right\|_{L^2}^2 + \left\|\sum\limits_{i=1}^{N} (c_i - \gamma_i) \right\|_{L^2}^2  \right)
+ D \left(\fr{1}{z^2} \left\|\sum\limits_{i=1}^{N} z_i \na c_i\right\|_{L^2}^2 + \left\|\sum\limits_{i=1}^{N} \na c_i \right\|_{L^2}^2  \right) + D\|\rho \sqrt{\tilde{\rho}} \|_{L^2}^2 =0.
\ee 
In view of the Poincar\'e inequality, we have 
\be 
 \fr{d}{dt} \left(\fr{1}{z^2} \left\|\sum\limits_{i=1}^{N} z_i(c_i - \gamma_i) \right\|_{L^2}^2 + \left\|\sum\limits_{i=1}^{N} (c_i - \gamma_i) \right\|_{L^2}^2  \right)
+ 2cD \left(\fr{1}{z^2}\left\|\sum\limits_{i=1}^{N} z_i(c_i - \gamma_i) \right\|_{L^2}^2 + \left\|\sum\limits_{i=1}^{N} (c_i - \gamma_i) \right\|_{L^2}^2  \right) \le 0.
\ee  We deduce that $\rho$ and $\tilde{\rho}$ decay exponentially in time to their boundary data and obey
\be 
\fr{1}{z^2} \left\|\sum\limits_{i=1}^{N} z_i(c_i(t) - \gamma_i) \right\|_{L^2}^2 + \left\|\sum\limits_{i=1}^{N} (c_i(t) - \gamma_i) \right\|_{L^2}^2  
\le \left(\fr{1}{z^2} \left\|\sum\limits_{i=1}^{N} z_i(c_i(0) - \gamma_i) \right\|_{L^2}^2 + \left\|\sum\limits_{i=1}^{N} (c_i(0) - \gamma_i) \right\|_{L^2}^2  \right) e^{-2cDt}
\ee for all $t \ge 0$. Therefore, it holds that, for all $t \ge 0$, 
\be \la{specialcase1}
\|\rho(t)\|_{L^2}^2 \le \left( \left\|\sum\limits_{i=1}^{N} z_i(c_i(0) - \gamma_i) \right\|_{L^2}^2 + z^2 \left\|\sum\limits_{i=1}^{N} (c_i(0) - \gamma_i) \right\|_{L^2}^2  \right) e^{-2cDt}.
\ee

\smallskip
\noindent{\bf{Step 2. Ionic concentrations bounds.}} For $i \in \left\{1, \dots, N \right\}$, the $L^2$ norm of $c_i - \gamma_i$ evolves according to 
\be \la{specialcase3}
\fr{d}{dt} \|c_i - \gamma_i\|_{L^2}^2
+ D\|\na c_i\|_{L^2}^2
\le CD \|\na \Phi\|_{L^2}^2 \|\rho\|_{L^2}^2 \|c_i - \gamma_i\|_{L^2}^2 
+ CD \gamma_i^2 \|\rho\|_{L^2}^2,
\ee as shown in \eqref{prop5eq1}. From \eqref{specialcase3}, we have 
\be \la{specialcase4}
\fr{d}{dt} \|c_i - \gamma_i\|_{L^2}^2
+ D\|\na c_i\|_{L^2}^2
\le CD \|\rho\|_{L^2}^4 \|c_i - \gamma_i\|_{L^2}^2 
+ CD \gamma_i^2 \|\rho\|_{L^2}^2
\ee after making use of elliptic regularity estimates. By the Poincar\'e inequality, we obtain 
\be 
\fr{d}{dt} \|c_i - \gamma_i\|_{L^2}^2
+ D\left(c - C\|\rho\|_{L^2}^4 \right) \|c_i - \gamma_i\|_{L^2}^2
\le CD \gamma_i^2 \|\rho\|_{L^2}^2.
\ee We multiply by the integrating factor $e^{cDt - CD\int_{0}^{t} \|\rho\|_{L^2}^4} ds$, integrate in time from $0$ to $t$, use the decaying estimate \eqref{specialcase1}, and infer that 
\be 
\|c_i(t) - \gamma_i\|_{L^2}^2 
\le C \left(\sum\limits_{j=1}^{N} \|c_j(0) - \gamma_j\|_{L^2}^2 \right) e^{C\sum\limits_{j=1}^{N} \|c_j(0)  - \gamma_j\|_{L^2}^4} e^{-cDt},
\ee for all $t \ge 0$. Integrating \eqref{specialcase4} in time from $0$ to $t$, we conclude that, for all $t \ge 0$
\be 
\int_{0}^{t} \|\na c_i(s)\|_{L^2}^2 ds
\le  C \left(\sum\limits_{j=1}^{N} \|c_j(0) - \gamma_j\|_{L^2}^2 \right) e^{C\sum\limits_{j=1}^{N} \|c_j(0)  - \gamma_j\|_{L^2}^4}.
\ee 
\end{proof}

\beg{rem} In view of Proposition \ref{prop9}, we deduce that the ionic concentrations decay in the spatial $L^2$ norm to their boundary values exponentially fast in time provided that the species have equal diffusivities and absolute valences. No smallness conditions are imposed neither on the initial and boundary data nor on the forcing terms $f$ and $g$. If the initial concentrations are spatially $L^p(\mathcal{O})$ regular for some even integer $p$, then we deduce that the decay holds in $L^p$ as well. This is obtained as a consequence of Theorem \ref{t4}. 
\end{rem}

\subsection{Regularity of the Invariant Measures}
In this subsection, we address the regularity of any invariant measure associated with the initial boundary value S-NPNS problem. We shall start by establishing logarithmic moment bounds for higher-order derivatives of the solution. 

\beg{prop} \label{boundedsmooth} Let $f \in H$ and $g \in \mathcal{D}(A^{\fr{1}{2}})$. Suppose $c_i(0) \in H^1(\mathcal{O})$ for all $i \in \left\{1, \dots, N\right\}$ and $u_0 \in \mathcal{D}(A^{\fr{1}{2}})$. {Furthermore, suppose that $g$ is sufficiently small in $L^2$.} Under the hypotheses of Theorem \ref{kbprocedure}, we have 
\be \label{H2moment}
\E \int_{0}^{T} \log \left(1 + \|\Delta c_i(t)\|_{L^2}^2 + \|Au(t)\|_{L^2}^2 \right) dt \le R_3(\|\na u_0\|_{L^2}, \|\na c_i(0)\|_{L^2}, \|\na g\|_{L^2}) + R_4(\|f\|_{L^2}, \|\na g\|_{L^2})T 
\ee for all times $T \ge 0$, where $R_3$ is a positive constant depending only on $\|\na u_0\|_{L^2}, \|\na c_i(0)\|_{L^2}, g$, the parameters of the problems, and some universal constants, and $R_4$ is a positive constant depending only on $f$, $g$, the parameters of the problem and some universal constants.
\end{prop}

\begin{proof} The gradient of the $i$-th concentration $c_i$ evolves in $L^2$ according to the deterministic energy equality
\be 
\beg{aligned}
\fr{d}{dt} \|\na c_i\|_{L^2}^2
+ 2D \|\Delta c_i\|_{L^2}^2
&= 2(u \cdot \na c_i, \Delta c_i)_{L^2} 
+ 2Dz_i(\na c_i \cdot \na \Phi), \Delta c_i)_{L^2} 
\\&\quad\quad+ 2D z_i ((c_i - \gamma_i)\Delta \Phi, \Delta c_i)_{L^2}
+ 2Dz_i\gamma_i (\Delta \Phi, \Delta c_i)_{L^2},
\end{aligned}
\ee and consequently, their sum obeys
\be 
\beg{aligned}
\fr{d}{dt} \sum\limits_{i=1}^{N} \|\na c_i\|_{L^2}^2
+ 2D \sum\limits_{i=1}^{N} \|\Delta c_i\|_{L^2}^2
&= 2\sum\limits_{i=1}^{N} (u \cdot \na c_i, \Delta c_i)_{L^2} 
+ 2D\sum\limits_{i=1}^{N} z_i(\na c_i \cdot \na \Phi), \Delta c_i)_{L^2} 
\\&\quad\quad+ 2D \sum\limits_{i=1}^{N} z_i ((c_i - \gamma_i)\Delta \Phi, \Delta c_i)_{L^2}
+ 2D\sum\limits_{i=1}^{N} z_i\gamma_i (\Delta \Phi, \Delta c_i)_{L^2}.
\end{aligned}
\ee In contrast, the stochastic evolution of the velocity in the spatial norm of $\mathcal{D}(A^{\fr{1}{2}})$ is described by 
\be 
\beg{aligned}
d \|A^{\fr{1}{2}} u\|_{L^2}^2
+ 2\|Au\|_{L^2}^2 dt
&= -2(B(u,u), Au)_{L^2} dt
- 2(\rho \na \Phi, Au)_{L^2} dt
\\&\quad\quad - 2(f, Au)_{L^2} dt
+ \|A^{\fr{1}{2}} g\|_{L^2}^2 dt
- 2(g, Au)_{L^2} dW.
\end{aligned}
\ee 
We consider the stochastic processes
\be 
X(t) = \log \left(1+ \|A^{\fr{1}{2}} u(t)\|_{L^2}^2 + \sum\limits_{i=1}^{N} \|\na c_i(t)\|_{L^2}^2 \right)  \quad \text{and} \quad 
Y(t) = \|Au(t)\|_{L^2}^2 + D\sum\limits_{i=1}^{N} \|\Delta c_i(t)\|_{L^2}^2.
\ee An application of It\^o's lemma gives 
\be \la{itodiff}
\beg{aligned}
d X
+ 2Ye^{-X} dt
&= 2e^{-X} \sum\limits_{i=1}^{N} (u \cdot \na c_i, \Delta c_i)_{L^2} dt
+ 2e^{-X} D\sum\limits_{i=1}^{N} z_i(\na c_i \cdot \na \Phi), \Delta c_i)_{L^2} dt
\\&\quad\quad+ 2e^{-X}D \sum\limits_{i=1}^{N} z_i ((c_i - \gamma_i)\Delta \Phi, \Delta c_i)_{L^2} dt
+ 2e^{-X}D\sum\limits_{i=1}^{N} z_i\gamma_i (\Delta \Phi, \Delta c_i)_{L^2} dt
\\&\quad\quad-2e^{-X} (B(u,u), Au)_{L^2} dt
- 2e^{-X} (\rho \na \Phi, Au)_{L^2} dt
- 2e^{-X}(f, Au)_{L^2} dt
\\&\quad\quad+ e^{-X}\|A^{\fr{1}{2}} g\|_{L^2}^2 dt
- 2e^{-2X}(g, Au)_{L^2}^2 dt
- 2e^{-X} (g, Au)_{L^2} dW.
\end{aligned}
\ee 
Now we estimate the nonlinear terms. In view of the divergence-free condition obeyed by the velocity, we integrate by parts, estimate using $L^4$ interpolation inequalities and the boundedness of $e^{-X}$ by $1$, and obtain 
\be 
\beg{aligned}
&\left|2e^{-X} \sum\limits_{i=1}^{N} (u \cdot \na c_i, \Delta c_i)_{L^2} \right|
\le Ce^{-X} \sum\limits_{i=1}^{N} \int_{\mathcal{O}} |\na u| |\na c_i|^2 dx
\\&\quad\quad\le Ce^{-X} \sum\limits_{i=1}^{N} \|\na u\|_{L^2} \|\na c_i\|_{L^4}^2
\le Ce^{-X}  \sum\limits_{i=1}^{N} \|\na u\|_{L^2} \|\na c_i\|_{L^2} \|\Delta c_i\|_{L^2} 
\\&\quad\quad\le \fr{1}{16} e^{-X}Y + Ce^{-X} \|\na u\|_{L^2}^2 \sum\limits_{i=1}^{N} \|\na c_i\|_{L^2}^2
\le \fr{1}{16} e^{-X}Y + C\|\na u\|_{L^2}^2.
\end{aligned}
\ee 
Elliptic estimates provide bounds on the $L^4$ norm of $\na \Phi$ and yield
\be 
\beg{aligned}
&\left|2e^{-X} D\sum\limits_{i=1}^{N} z_i(\na c_i \cdot \na \Phi), \Delta c_i)_{L^2}  \right| \\
&\quad\quad \le 2De^{-X} \sum\limits_{i=1}^{N} \|\na c_i\|_{L^4} \|\na \Phi\|_{L^4} \|\Delta c_i\|_{L^2}
\le CDe^{-X} \sum\limits_{i=1}^{N} \|\na \Phi\|_{L^2}^{\fr{1}{2}}\|\rho\|_{L^2}^{\fr{1}{2}} \|\na c_i\|_{L^2}^{\fr{1}{2}} \|\Delta c_i\|_{L^2}^{\fr{3}{2}}\\
&\quad \quad \le \fr{1}{16}e^{-X}Y + Ce^{-X} \|\na \Phi\|_{L^2}^2 \|\rho\|_{L^2}^2 \sum\limits_{i=1}^{N} \|\na c_i\|_{L^2}^2
\le \fr{1}{16} e^{-X}Y + C \|\na \Phi\|_{L^2}^2 \|\rho\|_{L^2}^2.
\end{aligned}
\ee 
By making use of the Poisson equation obeyed by $\Phi$, we have
\be 
\beg{aligned}
&\left|  2e^{-X}D \sum\limits_{i=1}^{N} z_i ((c_i - \gamma_i)\Delta \Phi, \Delta c_i)_{L^2}\right| \\
&\quad\quad\le CDe^{-X} \sum\limits_{i=1}^{N} |z_i| \|c_i - \gamma_i\|_{L^4} \|\rho\|_{L^4} \|\Delta c_i\|_{L^2} 
\le CDe^{-X} \sum\limits_{i=1}^{N} |z_i| \|c_i - \gamma_i\|_{L^2}^{\fr{1}{2}} \|\na c_i\|_{L^2}^{\fr{1}{2}} \|\rho\|_{L^2}^{\fr{1}{2}} \|\na \rho\|_{L^2}^{\fr{1}{2}} \|\Delta c_i\|_{L^2} \\
&\quad\quad\le \fr{1}{16}e^{-X}Y + Ce^{-X} \sum\limits_{i=1}^{N} \|c_i - \gamma_i\|_{L^2}^2 \|\na c_i\|_{L^2}^2
+ Ce^{-X} \|\rho\|_{L^2}^2 \|\na \rho\|_{L^2}^2 \le \fr{1}{16}e^{-X}Y + C\sum\limits_{i=1}^{N} \|c_i - \gamma_i\|_{L^2}^2 
+ C \|\rho\|_{L^2}^2
\end{aligned}
\ee and
\be 
\left|2e^{-X}D\sum\limits_{i=1}^{N} z_i\gamma_i (\Delta \Phi, \Delta c_i)_{L^2}\right|
\le \fr{1}{16}e^{-X}Y + C\|\rho\|_{L^2}^2.
\ee Due to the self-adjointness of the Leray projector, the Ladyzhenskaya's $L^4$ inequality, and the ellipticity of the Stokes operator, we bound
\be 
\beg{aligned}
&\left|2e^{-X} (B(u,u), Au)_{L^2}\right|
= 2e^{-X} (u \cdot \na u, Au)_{L^2}
\le Ce^{-X} \|u\|_{L^2}^{\fr{1}{2}} \|\na u\|_{L^2} \|\Delta u\|_{L^2}^{\fr{1}{2}} \|Au\|_{L^2}
\\&\quad\quad\le Ce^{-X} \|u\|_{L^2}^{\fr{1}{2}} \|A^{\fr{1}{2}}u\|_{L^2} \|Au\|_{L^2}^{\fr{3}{2}}
\le \fr{1}{16}e^{-X}Y + Ce^{-X} \|u\|_{L^2}^2\|\na u\|_{L^2}^4
\le \fr{1}{16}e^{-X}Y + C\|u\|_{L^2}^2\|\na u\|_{L^2}^2.
\end{aligned}
\ee Exploiting again the properties of $\PP$ and interpolating, we obtain
\be 
\beg{aligned}
&\left|2e^{-X} (\rho \na \Phi, Au)_{L^2} \right|
\le Ce^{-X} \|\rho\|_{L^4} \|\na \Phi\|_{L^4} \|Au\|_{L^2}
\le Ce^{-X} \|\rho\|_{L^2} \|\na \rho\|_{L^2}^{\fr{1}{2}} \|\na \Phi\|_{L^2}^{\fr{1}{2}} \|Au\|_{L^2}
\\&\quad\quad\le \fr{1}{16}e^{-X}Y + Ce^{-X}\|\na \Phi\|_{L^2}^2 \|\rho\|_{L^2}^2 
+ Ce^{-X} \|\rho\|_{L^2}^2 \|\na \rho\|_{L^2}^2
\le \fr{1}{16}e^{-X}Y + C\|\na \Phi\|_{L^2}^2 \|\rho\|_{L^2}^2 
+ C\|\rho\|_{L^2}^2.
\end{aligned}
\ee 
{The linear terms can be handled easily by using Cauchy-Schwarz and Young's inequality.}
Thus, the It\^o's differential equality \eqref{itodiff} gives rise to the energy inequality
\be 
\beg{aligned}
dX + Ye^{-X}dt
&\le C\sum\limits_{i=1}^{N} \|c_i - \gamma_i\|_{L^2}^2 dt + C\|\na u\|_{L^2}^2 dt+ C\|u\|_{L^2}^2 \|\na u\|_{L^2}^2 dt+ C\|\rho\|_{L^2}^2 dt+ C\|\na \Phi\|_{L^2}^2 \|\rho\|_{L^2}^2 dt
\\&\quad\quad\quad\quad+ C\|f\|_{L^2}^2 dt+ C\|\na g\|_{L^2}^2dt 
-2 e^{-X} (g, Au)_{L^2} dW.
\end{aligned}
\ee
Integrating in time from $0$ to $t$, applying $\E$, using the moment bounds obtained in Proposition \ref{prop7} and Remark \ref{quadraticremark}
we deduce that the time integral of $\E (Ye^{-X})$ grows at most linearly in time. Using, in addition, the concentration gradient estimates derived in Propositions \ref{prop8} and \ref{prop9}, we obtain \eqref{H2moment}. 
\end{proof}

\beg{Thm} \label{boundeddomainh2regularity} 
Let $f \in H$, and $g \in \mathcal{D}(A^{\fr{1}{2}})$ be sufficiently small. Let $\mu$ be an invariant measure for the S-NPNS problem \eqref{sys:SNPNS} with boundary conditions 
\be 
u|_{\pa \mathcal{O}} = 0, \quad \Phi|_{\pa \mathcal{O}} = \gamma, \quad (c_1,\dots,c_N)|_{\pa \mathcal{O}}  = (\gamma_1, \dots,  \gamma_N).
\ee  Under the hypotheses of Theorem \ref{kbprocedure}, it holds that
\be 
\int_{\tilde{\mathcal{H}}} \log (1 + \|\omega\|_{H^2}^2) d\mu(\omega) < \infty.
\ee In other words, the invariant measure $\mu$ is supported on the Sobolev space $H^2$. 
\end{Thm}

\begin{proof} The invariant probability measure $\mu$ satisfies 
\be \label{invariancereforumlation}
\int_{\tilde{\mathcal{H}}} \phi (\omega_0) d\mu(\omega_0) = \int_{\tilde{\mathcal{H}}} \int_{\tilde{\mathcal{H}}} \fr{1}{T} \int_{0}^{T} P_t (\omega_0, d\omega) \phi (\omega) dt d\mu(\omega_0),
\ee for any $T > 0$ and $\phi \in C_b(\tilde{\mathcal{H}})$. For an integer $n \ge 1$, we denote by $P_n$ and $\tilde{P_n}$ the projections onto the spaces spanned by the first $n$ eigenfunctions of the Stokes operator $A$ and the homogeneous Dirichlet Laplacian $-\Delta$, respectively. The operators $P_n$ and $\tilde{P_n}$ commute with  $A^{\fr{1}{2}}$ and $(-\Delta)^{\fr{1}{2}}$, respectively. 
For an integer $n \ge 1$, a real number $R>0$, and a vector $\omega = (u, c_1, \dots, c_N) \in \tilde{\mathcal{H}}$, we define 
\be 
\psi_{n,R} (\omega) = \log \left(1 + \sum\limits_{i=1}^{N} \|(-\Delta)^{\fr{1}{2}} \tilde{P_n} c_i \|_{L^2}^2 + \|A^{\fr{1}{2}} P_n u\|_{L^2}^2 \right) \wedge R,
\ee 
and note that $\psi_{n,R}$ is well-defined on $\tilde{\mathcal{H}}$ and obeys $\psi_{n,R} \in C_b(\tilde{\mathcal{H}})$. As the operators $P_n$ and $A^{\fr{1}{2}}$, and $\tilde{P}_n$ and $(-\Delta)^{\fr{1}{2}}$ commutes, and due to the boundedness of the projections $P_n$ and $\tilde{P}_n$ on $L^2$, we estimate, for any $T>0$,
\be 
\begin{aligned} 
&\left|\fr{1}{T} \int_{0}^{T} \int_{\tilde{\mathcal{H}}} P_t(\omega_0, d\omega) \psi_{n,R}(\omega) dt \right|
= \left|\fr{1}{T} \int_{0}^{T} \E \psi_{n,R}(\omega(t, \omega_0)) dt \right| 
\\&\le \fr{1}{T} \E \int_{0}^{T} \log \left(1 +  \sum\limits_{i=1}^{N} \|\na c_i \|_{L^2}^2 + \|A^{\fr{1}{2}}u\|_{L^2}^2 \right)
\le C_1(\|c_i(0)\|_{L^2}, \|u_0\|_{L^2})T^{-1} + C_2(\|f\|_{L^2}, \|g\|_{L^2})
\end{aligned}
\ee

Let $B_{\tilde{\mathcal{H}}}(\rho)$ be the ball
$B_{\tilde{\mathcal{H}}}(\rho) = \left\{\omega \in \tilde{\mathcal{H}}: \|\omega\|_{\tilde{\mathcal{H}}}^2 \le \rho^2 \right\}.$
In view of the invariance property \eqref{invariancereforumlation}, we have
\be
\begin{aligned}
&\int_{\tilde{\mathcal{H}}} \psi_{n,R} (\omega_0) d\mu(\omega_0)\\
&\quad\le \int_{B_{\tilde{\mathcal{H}}}(\rho)} \left|\fr{1}{T} \int_{0}^{T} \int_{\tilde{\mathcal{H}}} P_t(\omega_0, d\omega) \psi_{n,R}(\omega) dt \right| d\mu(\omega_0)+ \int_{\tilde{\mathcal{H}} \setminus B_{\tilde{\mathcal{H}}}(\rho)} \left|\fr{1}{T} \int_{0}^{T} \int_{\tilde{\mathcal{H}}} P_t(\omega_0, d\omega) \psi_{n,R}(\omega) dt \right| d\mu(\omega_0) 
\\&\quad\le (C_1(\rho) T^{-1}+ C_2(\|f\|_{L^2}, \|g\|_{L^2})) \mu(B_{\tilde{\mathcal{H}}}(\rho)) + R\mu(\tilde{\mathcal{H}} \setminus B_{\tilde{\mathcal{H}}}(\rho)).
\end{aligned}
\ee
We choose a sufficiently large radius $\rho$ so that 
\be 
R\mu(\tilde{\mathcal{H}} \setminus B_{\tilde{\mathcal{H}}}(\rho)) \le 1,
\ee and then we pick a sufficiently large time $T$ such that 
\be 
C_1(\rho) T^{-1} \le 1.
\ee These choices of $\rho$ and $T$ allows us to obtain the bound
\be 
\int_{\tilde{\mathcal{H}}} \psi_{n,R} (\omega_0) d\mu(\omega_0) \le  2 + C_2(\|f\|_{L^2}, \|g\|_{L^2}), 
\ee which yields 
\be 
\int_{\tilde{\mathcal{H}}}  \left(\log \left(1 + \sum\limits_{i=1}^{N} \|\na c_i(0)\|_{L^2}^2 + \|A^{\fr{1}{2}}u\|_{L^2}^2 \right) \wedge R \right) d\mu (\omega_0) \le  2 +  C_2(\|f\|_{L^2}, \|g\|_{L^2})
\ee after an application of Fatou's lemma. By the Monotone Convergence Theorem, it holds that 
\be 
\int_{\tilde{\mathcal{H}}} \log \left(1 + \sum\limits_{i=1}^{N} \|\na c_i(0)\|_{L^2}^2 + \|A^{\fr{1}{2}}u\|_{L^2}^2 \right) d\mu(\omega_0) \le  2 +  C_2(\|f\|_{L^2}, \|g\|_{L^2}),
\ee and thus, the invariant measure $\mu$ is supported on $H^1$. Now we upgrade the regularity of the invariant measure and fix a vector $\omega$ in $H^1$. We define 
\be 
\psi_{n,R}^2 (w) = \log \left(1 + \sum\limits_{i=1}^{N} \|\Delta P_n c_i \|_{L^2}^2 +  \|A P_n u \|_{L^2}^2 \right) \wedge R,
\ee make use of Proposition \ref{boundedsmooth}, and repeat the same argument as above to obtain 
\be 
\int_{H^1} \log \left(1 + \sum\limits_{i=1}^{N} \|\Delta c_i\|_{L^2}^2 +\|\Delta u\|_{L^2}^2 \right) d\mu(\omega_0) \le C_3(\|f\|_{L^2}, \|\na g\|_{L^2}).
\ee Therefore, the invariant measure $\mu$ is supported on $H^2$. 
\end{proof}

\subsection{Uniqueness of the Invariant Measure} 

In this subsection, we employ asymptotic coupling techniques to study the uniqueness of invariant measures for the S-NPNS system.
    \begin{Thm} \label{uniquemeasure}
    Fix N species with equal sufficiently large diffusivities $D_1 = \dots = D_N = D$ and equal absolute valences $|z_1| = \dots = |z_N|$. Let $\gamma$ be a given real number and $\gamma_1, \dots, \gamma_N$ be {
    small} positive real numbers satisfying condition \eqref{upgradedcond}. Let $f \in H$ and $g \in H$ be time-independent. There exists an integer $n:= n(f, g)$ depending only on the body forces $f$, the noise $g$, the parameters of the problem, and some universal constants such that if $P_n H \subset range(g)$, then there exists at most one ergodic invariant probability measure for the Markov semigroup \eqref{markovsemigroup} associated with the S-NPNS problem \eqref{sys:SNPNS} with boundary conditions 
\be \label{bcon}
u|_{\pa \mathcal{O}} = 0, \quad \Phi|_{\pa \mathcal{O}} = \gamma, \quad (c_1,\dots,c_N)|_{\pa \mathcal{O}}  = (\gamma_1, \dots,  \gamma_N).
\ee 
Here, the operator $P_n$ is the projection onto the space spanned by the first $n$ eigenfunctions of the Stokes operator $A$.
\end{Thm}

\begin{proof} 

Fix $(u_0, c_1(0), \dots, c_N(0))$ and $(\tilde{u}_0, \tilde{c}_1(0), \dots, \tilde{c}_N(0))$ in $\mathcal{H}$, where $\mathcal{H}$ is the space defined by \eqref{Hspace}. We denote by $(u(t), c_1(t), \dots, c_N(t))$ the solution to the S-NPNS problem \eqref{sys:SNPNS} with boundary conditions \eqref{bcon} and initial data $(u_0, c_1(0), \dots, c_N(0))$, and  by $(\tilde{u}(t), \tilde{c}_1(t), \dots, \tilde{c}_N(t))$ the solution to the system
\begin{subequations}
    \begin{align}
&d\tilde{u} + \tilde u \cdot \na \tilde u dt - \Delta \tilde u dt+ \na \tilde p dt = - \tilde \rho \na \tilde \Phi dt + f dt + {\bf{1}}_{\tau_k > t} \lambda P_n (u-\tilde u)dt + gdW,
\\ 
&\pa_t \tilde{c}_i + \tilde{u} \cdot \na \tilde{c}_i - D \Delta \tilde c_i = Dz_i \na \cdot (\tilde{c}_i \cdot \na \tilde \Phi),
\\
&-\Delta \tilde{\Phi} = \tilde \rho = \sum\limits_{i=1}^{N} z_i \tilde{c}_i,
\\ 
&\na \cdot \tilde{u} = 0,
    \end{align}
\end{subequations} 
with boundary conditions
\[
\tilde{u}|_{\pa \mathcal{O}} = 0, \quad \tilde{\Phi}|_{\pa \mathcal O} = \gamma, \quad (\tilde{c}_1, \dots, \tilde{c}_N)|_{\pa \mathcal{O}} = (\gamma_1, \dots, \gamma_N)
\]
and initial data $(\tilde{u}_0, \tilde{c}_1(0), \dots, \tilde{c}_N(0))$. Define a stopping time  $\tau_k$ by 
\be 
\tau_k = \inf\limits_{t \ge 0} \left\{\int_{0}^{t} \|P_n (u- \tilde{u})\|_{L^2}^2 ds \ge k \right\},
\ee
and $k$, $n$ and $\lambda$ are constants to be determined in such a way that the set $\left\{\tau_k = \infty \right\}$ has a nonvanishing probability and the differences $\|u - \tilde{u}\|_{L^2}$ and $\|c_i - \tilde{c}_i\|_{L^2}$ converge to $0$ as $t\to\infty$ on $\left\{\tau_k = \infty \right\}$ for all $i \in \left\{1, \dots, N \right\}$. This construction implies the uniqueness of invariant measures for \eqref{sys:SNPNS} with boundary data \eqref{bcon} based on the asymptotic coupling technique of \cite{glatt2017unique}.

The proof is divided into two main steps.

\smallskip
\noindent{\bf{Step 1. }} Fix an $R>0$. There exists a positive universal constant $C$ such that the set
\be 
A_R = \left\{\omega \in \Omega: \sup\limits_{t \ge 0} \left\{\fr{1}{2}\int_{0}^{t} \|\na u(s)\|_{L^2}^2 ds - \|u_0\|_{L^2}^2 - \|\na \Phi_0\|_{L^2}^2 -  \|g\|_{L^2}^2 t - C\|f\|_{L^2}^2t \right\} > R \right\}
\ee obeys 
\be \label{expmartin}
\PP(A_R) \le \exp \left\{-\fr{R}{8\|A^{-\fr{1}{2}}g\|_{L^2}^2} \right\}.
\ee Indeed, integrating the differential inequality \eqref{expm1} in time from $0$ to $t$ and taking the supremum over all positive times give rise to the inequality
\be 
\beg{aligned}
&\sup\limits_{t \ge 0} \left\{\fr{1}{2}\int_{0}^{t} \|A^{\fr{1}{2}} u(s)\|_{L^2}^2 ds - \|u_0\|_{L^2}^2 - \|\na \Phi_0\|_{L^2}^2 - \|g\|_{L^2}^2 t - C\|f\|_{L^2}^2t \right\}
\\&\quad\quad \le \sup\limits_{t \ge 0} \left\{\int_{0}^{t} 2(A^{-\fr{1}{2}}g, A^{\fr{1}{2}}u)_{L^2} dW - \fr{1}{8\|A^{-\fr{1}{2}}g\|_{L^2}^2} \int_{0}^{t} 4\|A^{-\fr{1}{2}}g\|_{L^2}^2 \|A^{\fr{1}{2}} u\|_{L^2}^2 ds \right\}.
\end{aligned}
\ee 
Consequently, \eqref{expmartin} follows from exponential martingale estimates. 

\smallskip
\noindent{\bf{Step 2.}} We define the following differences
\be 
U = u - \tilde{u}, \quad P = p - \tilde{p}, \quad R = \rho - \tilde{\rho}, \quad \Psi = \Phi - \tilde{\Phi}, \quad C_1 = c_1 - \tilde{c}_1, \;\dots,\; C_N = c_N - \tilde{c}_N,
\ee 
which evolve according to the deterministic system of equations
\noeqref{phie}
\begin{subequations}\label{sys:difference}
   \begin{align} 
&\pa_t U - \Delta U + \na P + {\bf{1}}_{\tau_k > t} \lambda P_n U = - U \cdot \na u - \tilde{u} \cdot \na U - R \na \Phi - \tilde{\rho} \cdot \na \Psi, \label{Ue}
\\
&\pa_t C_i - D \Delta C_i = -U \cdot \na c_i - \tilde{u} \cdot \na C_i + Dz_i \na \cdot (C_i \na \Phi) + Dz_i \na \cdot (\tilde{c}_i \na \Psi),\label{Cie} 
\\
& -\Delta \Psi = R = \sum\limits_{i=1}^{N} z_i C_i, \quad \na \cdot U = 0, \label{phie}
   \end{align}
\end{subequations}
with boundary conditions
\be 
U|_{\pa \mathcal{O}} = 0, \Psi|_{\pa \mathcal{O}} = 0, (C_1, \dots, C_n)|_{\pa \mathcal{O}} = (0, \dots, 0). 
\ee 
We take the scalar products in $L^2$ of the $U$-equation \eqref{Ue} with $U$ and the $C_i$-equations \eqref{Cie} with $C_i$ and add them. This results in the following energy equality
\be \label{asymptoticcoupling1}
\beg{aligned}
&\fr{1}{2} \fr{d}{dt} \left(\|U\|_{L^2}^2 + \sum\limits_{i=1}^{N} \|C_i\|_{L^2}^2 \right)
+ \|\na U\|_{L^2}^2 + D\sum\limits_{i=1}^{N} \|\na C_i\|_{L^2}^2 + {\bf{1}}_{\tau_k > t} \lambda (P_nU, U)_{L^2} 
\\&\quad\quad= - (U \cdot \na u, U)_{L^2} - (R \na \Phi, U)_{L^2} - (\tilde{\rho} \cdot \na \Psi, U)_{L^2} 
\\&\quad\quad\quad\quad- \sum\limits_{i=1}^{N} (U \cdot \na c_i, C_i)_{L^2} 
- \sum\limits_{i=1}^{N} Dz_i (\na \cdot (C_i \na \Phi), C_i)_{L^2}
- \sum\limits_{i=1}^{N} Dz_i (\na \cdot (\tilde{c}_i \na \Psi), C_i)_{L^2}.
\end{aligned}
\ee 
We estimate the nonlinearities term by term. An application of the Ladyzhenskaya's interpolation inequality gives
\be 
\left|(U \cdot \na u, U)_{L^2}\right|
\le C\|\na u\|_{L^2} \|U\|_{L^2} \|\na U\|_{L^2}.
\ee 
Due to the elliptic regularity obeyed by $\Phi$ and $\Psi$ and the Poincar\'e inequality applied to the boundary vanishing scalar function $\rho$ and $\tilde{\rho}$, we have
\be 
\left|(R \na \Phi, U)_{L^2} \right|
\le C\|R\|_{L^2} \|\na \Phi\|_{L^{\infty}}
\|U\|_{L^2}
\le C\|\na {\rho}\|_{L^2} \left(\sum\limits_{i=1}^{N} \|C_i\|_{L^2} \right) \|U\|_{L^2}
\ee and
\be 
\left|(\tilde{\rho} \cdot \na \Psi, U)_{L^2} \right|
\le C\|\tilde{\rho}\|_{L^4} \|\na \Psi\|_{L^4} \|U\|_{L^2}
\le C\|\na \tilde{\rho}\|_{L^2} \left(\sum\limits_{i=1}^{N} \|C_i\|_{L^2} \right) \|U\|_{L^2}.
\ee 
Simultaneous interpolations in $L^4$ implemented on the boundary vanishing differences $U$ and $C_i$ yield
\be 
\left|\sum\limits_{i=1}^{N} (U \cdot \na c_i, C_i)_{L^2}  \right|
\le C\sum\limits_{i=1}^{N} \|U\|_{L^4} \|\na c_i\|_{L^2} \|C_i\|_{L^4}
\le C\sum\limits_{i=1}^{N}  \|U\|_{L^2}^{\fr{1}{2}}  \|\na U\|_{L^2}^{\fr{1}{2}} \|\na c_i\|_{L^2} \|C_i\|_{L^2}^{\fr{1}{2}}\|\na C_i\|_{L^2}^{\fr{1}{2}}.
\ee 
Standard elliptic estimates and the Poincar\'e inequality bring that
\be 
\left|\sum\limits_{i=1}^{N} Dz_i (\na \cdot (C_i \na \Phi), C_i)_{L^2} \right| 
\le D\sum\limits_{i=1}^{N} |z_i| \|C_i\|_{L^2} \|\na \Phi\|_{L^{\infty}} \|\na C_i\|_{L^2} 
\le CD\sum\limits_{i=1}^{N}  \|C_i\|_{L^2} \|\na \rho\|_{L^{2}} \|\na C_i\|_{L^2}
\ee  and {
\be \label{asymptoticcoupling2}
\beg{aligned}
&\left|\sum\limits_{i=1}^{N} Dz_i (\na \cdot (\tilde{c}_i \na \Psi), C_i)_{L^2} \right|
\le D \sum\limits_{i=1}^{N} |z_i| \|\tilde{c}_i\|_{L^4} \|\na \Psi\|_{L^4} \|\na C_i\|_{L^2}
\\&\quad\quad\le CD \left(\sum\limits_{i=1}^{N} \left[\|\tilde c_i\|_{L^2} + \|\na \tilde{c}_i\|_{L^2} \right]  \|\na C_i\|_{L^2}\right) \left(\sum\limits_{i=1}^{N} \|C_i \|_{L^2} \right).
\end{aligned}
\ee }
Putting \eqref{asymptoticcoupling1}--\eqref{asymptoticcoupling2} and applying Young's inequality imply that
\be 
\beg{aligned}
&\fr{d}{dt} \left(\|U\|_{L^2}^2 + \sum\limits_{i=1}^{N} \|C_i\|_{L^2}^2 \right)
+ \|A^{\fr{1}{2}} U\|_{L^2}^2 + D\sum\limits_{i=1}^{N} \|\na C_i\|_{L^2}^2 + {\bf{1}}_{\tau_k > t} \lambda \|P_n U\|_{L^2}^2
\\&\quad\quad\le C\left(\|\na u\|_{L^2}^2 + (1+D)\sum\limits_{i=1}^{N} \left(\|\na c_i\|_{L^2}^2  + \|\na \tilde{c}_i\|_{L^2}^2  + {\|\tilde c_i\|_{L^2}^2}\right) \right) \left(\|U\|_{L^2}^2 + \sum\limits_{i=1}^{N} \|C_i\|_{L^2}^2 \right).
\end{aligned}
\ee 
Now we make use of the generalized Poincar\'e inequality
\be 
\|Q_n U\|_{L^2}^2 \le \mu_{n+1}^{-1} \|A^{\fr{1}{2}}U\|_{L^2}^2,
\ee 
where $Q_n$ is the orthogonal projection of $H$ onto the space spanned by the first $n$ eigenfunctions of the Stokes operator $A$ and $\mu_{n+1}$ is the $(n+1)$-th eigenvalue of $A$. We deduce that 
\be 
\beg{aligned}
&\|A^{\fr{1}{2}} U\|_{L^2}^2 + 1_{\tau_k > t} \lambda \|P_n U\|_{L^2}^2
\ge \mu_{n+1} \|Q_n U\|_{L^2}^2 + 1_{\tau_k > t} \lambda \|P_n U\|_{L^2}^2
\\&\quad\quad
\ge 1_{\tau_k > t} \mu_{n+1} \left(\|Q_n U\|_{L^2}^2 + \|P_n U\|_{L^2}^2\right) 
= 1_{\tau_k > t} \mu_{n+1}\|U\|_{L^2}^2,
\end{aligned}
\ee 
provided that $\lambda \ge \mu_{n+1}$. Consequently, on the time interval $[0, \tau_k]$, it holds that 
\be 
\beg{aligned}
&\fr{d}{dt} \left(\|U\|_{L^2}^2 + \sum\limits_{i=1}^{N} \|C_i\|_{L^2}^2 \right)
+\min\left\{cD, \mu_{n+1} \right\} \left(\|U\|_{L^2}^2 + \sum\limits_{i=1}^{N} \|C_i\|_{L^2}^2 \right)
\\&\quad\quad\le C\left(\|\na u\|_{L^2}^2 + (1+D)\sum\limits_{i=1}^{N} \left(\|\na c_i\|_{L^2}^2  + \|\na \tilde{c}_i\|_{L^2}^2  + {\gamma_i^2}\right)  \right) \left(\|U\|_{L^2}^2 + \sum\limits_{i=1}^{N} \|C_i\|_{L^2}^2 \right),
\end{aligned}
\ee 
where $c$ is the Poincar\'e constant. Then Gronwall's inequality produces
\be \label{asycoupling}
\beg{aligned}
\|U(t)\|_{L^2}^2 + \sum\limits_{i=1}^{N} \|C_i(t)\|_{L^2}^2
&\le \left(\|U(0)\|_{L^2}^2 + \sum\limits_{i=1}^{N} \|C_i(0)\|_{L^2}^2 \right) \exp\Big({- \min \left\{cD, \mu_{n+1} \right\}t}\Big) 
\\
&\hspace{1cm}\times\exp\left({C\int_{0}^{t} \left(\|\na u\|_{L^2}^2 + (1+D)\sum\limits_{i=1}^{N} \left(\|\na c_i\|_{L^2}^2  + \|\na \tilde{c}_i\|_{L^2}^2 + {\gamma_i^2}  \right)  \right)  ds}\right),
\end{aligned}
\ee 
on $[0, \tau_k]$. In view of \eqref{gradientuniformcontrol} obeyed by $\na c_i$, and that $\na \tilde c_i$ satisfies a similar estimate, the bound \eqref{asycoupling} reduces to 
\be 
\|U(t)\|_{L^2}^2 + \sum\limits_{i=1}^{N} \|C_i(t)\|_{L^2}^2
\le C_0 \exp\left({- \min \left\{cD, \mu_{n+1} \right\}t} \right) \exp\left({{C(1+D)\gamma_i^2 t } + C\int_{0}^{t} \|\na u(s)\|_{L^2}^2 dt}\right), 
\ee 
on $[0, \tau_k]$ for some constant $C_0$ depending on the $L^2$ norms of $c_i(0), \tilde{c}_i(0), u_0, \tilde{u}_0$, the boundary data, the parameters of the problem, and some universal constants. By \textbf{Step 1}, the complement $A_R^c$ of the set $A_R$ obeys
$\PP (A_R^c) > 0$,
for any $R>0$. Moreover, we have 
\be 
\begin{aligned}
    &\|U(t)\|_{L^2}^2 + \sum\limits_{i=1}^{N} \|C_i(t)\|_{L^2}^2
\\
\le &C_0  \exp\left({2C\|u_0\|_{L^2}^2 + 2C\|\na \Phi_0\|_{L^2}^2} \right)  \exp\left({- \min \left\{cD, \mu_{n+1} \right\}t + C\|g\|_{L^2}^2 t + C\|f\|_{L^2}^2t +  {C(1+D)\gamma_i^2 t}}\right)
\end{aligned}
\ee 
for any $t \in [0, \tau_k]$ and $\omega \in A_R^c$. If, in addition, the relation 
\be \label{crucialcondition}
C\|g\|_{L^2}^2 + C\|f\|_{L^2}^2 + {  C(1+D)\gamma_i^2} \le
\fr{\min \left\{ cD, \mu_{n+1}\right\}}{2}
\ee 
holds, then we obtain 
\be 
\|U(t)\|_{L^2}^2 + \sum\limits_{i=1}^{N} \|C_i(t)\|_{L^2}^2
\le C_0e^{2C\|u_0\|_{L^2}^2 + 2C\|\na \Phi_0\|_{L^2}^2}   e^{- \fr{1}{2}\min \left\{cD, \mu_{n+1} \right\}t } 
\ee 
on $A_R^c$ for any $t \in [0, \tau_k].$ Therefore, we choose $n$ and the diffusivity $D$ to be sufficiently large and $k$ so that $A_R^c \subset \left\{\tau_k =\infty \right\}$ and conclude that the $L^2$ norm of both $U$ and $C_i$ converges to 0 in time on the nontrivial set $A_R^c$, completing the proof of Theorem \ref{uniquemeasure}. 
  \end{proof}

\section{The Periodic Case} \label{per}

In this section, we address the ergodicity of the S-NPNS system \eqref{sys:SNPNS} on the periodic box $\TT^2 = [0, 2\pi]^2$ with periodic boundary conditions. It is evident that the condition 
\be 
\sum\limits_{i=1}^{N} \frac{z_i}{|\TT^2|} \int_{\TT^2} c_i(x,t) dx = 0
\ee holds for all positive times, a fact that follows from integrating the Poisson equation obeyed by $\Phi$ spatially over $\TT^2$. 

We denote by $H_{per}$ the space of periodic, mean-free, and divergence-free two-dimensional vector fields, by $L_{per}^2$ the space of periodic $L^2(\TT^2)$ integrable functions, and by $H_{per}^s$ the space of periodic $H^s(\TT^2)$ Sobolev functions. In the sequel, we write $H$ instead of $H_{per}$ for simplicity. 

We consider the product space 
\be 
\mathcal{H}_{per} := H \times L^2 \times \dots L^2
\ee  
of vectors $\omega = (v, \xi_1, \dots, \xi_N)$ equipped with the norm 
$\|\omega\|_{\mathcal{H}_{per}}^2 = \|v\|_{L^2}^2 + \sum\limits_{i=1}^{N} \|\xi_i - \bar{\xi}_i\|_{L^2}^2,$
where $\bar{\xi}_i$ denotes the average of $\xi$ over $\TT^2$. For a given vector $K = (K_1, \dots, K_N)$ of {nonnegative} real number $K_1, \dots, K_N$, we consider the space $\tilde{\mathcal{H}}_{per}$ consisting of vectors $(v, \xi_1, \dots, \xi_N) \in \mathcal{H}_{per}$ such that the scalar functions $\xi_1, \dots, \xi_N$ are nonnegative a.e. and satisfy 
\be  \label{torusmaincondition}
\int_{\TT^2} \xi_i(x) dx = K_i, \quad \text{for all} \; i \in \left\{1, \dots, N \right\}.
\ee

For a positive time $t \ge 0$ and a Borel set $A \in \mathcal{B}(\tilde{\mathcal{H}}_{per})$, we define the Markov transition kernels associated with the S-NPNS system \eqref{sys:SNPNS} with periodic boundary conditions by 
\be 
P_t^{per}(\omega_0, A) := \PP (\omega (t, \omega_0) \in A),
\ee
where $\omega(t, \omega_0)$ denotes the solution $\omega = (u, c_1, \dots, c_N)$ to the problem \eqref{sys:SNPNS} with periodic boundary conditions and initial datum $\omega_0 = (u_0, c_1(0), \dots, c_N(0))$.
Let $\mathcal{M}_b(\tilde{\mathcal{H}}_{per})$ be the collection of bounded real-valued Borel measurable functions on $\tilde{\mathcal{H}}_{per}$. For each $t \ge 0$ and $\phi \in \mathcal{M}_b(\tilde{\mathcal{H}}_{per})$, we define the Markovian semigroup, denoted by $\tilde P_t^{per}$, by 
\be \label{markovperiodic}
\tilde P_t^{per} \phi (\omega_0) := \mathbb{E}\phi(\omega(t, \omega_0)) = \int_{\mathcal{H}} \phi (\omega) P_t^{per} (\omega_0, d\omega).
\ee 

Let $C_b(\tilde{\mathcal{H}}_{per})$ be the space of continuous bounded real-valued functions on $\tilde{\mathcal{H}}_{per}$. As shown in Theorem \ref{feller}, the semigroup $\left\{\tilde P_t^{per}\right\}_{t \ge 0}$ is Markov Feller on $C_b(\tilde{\mathcal{H}}_{per})$. 

The main goal of this section is to prove the existence of smooth invariant measures for $\left\{\tilde P_t^{per}\right\}_{t \ge 0}$ for an arbitrary number of ionic species with different diffusivities and valences, provided that $K$ is sufficiently small. No size restrictions are imposed on the noise $g$ nor on the body forces $f$, which improve the results obtained on bounded domains. 
{
The following theorem is the main result of this section.

\beg{Thm} \label{maintheoremonthetorus} 
Fix $N$ ionic species with diffusivities $D_1, \dots, D_N$ and valences $z_1, \dots, z_N$. Let $m \ge 0$ be a nonnegative integer. Suppose $f \in H^m$ and $g \in H^m$ are time-independent, divergence-free, and mean-free. If the constant $K$ defined in \eqref{torusmaincondition} is sufficiently small, then there exists an ergodic invariant probability measure $\mu^{per}$ for the Markov semigroup \eqref{markovperiodic} associated with the periodic S-NPNS problem, that is,
\be 
\int_{\tilde{\mathcal{H}}_{per}} \phi(\omega) d\mu^{per}(\omega) 
= \int_{\tilde{\mathcal{H}}_{per}} \tilde{P}_t^{per} \phi(\omega) d\mu^{per}(\omega)
\ee for any $\phi \in C_b(\tilde{\mathcal{H}}_{per})$. Moreover, the invariant measure $\mu^{per}$ is $H^m$ regular, that is, 
\be 
\int_{\tilde{\mathcal{H}}_{per}} \log \left(1 + \|\omega\|_{H^m}^2 \right) d\mu^{per}(\omega) < \infty.
\ee {If the minimum value of the diffusivities is sufficiently large, then there exists an integer $n:= n(f, g)$ depending only on the body forces $f$, the noise $g$, the parameters of the problem, and some universal constants such that if $P_n H \subset range(g)$, then the invariant measure is unique.}
\end{Thm}

Theorem \ref{maintheoremonthetorus} is a consequence of Propositions \ref{torusmoment18} and \ref{torussmoothness1} below.
} 

\beg{prop} \label{torusmoment18} Let $u_0 \in H$ and $c_i(0) \in L^2$ be nonnegative for all $i \in \left\{1, \dots, N \right\}$. Assume $f \in H$ and $g \in H$ are time-independent. Suppose the initial spatial averages of the ionic concentrations are sufficiently small. {
Then the quadratic moment bound 
\be \label{torusmoment16}
\E\int_{0}^{t} \sum\limits_{i=1}^{N} D_i \|\na c_i(s)\|_{L^2}^2  ds  
\le \mathcal{J}_0
\ee holds for all $t \ge 0$, where $\mathcal{J}_0$ is a nonnegative constant depending on the valences, diffusivities, the number of ionic species, the $L^2$ norm of the initial velocity and concentrations, the body forces $f$, the noise $g$, and some universal constants such that $\mathcal{J}_0 = 0$ when $c_i(0) = \bar{c}_i$ and $u_0 = 0$.}
\end{prop}

\begin{proof} The proof is divided into four main steps. 

\smallskip
\noindent{\bf{Step 1. Charge density $L^2$ quadratic moment bounds.}} We fix a species index $i \in \left\{1, \dots, N\right\}$ and take the $L^2(\TT^2)$ inner product of the equation obeyed by the ionic concentrations $c_i$ with $\log \left(\fr{c_i}{\bar{c}_i}\right)$. In view of the identity
\be 
\int_{\TT^2} (\pa_t c_i) \log \left(\fr{c_i}{\bar{c}_i}\right) dx
= \int_{\TT^2} \pa_t \left(c_i \log \fr{c_i}{\bar{c}_i} - c_i + \bar{c}_i\right) dx
= \fr{d}{dt} \int_{\TT^2} \left(c_i \log \fr{c_i}{\bar{c}_i} - c_i + \bar{c}_i\right) dx,
\ee and the nonlinearity cancellation
\be 
\int_{\TT^2} (u \cdot \na c_i) \log \frac{c_i}{\bar{c}_i} dx
= - \int_{\TT^2} (u \cdot \na 
\log c_i) c_i dx
= - \int_{\TT^2} u \cdot \na c_i dx
= 0,
\ee we obtain the energy equality
\be 
\fr{d}{dt} E_i 
= D_i \int_{\TT^2} \na \cdot \left(\na c_i + z_ic_i \na \Phi \right) \log \frac{c_i}{\bar{c}_i} dx
= - D_i \int_{\TT^2} \left(\na c_i + z_ic_i \na \Phi \right) \cdot \fr{\na c_i}{c_i} dx ,
\ee where $E_i(t)$ is the energy defined by
\be 
E_i(t)= \int_{\TT^2} \left(c_i(t) \log \fr{c_i(t)}{\bar{c}_i} - c_i(t) + \bar{c}_i\right) dx,
\ee at time $t \ge 0$. Obviously, $E_i(0) = 0$ when $c_i(0) = \bar{c}_i$. We rewrite the forcing migration term as 
\be 
\beg{aligned}
&- D_i \int_{\TT^2} \left(\na c_i + z_ic_i \na \Phi \right) \cdot \fr{\na c_i}{c_i} dx 
\\&= - D_i \int_{\TT^2} \fr{1}{c_i} \left(\na c_i + z_ic_i \na \Phi \right) \cdot \left(\na c_i + z_ic_i  \na \Phi \right) dx
+ D_i \int_{\TT^2} \fr{1}{c_i} \left(\na c_i + z_ic_i \na \Phi \right) \cdot \left( z_ic_i  \na \Phi \right) dx
\\&= - D_i \left\|\fr{\na c_i + z_ic_i \na \Phi}{\sqrt{c_i}} \right\|_{L^2}^2
+ D_i z_i\int_{\TT^2} \left(\na c_i + z_ic_i \na \Phi \right) \cdot {\na \Phi}  dx,
\end{aligned}
\ee and infer that
\be \label{torusmoment1}
\fr{d}{dt}E_i
+ D_i \left\|\fr{\na c_i + z_ic_i \na \Phi}{\sqrt{c_i}} \right\|_{L^2}^2
= D_i z_i\int_{\TT^2} \left(\na c_i + z_ic_i \na \Phi \right) \cdot {\na \Phi}  dx.
\ee
Seeking a cancellation of the term on the right-hand side of \eqref{torusmoment1}, we observe that 
\begin{align} 
& \sum\limits_{i=1}^{N} - D_i z_i\int_{\TT^2} \left(\na c_i + z_ic_i \na \Phi \right) \cdot {\na \Phi}  dx
= \sum\limits_{i=1}^{N} D_i z_i \int_{\TT^2} \na \cdot (\na c_i + z_i c_i \na \Phi) \Phi dx
= \sum\limits_{i=1}^{N} z_i \int_{\TT^2} \left(\pa_t c_i + u \cdot \na c_i\right) \Phi dx
\\&=\int_{\TT^2} \Phi \pa_t \rho  dx 
+ \int_{\TT^2} u \cdot \na \rho \Phi dx
= -\int_{\TT^2} \Phi \pa_t \Delta \Phi dx
- \int_{\TT^2} u \cdot \na \Phi  \rho dx
= \fr{1}{2}  \fr{d}{dt} \|\na \Phi\|_{L^2}^2
- \int_{\TT^2} u \cdot \na \Phi \rho dx.\label{torusmoment2}
\end{align}
Adding \eqref{torusmoment1} and \eqref{torusmoment2} gives
\be 
\fr{d}{dt} \left(\sum\limits_{i=1}^{N} E_i + \fr{1}{2} \|\na \Phi\|_{L^2}^2 \right)
+ \sum\limits_{i=1}^{N} D_i \left\|\fr{\na c_i + z_ic_i \na \Phi}{\sqrt{c_i}} \right\|_{L^2}^2 = \int_{\TT^2} u \cdot \na \Phi \rho dx.
\ee
Coupled with the stochastic $L^2$ evolution of the velocity $u$ described by 
\be 
\fr{1}{2} d \|u\|_{L^2}^2 
+ \|\na u\|_{L^2}^2 dt
= -(\rho \na \Phi, u)_{L^2} dt
+ (f,u)_{L^2} dt
+ \fr{1}{2} \|g\|_{L^2}^2 dt
+ (g,u)_{L^2} dW,
\ee we obtain the stochastic evolution equation
\be \label{torusmoment3}
\beg{aligned}
& d\left(\fr{1}{2} \|u\|_{L^2}^2 + \sum\limits_{i=1}^{N} E_i + \fr{1}{2} \|\na \Phi\|_{L^2}^2 \right)
+ \left( \|\na u\|_{L^2}^2  + \sum\limits_{i=1}^{N} D_i \left\|\fr{\na c_i + z_ic_i \na \Phi}{\sqrt{c_i}} \right\|_{L^2}^2 \right) dt
\\&\quad\quad\quad\quad= (f,u)_{L^2} dt
+ \fr{1}{2} \|g\|_{L^2}^2 dt
+ (g,u)_{L^2} dW.
\end{aligned}
\ee In view of the estimate
\be 
|(f,u)_{L^2}| \le \fr{1}{2} \|\na u\|_{L^2}^2 + C\|f\|_{L^2}^2,
\ee the equality \eqref{torusmoment3} gives rise to the differential inequality
\be \label{torusmoment4}
\beg{aligned}
& d\left(\fr{1}{2} \|u\|_{L^2}^2 + \sum\limits_{i=1}^{N} E_i + \fr{1}{2} \|\na \Phi\|_{L^2}^2 \right)
+ \left( \fr{1}{2} \|\na u\|_{L^2}^2  + \sum\limits_{i=1}^{N} D_i \left\|\fr{\na c_i + z_ic_i \na \Phi}{\sqrt{c_i}} \right\|_{L^2}^2 \right) dt
\\&\quad\quad\quad\quad
\le C\|f\|_{L^2}^2 dt
+ \fr{1}{2} \|g\|_{L^2}^2 dt
+ (g,u)_{L^2} dW,
\end{aligned}
\ee which yields the moment bound
\be 
\E \int_{0}^{t} \left( \fr{1}{2} \|\na u\|_{L^2}^2  + \sum\limits_{i=1}^{N} D_i \left\|\fr{\na c_i + z_ic_i \na \Phi}{\sqrt{c_i}} \right\|_{L^2}^2 \right) ds
\le \fr{1}{2} \|u_0\|_{L^2}^2 + \sum\limits_{i=1}^{N} E_i(0) + \fr{1}{2} \|\na \Phi_0\|_{L^2}^2 + C\left(\|f\|_{L^2}^2 + \|g\|_{L^2}^2\right) t, 
\ee 
after integrating in time from $0$ to $t$ and applying $\E$. 
The dissipation arising from the evolution of the ionic concentrations can be controlled as follows:
\be 
\beg{aligned}
\sum\limits_{i=1}^{N} D_i \left\|\fr{\na c_i + z_ic_i \na \Phi}{\sqrt{c_i}} \right\|_{L^2}^2
&\ge D \sum\limits_{i=1}^{N}  \left\|2\na \sqrt{c_i} + z_i\sqrt{c_i} \na \Phi \right\|_{L^2}^2
\\&= 4D\sum\limits_{i=1}^{N} \|\na \sqrt{c_i}\|_{L^2}^2
+ D\sum\limits_{i=1}^{N} z_i^2 \|\sqrt{c_i} \na \Phi\|_{L^2}^2
+ 2D (\na \rho, \na \Phi)_{L^2}
\ge D(\na \rho, \na \Phi)_{L^2}
= D\|\rho\|_{L^2}^2
\end{aligned}
\ee where $D$ is the minimum value of the diffusivities $D_1, \dots, D_N.$ It follows that the quadratic moment bound
\be \label{torusmoment17}
\E \int_{0}^{t} \left(\fr{1}{2} \|\na u\|_{L^2}^2 + D\|\rho\|_{L^2}^2 \right) ds
\le \fr{1}{2} \|u_0\|_{L^2}^2 + \sum\limits_{i=1}^{N} E_i(0) + \fr{1}{2} \|\na \Phi_0\|_{L^2}^2 + C\left(\|f\|_{L^2}^2 + \|g\|_{L^2}^2\right) t 
\ee holds for all times $t \ge 0$.

\smallskip
\noindent{\bf{Step 2. Charge density exponential moment bounds.}} In view of the stochastic inequality
\be
\beg{aligned}
& d\left(\fr{1}{2} \|u\|_{L^2}^2 + \sum\limits_{i=1}^{N} E_i + \fr{1}{2} \|\na \Phi\|_{L^2}^2 \right)
+ \left( \fr{1}{2} \|\na u\|_{L^2}^2  + D\|\rho\|_{L^2}^2 \right) dt
\le C\|f\|_{L^2}^2 dt
+ \fr{1}{2} \|g\|_{L^2}^2 dt
+ (g,u)_{L^2} dW,
\end{aligned}
\ee 
we have 
\be 
\beg{aligned}
&\E \exp{\left(\mu D\int_{0}^{t} \|\rho(s)\|_{L^2}^2 ds\right) }
\\&\le \exp \left(\mu a_0 + C\mu \left(\|f\|_{L^2}^2 + \|g\|_{L^2}^2\right) t \right) \E \exp \left(\int_{0}^{t} \mu (g,u)_{L^2} dW - \fr{\mu }{2} \int_{0}^{t} \|\na u(s)\|_{L^2}^2 ds  \right)
\end{aligned}
\ee 
for any $t \ge 0$ and  $\mu > 0$, where
\be \label{torusmoment7}
a_0 = \fr{1}{2} \|u_0\|_{L^2}^2 + \sum\limits_{i=1}^{N} E_i(0) + \fr{1}{2} \|\na \Phi_0\|_{L^2}^2.
\ee By making use of exponential martingale estimates, we infer that 
\be 
\E \exp{\left(\mu D\int_{0}^{t} \|\rho(s)\|_{L^2}^2 ds\right) }
\le \exp \left(\mu a_0 + C\mu \left(\|f\|_{L^2}^2 + \|g\|_{L^2}^2\right) t \right) 
\ee for any $t \ge 0$, provided that $\mu \in \left(0, \fr{1}{2 \|(-\Delta)^{-\fr{1}{2}}g \|_{L^2}^2} \right).$

\smallskip
\noindent{\bf{Step 3. Concentrations $L^2$ moment bounds.}} The sum of the $L^2$ norms of the ionic concentrations evolves in time according to the nonlinear deterministic equation 
 \be 
\fr{1}{2} \fr{d}{dt} \sum\limits_{i=1}^{N} \|c_i - \bar{c}_i\|_{L^2}^2 
+ \sum\limits_{i=1}^{N} D_i\|\na c_i\|_{L^2}^2 
= -  \sum\limits_{i=1}^{N} D_iz_i \int_{\mathcal{\TT^2}} (c_i - \bar{c}_i) \na \Phi \cdot \na c_i dx
-  \sum\limits_{i=1}^{N} D_i\bar{c}_i  \int_{\mathcal{\TT^2}} \na \Phi \cdot \na c_i dx.
\ee 
In view of the periodic elliptic estimate \eqref{torusmom} and $L^p$ interpolation inequalities, we have 
\be 
\|\na \Phi\|_{L^4}^4
\le C\|\rho\|_{L^{\fr{4}{3}}}^4
\le C\|\rho\|_{L^1}^2 \|\rho\|_{L^2}^2, 
\ee which reduces to 
\be 
\|\na \Phi\|_{L^4}^4
\le C\|\rho\|_{L^{\fr{4}{3}}}^4
\le C\left(\sum\limits_{i=1}^{N} |z_i| \int_{\TT^2} c_i(x,0) dx\right)^2 \|\rho\|_{L^2}^2
\ee due to the nonnegativity and conservation of the spatial averages of the ionic concentrations. As a consequence, the electromigration nonlinear term bounds as 
\be 
\beg{aligned}
&\left|\sum\limits_{i=1}^{N} D_iz_i \int_{\mathcal{\TT^2}} (c_i - \bar{c}_i) \na \Phi \cdot \na c_i dx \right|
\le \sum\limits_{i=1}^{N} D_i|z_i| \|c_i - \bar{c}_i\|_{L^4} \|\na \Phi\|_{L^4} \|\na c_i\|_{L^2}
\\&\quad\quad\le C\left(\sum\limits_{i=1}^{N} |z_i| \int_{\TT^2} c_i(x,0) dx\right)^{\fr{1}{2}}\sum\limits_{i=1}^{N} D_i|z_i| \|c_i - \bar{c}_i\|_{L^2}^{\fr{1}{2}} \|\rho\|_{L^2}^{\fr{1}{2}} \|\na c_i\|_{L^2}^{\fr{3}{2}}
\\&\quad\quad\le \sum\limits_{i=1}^{N} \fr{D_i}{4} \|\na c_i\|_{L^2}^2
+ C\left(\sum\limits_{i=1}^{N} |z_i| \int_{\TT^2} c_i(x,0) dx\right)^2 \|\rho\|_{L^2}^{2} \sum\limits_{i=1}^{N} D_i z_i^4 \|c_i - \bar{c}_i\|_{L^2}^2
\end{aligned}
\ee due to H\"older, Ladyzhenskaya, and Young inequalities. We infer that 
\be \label{torusmoment15}
\fr{d}{dt} \sum\limits_{i=1}^{N} \|c_i - \bar{c}_i\|_{L^2}^2 
+ \sum\limits_{i=1}^{N} D_i\|\na c_i\|_{L^2}^2
\le  C\left(\sum\limits_{i=1}^{N} |z_i| \int_{\TT^2} c_i(x,0) dx\right)^2 \|\rho\|_{L^2}^{2} \sum\limits_{i=1}^{N} D_i z_i^4 \|c_i - \bar{c}_i\|_{L^2}^2 
+ C\left(\sum\limits_{i=1}^{N} D_i \bar{c}_i^2 \right)\|\rho\|_{L^2}^2.
\ee Applying  the Poincar\'e inequality
$c\|c_i - \bar{c}_i\|_{L^2}^2 \le \|\na c_i\|_{L^2}^2 $ 
produces
\be 
\fr{d}{dt} \sum\limits_{i=1}^{N} \|c_i - \bar{c}_i\|_{L^2}^2
+\left(cD - r(t) \right) \sum\limits_{i=1}^{N} \|c_i - \bar{c}_i\|_{L^2}^2 \le 0
\ee where $D$ is the minimum of the diffusivities and 
\be 
r(t) = C \left(\max\limits_{1 \le i \le N} D_i \right) \left(\max\limits_{1 \le i \le N} |z_i|\right)^6 \left(\sum\limits_{i=1}^{N} \int_{\TT^2} c_i(x,0) dx\right)^2  \|\rho(t)\|_{L^2}^2 + C\left(\max\limits_{1 \le i \le N}  |z_i|\right)^2 \sum\limits_{i=1}^{N} D_i \left(\int_{\TT^2} c_i(x,0)dx\right)^2.
\ee Multiplying by the integrating factor and integrating in time from $0$ to $t$, we obtain {
\be \label{squareb}
\sum\limits_{i=1}^{N} \|c_i(t) - \bar{c}_i\|_{L^2}^2
\le  \left(\sum\limits_{i=1}^{N} \|c_i(0) - \bar{c}_i\|_{L^2}^2\right) \exp \left(-cDt + \int_{0}^{t} r(s) ds \right)
\ee
for any $t \ge 0$. Now let $p \in \left\{1, 2, 3, 4\right\}.$ From \eqref{squareb}, we deduce the moment bound
\be \la{torusmoment6}
\E \left(\sum\limits_{i=1}^{N} \|c_i(t) - \bar{c}_i\|_{L^2}^2\right)^p
\le  \left(\sum\limits_{i=1}^{N} \|c_i(0) - \bar{c}_i\|_{L^2}^2\right)^p \E\exp \left(-cpDt + p\int_{0}^{t} r(s) ds \right)
\ee holds for all $t \ge 0.$} As a consequence of Step 2, it holds that 
\be 
\E \exp \left(\mu D  \int_{0}^{t} \|\rho(s)\|_{L^2}^2 ds \right) 
\le \exp \left(\mu a_0 + C'\mu \left(\|f\|_{L^2}^2 + \|g\|_{L^2}^2\right) t \right) 
\ee where  
\be \label{torusmoment10}
\mu: = \fr{C{p}}{D} \left(\max\limits_{1 \le i \le N} D_i \right) \left(\max\limits_{1 \le i \le N} |z_i|\right)^6 \left(\sum\limits_{i=1}^{N} \int_{\TT^2} c_i(x,0) dx\right)^2 \le \min\left\{\fr{1}{2\|(-\Delta)^{-\fr{1}{2}}g\|_{L^2}^2}, 1 \right\}.
\ee Furthermore, if $\mu$ is chosen so that 
\be \label{torusmoment11}
\mu \le \fr{cD}{4C'\left(\|f\|_{L^2}^2 + \|g\|_{L^2}^2 \right)}
\ee and the initial spatial averages are chosen so that 
\be \label{torusmoment12}
 C \left(\max\limits_{1 \le i \le N}  |z_i|\right)^2 \sum\limits_{i=1}^{N} D_i \left(\int_{\TT^2} c_i(x,0)dx\right)^2
 \le \fr{cD}{4},
\ee the bound \eqref{torusmoment6} boils down to {
\be \label{torusmoment13}
\E \left(\sum\limits_{i=1}^{N} \|c_i(t) - \bar{c}_i\|_{L^2}^2\right)^p
\le \left(\sum\limits_{i=1}^{N} \|c_i(0) - \bar{c}_i\|_{L^2}^2 \right)^p e^{a_0} e^{-\fr{cpD}{2}t},
\ee }for any $t \ge 0$, where $a_0$ is given by \eqref{torusmoment7}. Thus, we assume that the spatial averages of the initial ionic concentrations are sufficiently small so that \eqref{torusmoment10}, \eqref{torusmoment11}, and \eqref{torusmoment12} are satisfied {for $p=1, 2, 3, 4$}. Under this smallness condition, the decaying-in-time estimate \eqref{torusmoment13} holds for any $t \ge 0$ {and $p \in \left\{1, 2, 3, 4\right\}.$} 

\smallskip
\noindent{\bf{Step 4. Concentrations $H^1$ quadratic moment bounds.}} 
Integrating \eqref{torusmoment15} in time from $0$ to $t$ and applying the expectation $\E$, we deduce the existence of a positive constant $C$ depending on the parameters of the problem, the $L^2$ norms of $f$ and $g$, and some universal constants such that the moment estimate   
\be 
\E \int_{0}^{t} \sum\limits_{i=1}^{N} D_i \|\na c_i(s)\|_{L^2}^2 ds
\le C\int_{0}^{t} \E \left(\left(\sum\limits_{i=1}^{N} \|c_i(s) - 
 \bar{c}_i\|_{L^2}^2 \right)^2 + \left(\sum\limits_{i=1}^{N} \|c_i(s) - 
 \bar{c}_i\|_{L^2}^2 \right) \right) ds
\ee holds for any $t \ge 0$. As a consequence of \eqref{torusmoment13}, we deduce that \eqref{torusmoment16} holds, ending the proof of Proposition \ref{torusmoment18}.  
\end{proof}

\beg{rem} The moment bounds in the periodic setting hold for the S-NPNS system in the case of $N$ ionic species with different diffusivities and valences. Indeed, the potential $\Phi$ solving the Poisson equation equipped with periodic boundary conditions satisfies elliptic estimates that yield good control of its $L^4$ norm by the product of the $L^1$ and $L^2$ norms of the charge density $\rho$, for which we have exponential moment bounds. In contrast, this does not hold on bounded domains where $\|\na \Phi\|_{L^4}^4$ is controlled via interpolation by $\|\na \Phi\|_{L^2}^2 \|\rho\|_{L^2}^2$ for which exponential moment bounds are very challenging to obtain. In this latter case, different assumptions are imposed either on the parameters of the problem or the forcing terms, and different techniques are established to address the ergodicity of the model. 
\end{rem}

{
Interested in the existence and regularity of invariant measures for the periodic S-NPNS system, we seek higher-order Sobolev moment bounds. For that purpose, we also need the following product moment estimates:

\beg{lem} \label{productmomentbouds} Let $u_0 \in H$  and $c_i(0) \in L^2$ be nonnegative for all $i \in \left\{1, \dots, N \right\}$. Assume $f \in H$ and $g \in H$ are time-independent. Suppose the initial spatial averages of the ionic concentrations are sufficiently small. It holds that
\be \label{torusmoment30}
\E \int_{0}^{t} \|c_i(s) - \bar{c}_i\|_{L^2}^2 \|\na c_i(s)\|_{L^2}^2 ds \le \mathcal{J}_1
\ee 
\be \label{torusmoment31}
\E \int_{0}^{t} \|u(s)\|_{L^2}^2 \|\na u(s)\|_{L^2}^2 ds \le \mathcal{J}_2 + \mathcal{J}_3 t,
\ee for any $t \ge 0$, where $\mathcal{J}_1$ and $\mathcal{J}_2$ are nonnegative constants depending on the valences, diffusivities, the number of ionic species, the $L^2$ norm of the initial velocity and concentrations, the body forces $f$, the noise $g$, and some universal constants such that $\mathcal{J}_1 = \mathcal{J}_2 = 0$ when $c_i(0) = \bar{c}_i$ and $u_0 = 0$, whereas $J_3$ is a nonnegative constant depending on the valences, diffusivities, the number of ionic species, the body forces $f$, the noise $g$, and some universal constants.
\end{lem}

\beg{proof} We multiply the energy evolution \eqref{torusmoment15} by the sum $\sum\limits_{i=1}^{N} \|c_i - \bar{c}_i\|_{L^2}^2$, 
integrate in time the resulting differential inequality from $0$ to $t$, apply the expectation $\E$, make use of the decaying-in-time estimates \eqref{torusmoment13} for $p = 2,3$, and deduce \eqref{torusmoment30}. 

The fourth power of the $L^2$ norm of the velocity satisfies the stochastic equation
\be
\beg{aligned}
 &d \|u\|_{L^2}^4
+ 4\|u\|_{L^2}^2 \|\na u\|_{L^2}^2 dt
\\&= -4 \|u\|_{L^2}^2 (\rho \na \Phi, u)_{L^2} dt
+ 4\|u\|_{L^2}^2(f,u)_{L^2} dt
+  2\|u\|_{L^2}^2 \|g\|_{L^2}^2 dt
+ 4(g,u)_{L^2}^2 dt
+ 4\|u\|_{L^2}^2 (g,u)_{L^2} dW.
\end{aligned}
\ee We bound the nonlinear
term in $\rho$ as follows,
\be 
\beg{aligned}
4 \|u\|_{L^2}^2  |(\rho \na \Phi, u)_{L^2}|
&\le C\|u\|_{L^2}^2 \|\rho\|_{L^2} \|\na \Phi\|_{L^4} \|u\|_{L^4}
\le C \|\rho\|_{L^2}^2 \|u\|_{L^2}^2 \|\na u\|_{L^2}
\\&\le C \|\rho\|_{L^2}^2 \|u\|_{L^2}^{\fr{3}{2}} \|\na u\|_{L^2}^{\fr{3}{2}}
\le \fr{1}{2} \|u\|_{L^2}^2 \|\na u\|_{L^2}^2 + C\|\rho\|_{L^2}^8.
\end{aligned}
\ee The above estimates are based on interpolation, elliptic regularity, and applications of the Poincar\'e inequality to the mean-free velocity vector field $u$. As a consequence, we obtain the differential inequality
\be 
d\|u\|_{L^2}^4
+ \|u\|_{L^2}^2 \|\na u\|_{L^2}^2
\le C\|\rho\|_{L^2}^8 + C\|f\|_{L^2}^4 + C\|g\|_{L^2}^4+ 4\|u\|_{L^2}^2 (g,u)_{L^2} dW,
\ee which gives rise to the moment estimate \eqref{torusmoment31} after making use of the decaying bound \eqref{torusmoment13} for $p=4$.
\end{proof}
}

\beg{prop} \label{torussmoothness1} Fix an integer $m \ge 1$. Let $u_0 \in H^m \cap H$ and $c_i(0) \in H^m$ be nonnegative for all $i \in \left\{1, \dots, N \right\}$. Assume $f \in H$ and $g \in H$ are time-independent. Suppose the initial spatial averages of the ionic concentrations are sufficiently small. Then the logarithmic moment bound 
\be \label{torusmoment20}
\E\int_{0}^{t} \log \left(1 + \|\l^{m+1} u(s)\|_{L^2}^2 + \sum\limits_{i=1}^{N} D_i \|\l^{m+1} (c_i - \bar{c}_i)(s)\|_{L^2}^2 \right) ds  
\le \mathcal{J}_{0,m} + \mathcal{J}_{1,m}t
\ee holds for all $t \ge 0$, where $\mathcal{J}_{0,m}$ is a nonnegative constant depending on the valences, diffusivities, the number of ionic species, the $H^m$ norm of the initial velocity and concentrations, and some universal constants, whereas $\mathcal{J}_{1,m}$ is a positive constant depending on the valences, the diffusivities, the number of species, the $H^{m}$ norm of the body forces, the $H^{m}$ norm of the noise $g$, and some universal constants. 
\end{prop}

\begin{proof} We present a proof by induction. The base step $(m=1)$ follows along the lines of the proof of Proposition \ref{boundedsmooth}. Supposing that the statement of Proposition \ref{torussmoothness1}  holds at the $(m-1)$-th iteration, we show that it remains true at the subsequent $m$-th level. For that purpose, we address the stochastic evolution of the energies $\|\l^m (c_i - \bar{c}_i)\|_{L^2}^2$ and $\|\l^m u\|_{L^2}^2$.
Indeed, we have
\be 
\beg{aligned}
&\fr{d}{dt} \|\l^m (c_i - \bar{c}_i)\|_{L^2}^2 
+ 2D_i \|\l^{m+1} (c_i - \bar{c}_i)\|_{L^2}^2 
\\&= - 2(u \cdot \na c_i, \l^{2m}(c_i - \bar{c}_i))_{L^2}
+ 2D_i z_i (\na \cdot ((c_i - \bar{c}_i) \na \Phi), \l^{2m} (c_i - \bar{c}_i))_{L^2} 
+ 2D_i z_i \bar{c}_i (\Delta \Phi, \l^{2m}(c_i - \bar{c}_i))_{L^2}
\end{aligned}
\ee and
\be 
\beg{aligned}
&d \|\l^m u\|_{L^2}^2
+ 2\|\l^{m+1} u\|_{L^2}^2 dt
\\&=- 2(u \cdot \na u, \l^{2m} u)_{L^2} dt
- 2(\rho \na \Phi, \l^{2m} u)_{L^2} dt
+ 2(f, \l^{2m} u)_{L^2} dt
+ \|\l^m g\|_{L^2}^2 dt
- 2(g, \l^{2m} u)_{L^2} dW.
\end{aligned}
\ee For each $t \ge 0$, we consider the instantaneous stochastic processes 
\be 
\mathcal{X}_m(t)= \|\l^m u(t)\|_{L^2}^2 + \sum\limits_{i=1}^{N} \|\l^m (c_i - \bar{c}_i)(t)\|_{L^2}^2 \quad \text{and} \quad
\mathcal{Y}_m(t) = 2\|\l^{m+1} u(t)\|_{L^2}^2 + 2\sum\limits_{i=1}^{N} D_i \|\l^{m+1} (c_i - \bar{c}_i)(t)\|_{L^2}^2,
\ee and we note that the stochastic evolution of $\log (1 + \mathcal{X}_m)$ is described by
\be 
\beg{aligned}
d \log (1 + \mathcal{X}_m)
+ \frac{\mathcal{Y}_m}{1 + \mathcal{X}_m} dt
=\frac{\mathcal{A}_m}{1 + \mathcal{X}_m} dt - \fr{2(g, \l^{2m}u)_{L^2}^2}{(1+\mathcal{X}_m)^2} dt - \frac{2(g, \l^{2m} u)_{L^2}}{1 + \mathcal{X}_m} dW
\end{aligned}
\ee where
\be 
\beg{aligned}
\mathcal{A}_m 
&= - 2\sum\limits_{i=1}^{N} (u \cdot \na c_i, \l^{2m}(c_i - \bar{c}_i))_{L^2}
+ 2\sum\limits_{i=1}^{N} D_i z_i (\na \cdot ((c_i - \bar{c}_i) \na \Phi), \l^{2m} (c_i - \bar{c}_i))_{L^2} 
\\&\quad\quad+ 2\sum\limits_{i=1}^{N} D_i z_i \bar{c}_i (\Delta \Phi, \l^{2m}(c_i - \bar{c}_i))_{L^2}
- 2(u \cdot \na u, \l^{2m} u)_{L^2} 
- 2(\rho \na \Phi, \l^{2m} u)_{L^2} 
\\&\quad\quad+ 2(f, \l^{2m} u)_{L^2} 
+ \|\l^m g\|_{L^2}^2.
\end{aligned}
\ee 
In the sequel, we will prove that the estimate 
\be \label{torussmoothness3}
\fr{\mathcal{A}_m}{1 + \mathcal{X}_m}
\le \fr{1}{2} \fr{\mathcal{Y}_m}{1 + \mathcal{X}_m} + 
 {C \|u\|_{L^2}^2 \|\na u\|_{L^2}^2 + C\sum\limits_{i=1}^{N} \|c_i - \bar{c}_i\|_{L^2}^2 \|\na c_i\|_{L^2}^2}  + C\|\l^m g\|_{L^2}^2 + C\|\l^{m-1} f\|_{L^2}^2 + C \sum\limits_{i=1}^{N} D_i z_i^2 \bar{c}_i^2
\ee holds and deduce the differential inequality
\be 
\beg{aligned}
&d \log (1 + \mathcal{X}_m)
+ \fr{1}{2} \frac{\mathcal{Y}_m}{1 + \mathcal{X}_m} dt
\leq  
 {C \|u\|_{L^2}^2 \|\na u\|_{L^2}^2dt + C\sum\limits_{i=1}^{N} \|c_i - \bar{c}_i\|_{L^2}^2 \|\na c_i\|_{L^2}^2 dt} 
 \\&\quad\quad\quad\quad+  C\|\l^m g\|_{L^2}^2dt + C\|\l^{m-1} f\|_{L^2}^2dt + C \sum\limits_{i=1}^{N} D_i z_i^2 \bar{c}_i^2 dt - \frac{2(g, \l^{2m} u)_{L^2}}{1 + \mathcal{X}_m} dW,
\end{aligned}
\ee from which we obtain the bound 
\be \label{torussmoothness6}
\E \int_{0}^{t} \frac{\mathcal{Y}_m}{1 + \mathcal{X}_m} ds \le C_1(\|\l^m u_0\|_{L^2}, \|\l^m (c_i - \bar{c}_i(0))\|_{L^2}, f, g) + C_2(f, g)t
\ee after integrating in time from $0$ to $t$, applying $\E$, and making use of {Lemma \ref{productmomentbouds}}. We proceed to prove the estimate \eqref{torussmoothness3}. 
Integrating by parts, exploiting the divergence-free property of $u$, applying the product estimate {
\be 
\beg{aligned}
\|\l^m(h_1h_2)\|_{L^2} 
&\le C\|h_1\|_{L^{4}} \|\l^m h_2\|_{L^4} + \|h_2\|_{L^{4}} \|\l^m h_1\|_{L^4}
\\&\le  C\|h_1\|_{L^{2}}^{\fr{1}{2}} \|\na h_1\|_{L^2}^{\fr{1}{2}} \|\l^m h_2\|_{L^2}^{\fr{1}{2}} \|\l^{m+1} h_2\|_{L^2}^{\fr{1}{2}} 
+ \|h_2\|_{L^{2}}^{\fr{1}{2}} \|\na h_2\|_{L^2}^{\fr{1}{2}} \|\l^m h_1\|_{L^2}^{\fr{1}{2}}\|\l^{m+1} h_1\|_{L^2}^{\fr{1}{2}}
\end{aligned}
\ee} that holds for any mean-free functions $h_1$ and $h_2$,
we deduce the following bound
\be \label{torussmoothness4}
\beg{aligned}
&\frac1{1 + \mathcal{X}_m}\left| 2\sum\limits_{i=1}^{N} (u \cdot \na c_i, \l^{2m}(c_i - \bar{c}_i))_{L^2} \right| 
\le \frac{C\sum\limits_{i=1}^{N}\|\l^{m+1}(c_i - \bar{c}_i)\|_{L^2} \|\l^{m-1} \na \cdot (u(c_i - \bar{c}_i))\|_{L^2}}{1 + \mathcal{X}_m} 
\\&\le \fr{1}{12}\fr{\mathcal{Y}_m}{1 + \mathcal{X}_m} + \frac{C \sum\limits_{i=1}^{N}\left(\|\l^m u\|_{L^2}^2 {\|c_i - \bar{c}_i\|_{L^2}^2 \|\na c_i\|_{L^2}^2} + \|\l^m (c_i - \bar{c}_i)\|_{L^2}^2 { \|u\|_{L^2}^2 \|\na u\|_{L^2}^2} \right)}{1 + \mathcal{X}_m}
\\&\le \fr{1}{12}\fr{\mathcal{Y}_m}{1 + \mathcal{X}_m} + C \sum\limits_{i=1}^{N} { \|c_i - \bar{c}_i\|_{L^2}^2 \|\na c_i\|_{L^2}^2} + C\|u\|_{L^2}^2 \|\na u\|_{L^2}^2.
\end{aligned}
\ee Using, in addition, the elliptic regularity obeyed by the potential $\Phi$, we have
\be 
\beg{aligned}
&\frac1{1 + \mathcal{X}_m}\left|2\sum\limits_{i=1}^{N} D_i z_i (\na \cdot ((c_i - \bar{c}_i) \na \Phi), \l^{2m} (c_i - \bar{c}_i))_{L^2} \right|
\\
\le &\fr{1}{12}\fr{\mathcal{Y}_m}{1 + \mathcal{X}_m} + \frac{C \sum\limits_{i=1}^{N}\left(\|\l^m (c_i - \bar{c}_i)\|_{L^2}^2 { \|\na \Phi\|_{L^{2}}^2 \|\rho\|_{L^2}^2 } + \|\l^m \na \Phi\|_{L^2}^2 { \|c_i - \bar{c}_i\|_{L^{2}}^2 \|\na c_i\|_{L^2}^2 }\right)}{1 + \mathcal{X}_m}
\\
\le &\fr{1}{12}\fr{\mathcal{Y}_m}{1 + \mathcal{X}_m} + C \sum\limits_{i=1}^{N} { \|c_i - \bar{c}_i\|_{L^2}^2 \|\na c_i\|_{L^2}^2}.
\end{aligned}
\ee Integration by parts and the use of the Poisson equation obeyed by $\Phi$ yield
\be 
\frac1{1 + \mathcal{X}_m}\left| 2\sum\limits_{i=1}^{N} D_i z_i \bar{c}_i (\Delta \Phi, \l^{2m}(c_i - \bar{c}_i))_{L^2}\right| \le \fr{1}{12}\fr{\mathcal{Y}_m}{1 + \mathcal{X}_m} + C \sum\limits_{i=1}^{N} D_i z_i^2 \bar{c}_i^2.
\ee We apply again standard product estimates and continuous Sobolev embeddings to estimate the nonlinear terms
\be 
\frac1{1 + \mathcal{X}_m}\left|2(u \cdot \na u, \l^{2m} u)_{L^2} \right| 
\le \fr{1}{12}\fr{\mathcal{Y}_m}{1 + \mathcal{X}_m} + { C\|u\|_{L^2}^2 \|\na u\|_{L^2}^2},
\ee and
\be \label{torussmoothness5}
\frac1{1 + \mathcal{X}_m}\left|2(\rho \na \Phi, \l^{2m} u)_{L^2}\right| \le  \fr{1}{12}\fr{\mathcal{Y}_m}{1 + \mathcal{X}_m} + C \sum\limits_{i=1}^{N} { \|c_i - \bar{c}_i\|_{L^2}^2 \|\na c_i\|_{L^2}^2}.
\ee 
Putting \eqref{torussmoothness4}--\eqref{torussmoothness5} together, we obtain the desired bound \eqref{torussmoothness3} for $\mathcal{A}_m$ from which \eqref{torussmoothness6} follows. 

In view of the logarithmic estimate
\be 
E \int_{0}^{t} \log (1 + \mathcal{Y}_m) ds \le \E \int_{0}^{t} \frac{\mathcal{Y}_m}{1 + \mathcal{X}_m} ds + \E \int_{0}^{t} \log (1 + \mathcal{X}_m) ds
\ee and the induction hypothesis, we infer that Proposition \ref{torussmoothness1} holds at the $m$-th regularity stage. 
\end{proof}

\section{The Two-Species Model: Exponential Ergodicity} \label{two}

In this section, we consider the periodic S-NPNS model for two ionic species with equal diffusivities $D$ and valences $1$ and $-1$ respectively. This model is described by the system of equations
\noeqref{twospeciescharge}
\begin{subequations}\label{sys:npns-2species}
    \begin{align}
        &d u + u \cdot \na u dt - \Delta u dt + \na p dt = - \rho \na \Phi dt + f dt + g dW, \label{twospeciesvel}
        \\
        &\pa_t c_1 + u \cdot \na c_1 - D\Delta c_1 = D \na \cdot (c_1 \na \Phi),\label{twospeciescon1}
        \\
        &\pa_t c_2 + u \cdot \na c_2 - D \Delta c_2 = -D \na \cdot (c_2 \na \Phi), \label{twospeciescon2}
        \\
        &-\Delta \Phi = \rho = c_1 - c_2, \label{twospeciescharge}
        \\
        &\na \cdot u = 0. \label{twospeciesdiv}
    \end{align}
\end{subequations}
on $\TT^2$ with periodic boundary conditions. We define $\sigma$ to be the sum of the two ionic concentrations, that is
$\sigma = c_1 + c_2,$
and we denote its spatial average by $\bar{\sigma}$. By adding and subtracting \eqref{twospeciescon1} and \eqref{twospeciescon2}, we obtain the equations obeyed by $\rho$ and $\sigma$ as follows:
\begin{align} \label{twospeciesrho}
&\pa_t \rho + u \cdot \na \rho - D \Delta \rho 
= D \na \cdot (\sigma \na \Phi),\\
\label{twospeciessigma}
&\pa_t \sigma + u \cdot \na \sigma - D \Delta \sigma 
= D \na \cdot (\rho \na \Phi).
\end{align}

The following theorem is the main result of this section.


\beg{Thm} \label{maintheoremtwospecies} 
Suppose that Setting \ref{setting D} holds. There exists an ergodic invariant probability measure $\pi$ for the Markov semigroup \eqref{markovperiodic} associated with the periodic two-species S-NPNS problem \eqref{sys:npns-2species}. If $D$ is sufficiently large, then there exists an integer $n:= n(f, g)$ depending only on the body forces $f$, the noise $g$, the parameters of the problem, and some universal constants such that if $P_n H \subset range(g)$, then the invariant measure is unique. Moreover, there exist positive constants $r, C >0$ such that the decaying-in-time estimate
\be 
W_{\|\omega - \tilde{\omega}\|_{\mathcal{H}} \wedge  1} (P_t(\omega, \cdot), \pi)
\le C(1  + \|\omega\|_{\mathcal{H}}) e^{-rt},
\ee holds, where $W$ denotes the Wasserstein metric and is defined by \eqref{wasserstein}.
\end{Thm}

The proof of Theorem \ref{maintheoremtwospecies} follows from the generalized coupling framework presented in Appendix \ref{ExpErgFramework} and Propositions \ref{prop:twospecies-1} and \ref{prop:twospecies-2} below.

\beg{prop} \label{prop:twospecies-1}
Let $(u_0, c_1(0), c_2(0)) \in \tilde{\mathcal{H}}_{per}.$ For each time $t \ge 0$, we define the energies
\be 
Y_{E}(t) = \|u(t)\|_{L^2}^2 + \|\rho(t)\|_{L^2}^2 + \|\sigma(t) - 
\bar{\sigma}\|_{L^2}^2 + \|\na \Phi(t)\|_{L^2}^2
\ee and 
\be 
Y_{D}(t) = \|\na u\|_{L^2}^2 + 2D \left(\|\na \rho(t)\|_{L^2}^2 + \|\na \sigma(t)\|_{L^2}^2 + \|\rho(t)\|_{L^2}^2 + \|\sqrt{\sigma} \rho (t)\|_{L^2}^2 +  \|\sqrt{\sigma} \na \Phi (t)\|_{L^2}^2 \right).
\ee 
The following stochastic inequality holds, for all times $t \ge 0,$ 
\be \label{twospecies5}
Y_{E}(t) + \int_{0}^{t} Y_{D}(s) ds 
\le Y_{E}(0) + C\|g\|_{L^2}^2 t + C \|f\|_{L^2}^2 t + 2(g, u)_{L^2} dW.
\ee
\end{prop}

\begin{proof}
The stochastic evolution of the $L^2$ norm of the velocity $u$ is described by the stochastic equation
\be \label{twospecies1}
d \|u\|_{L^2}^2 + 2 \|\na u\|_{L^2}^2 dt
= -2(\rho \na \Phi, u)_{L^2} dt + 2(f,u)_{L^2}^2 dt 
+ \|g\|_{L^2}^2 dt
+ 2(g,u)_{L^2}dW.
\ee Taking the scalar product in $L^2$ of the charge density equation \eqref{twospeciesrho} obeyed by $\rho$ with $\Phi$ brings 
\be \label{twospecies2}
\fr{1}{2} \fr{d}{dt} \|\na \Phi\|_{L^2}^2 + D\|\rho\|_{L^2}^2 = - (u \cdot \na \rho, \Phi)_{L^2} - D (\sigma \na \Phi, \na \Phi)_{L^2}.
\ee Adding \eqref{twospecies1} and \eqref{twospecies2}, using the cancellation 
\be 
(\rho \na \Phi, u)_{L^2}  + (u \cdot \na \rho, \Phi)_{L^2}= 0 ,
\ee and applying the Cauchy-Schwarz inequality, we obtain 
\be \label{twospecies4}
\beg{aligned}
&d \left(\|u\|_{L^2}^2 + \|\na \Phi\|_{L^2}^2\right)
+ \left(\|\na u\|_{L^2}^2 + D\|\rho\|_{L^2}^2 + D\|\sqrt{\sigma} \na \Phi\|_{L^2}^2  \right) dt 
\le  \left(C\|f\|_{L^2}^2 + \|g\|_{L^2}^2\right) dt + 2(g,u)_{L^2} dW.
\end{aligned}
\ee Finally, we take the $L^2$ inner product of the equation \eqref{twospeciesrho} obeyed by $\rho$ with $\rho$ and the equation \eqref{twospeciessigma} obeyed by $\sigma$ with $\sigma$ and add them. In view of the cancellations
\be 
(u \cdot \na \rho, \rho)_{L^2} = (u \cdot \na \sigma, \sigma)_{L^2} = 0,
\ee 
\be 
\beg{aligned}
(\na \cdot (\sigma \na \Phi), \rho)_{L^2} + (\na \cdot (\rho \na \Phi), \sigma)_{L^2}
&= -(\sigma \na \Phi, \na \rho)_{L^2} + (\na \rho \cdot \na \Phi, \sigma)_{L^2}
+ (\rho \Delta \Phi, \sigma)_{L^2}
= - (\rho^2, \sigma)_{L^2},
\end{aligned}
\ee we obtain the deterministic energy equality
\be \label{twospecies3}
\fr{1}{2} \fr{d}{dt} \left(\|\rho\|_{L^2}^2 + \|\sigma - 
\bar{\sigma}\|_{L^2}^2 \right)
+ D \|\na \rho\|_{L^2}^2 + D\|\na \sigma\|_{L^2}^2
+ D\|\sqrt{\sigma} \rho\|_{L^2}^2 = 0,
\ee after making use of the identity
\be 
\fr{d}{dt} \|\sigma - \bar{\sigma}\|_{L^2}^2 = \fr{d}{dt} \|\sigma\|_{L^2}^2 -2 \fr{d}{dt} (\sigma, \bar{\sigma})_{L^2}
= \fr{d}{dt} \|\sigma\|_{L^2}^2 - 2\bar{\sigma} \fr{d}{dt} \int_{\TT^2} \sigma dx 
= \fr{d}{dt} \|\sigma\|_{L^2}^2.
\ee We add \eqref{twospecies4} and \eqref{twospecies3}, integrate in time, and deduce \eqref{twospecies5}.
\end{proof}

\beg{prop} \label{prop:twospecies-2}
Let $(u_0, c_1(0), c_2(0))$ and $(\tilde{u}_0, \tilde{c}_1(0), \tilde{c_2}(0)) \in \tilde{\mathcal{H}}_{per}$.  Let $\lambda > 0$ be a positive constant.  We denote by $(u(t), c_1(t), c_2(t))$ the solution to the two-species S-NPNS system \eqref{sys:npns-2species} with initial data $(u_0, c_1(0), c_2(0))$, and by $(\tilde{u}(t), \tilde{c}_1(t), \tilde{c}_2(t))$ the solution to the modified system
\noeqref{twospeciesmodified2}
\noeqref{twospeciesmodified3}
\noeqref{twospeciesmodified4}
\begin{subequations}\label{sys:npns-modified}
    \begin{align}
        \label{twospeciesmodified1}
&d \tilde{u} + \tilde{u} \cdot \na \tilde{u} dt- \Delta \tilde{u} dt + \na \tilde{p} dt= - \tilde{\rho} \na \tilde{\Phi} dt 
+ f dt 
+ \lambda P_n(u - \tilde{u}) dt
+ g dW
\\
\label{twospeciesmodified2}
&\pa_t \tilde{c}_1 + \tilde{u} \cdot \na \tilde{c}_1 - D\Delta \tilde{c}_1 = D\na \cdot (\tilde{c}_1 \na \tilde\Phi),
\\
\label{twospeciesmodified3}
&\pa_t \tilde{c}_2 + \tilde{u} \cdot \na \tilde{c}_2 - D\Delta \tilde{c}_2 = -D\na \cdot (\tilde{c}_2 \na \tilde\Phi),
\\
\label{twospeciesmodified4}
&-\Delta \tilde{\Phi} = \tilde{\rho} = \tilde{c}_1 - \tilde{c}_2,
\\
\label{twospeciesmodified5}
&\na \cdot \tilde{u} = 0,
    \end{align}
\end{subequations}
with initial data $(\tilde{u}_0, \tilde{c}_1(0), \tilde{c}_2(0))$. Letting $\tilde{\sigma} = \tilde{c}_1 + \tilde{c}_2$, we define the instantaneous energy
\be \label{twospeciesmodified6}
\mathcal{Q}(t) = \|u(t) - \tilde{u}(t)\|_{L^2}^2
+ \|\rho(t) - \tilde{\rho}(t)\|_{L^2}^2
+ \|\sigma(t) - \tilde{\sigma}(t)\|_{L^2}^2
+ \|\na (\Phi - \tilde{\Phi})(t)\|_{L^2}^2
\ee at a positive time $t$. We denote by $\lambda_{n+1}$ the $(n+1)$ eigenvalue of the periodic Laplacian. Then there exist positive constants $c, C>0$ such that the following dissipativity bound
\be \label{twospecieserg23}
\|\mathcal{Q}(t)\|_{L^2}^2
\le \|\mathcal{Q}(0)\|_{L^2}^2
e^{- \min \left\{cD, \lambda_{n+1} \right\}t + C\int_{0}^{t} \left(\|\na u(s)\|_{L^2}^2 + \|\na \rho(s)\|_{L^2}^2 +\|\sigma\|_{L^2}^2 + \|\na \sigma(s)\|_{L^2}^2 \right) ds}
\ee holds for any $t \ge 0$, provided that $\lambda \ge \lambda_{n+1}.$
\end{prop}

\begin{proof}
    We define the differences
    \be 
U = u - \tilde{u}, \quad R = \rho - \tilde{\rho},\quad S = \sigma - \tilde{\sigma}, \quad\Psi = \Phi - \tilde{\Phi}, \quad P = p - \tilde{p}.
    \ee 
    These differences satisfy the following system of equations
\noeqref{twospecieserg4}
\noeqref{twospecieserg5-1}
\begin{subequations}\label{sys:2species-difference}
    \begin{align}
        \label{twospecieserg1}
&\pa_t U - \Delta U + \na P + \lambda P_n U 
= - U \cdot \na u - \tilde{u} \cdot \na U
- R \na \Phi - \tilde{\rho} \na \Psi,
\\
\label{twospecieserg2}
&\pa_t R - D \Delta R 
= - U \cdot \na \rho - \tilde{u} \cdot \na R
+ D \na \cdot (S \na \Phi) + D \na \cdot (\tilde{\sigma} \na \Psi),
   \\
   \label{twospecieserg3}
&\pa_t S - D \Delta S 
= - U \cdot \na \sigma - \tilde{u} \cdot \na S
+ D \na \cdot (R \na \Phi) + D \na \cdot (\tilde{\rho} \na \Psi),
    \\ \label{twospecieserg4}
&- \Delta \Psi = R,
    \\ \label{twospecieserg5-1}
&\na \cdot U = 0.
    \end{align}
\end{subequations}

\noindent\textbf{Step 1. $L^2$ evolution of $R$ and $S$.} We take scalar the scalar products in $L^2$ of the equations \eqref{twospecieserg2} and \eqref{twospecieserg3} obeyed by $R$ and $S$ with $R$ and $S$ respectively and add them. We then obtain the energy equality
\be \label{twospecieserg5}
\beg{aligned}
&\fr{1}{2} \fr{d}{dt} \left(\|R\|_{L^2}^2 + \|S\|_{L^2}^2\right) + D\left(\|\na R\|_{L^2}^2 + \|\na S\|_{L^2}^2 \right) 
\\&\quad\quad= -(U \cdot \na \rho, R)_{L^2} 
- (U \cdot \na \sigma, S)_{L^2} 
+ D(\na \cdot (S \na \Phi), R)_{L^2} 
\\&\quad\quad\quad\quad+ D (\na \cdot (\tilde{\sigma} \na \Psi), R)_{L^2} + D (\na \cdot (R \na \Phi), S)_{L^2} 
+ D(\na \cdot (\tilde{\rho} \na \Psi), S)_{L^2}. 
\end{aligned}
\ee 
Integrating by parts and using the positivity of $\tilde{\sigma}$ produce
\be 
(\na \cdot (S \na \Phi), R)_{L^2} 
+ (\na \cdot (R \na \Phi), S)_{L^2} 
= - (S \na \Phi, \na R)_{L^2} 
+ (\na R \cdot  \na \Phi, S)_{L^2}
+ (R \Delta \Phi, S)_{L^2}
= - (R \rho, S)_{L^2}
\ee 
and 
\be \label{twospecieserg6}
\beg{aligned}
&(\na \cdot (\tilde{\sigma} \na \Psi), R)_{L^2}
+ (\na \cdot (\tilde{\rho} \na \Psi), S)_{L^2}
= - (\tilde{\sigma} \na \Psi, \na R)_{L^2}
+ (\na \tilde{\rho} \cdot \na \Psi, S)_{L^2}
+ (\tilde{\rho} \Delta \Psi, S)_{L^2}
\\&= (S \na \Psi, \na R)_{L^2} - (\sigma \na \Psi, \na R)_{L^2}
+ (\na \rho \cdot \na \Psi, S)_{L^2}
- (\na R \cdot \na \Psi, S)_{L^2}
+ (\rho \Delta \Psi, S)_{L^2}
- (R \Delta \Psi, S)_{L^2}
\\&= -(\sigma \na \Psi, \na R)_{L^2}
+ (\na \rho \cdot\na \Psi, S)_{L^2}
- (\rho R, S)_{L^2}
+ (R^2, S)_{L^2}
\\&= -(\sigma \na \Psi, \na R)_{L^2}
+ (\na \rho \cdot\na \Psi, S)_{L^2}
- (\rho R, S)_{L^2}
+ (R^2, \sigma)_{L^2}
- (R^2, \tilde{\sigma})_{L^2}
\\&\le -(\sigma \na \Psi, \na R)_{L^2}
+ (\na \rho \cdot\na \Psi, S)_{L^2}
- (\rho R, S)_{L^2}
+ (R^2, \sigma)_{L^2}.
\end{aligned}
\ee Putting \eqref{twospecieserg5}--\eqref{twospecieserg6} together, we deduce the differential inequality
\be 
\beg{aligned}
&\fr{1}{2} \fr{d}{dt} \left(\|R\|_{L^2}^2 + \|S\|_{L^2}^2\right) + D\left(\|\na R\|_{L^2}^2 + \|\na S\|_{L^2}^2 \right) 
\\&\quad\quad\le -(U \cdot \na \rho, R)_{L^2} 
- (U \cdot \na \sigma, S)_{L^2} 
- 2D(\rho, RS)_{L^2} - D(\sigma, \na \Psi \cdot \na R)_{L^2}
+ D(\na \rho, S\na \Psi)_{L^2}
+ D(\sigma, R^2)_{L^2},
\end{aligned} 
\ee whose all the nonlinear terms depend solely on the difference $R, S, U$ and the solution to 
 the S-NPNS system \eqref{sys:npns-2species}. The nondependency on the solution to the modified system \eqref{sys:npns-modified} is crucial to obtain the desired estimates. As $U$, $R$, and $S$ are mean-free, we have 
\be 
\beg{aligned}
&\fr{1}{2} \fr{d}{dt} \left(\|R\|_{L^2}^2 + \|S\|_{L^2}^2\right) + D\left(\|\na R\|_{L^2}^2 + \|\na S\|_{L^2}^2 \right) 
\le C\|U\|_{L^2}^{\fr{1}{2}}\|\na U\|_{L^2}^{\fr{1}{2}} \|R\|_{L^2}^{\fr{1}{2}} \|\na R\|_{L^2}^{\fr{1}{2}} \|\na \rho\|_{L^2} 
\\&\quad\quad\quad\quad+ C\|U\|_{L^2}^{\fr{1}{2}}\|\na U\|_{L^2}^{\fr{1}{2}} \|S\|_{L^2}^{\fr{1}{2}} \|\na S\|_{L^2}^{\fr{1}{2}} \|\na \sigma\|_{L^2} 
+2D\|\rho\|_{L^4} \|\na R\|_{L^2} \|S\|_{L^2} 
\\&\quad\quad\quad\quad\quad\quad+ D\|\sigma\|_{L^4} \|R\|_{L^2} \|\na R\|_{L^2}
+ D\|\na \rho\|_{L^2} \|S\|_{L^2} \|\na R\|_{L^2}
+ D\|\sigma\|_{L^2}\|R\|_{L^2}\|\na R\|_{L^2},
\end{aligned} 
\ee due to interpolation inequalities, the Poincar\'e inequality, and elliptic regularity. A straightforward application of Young's inequality gives rise to 
\be \label{twospecieserg20}
\beg{aligned}
& \fr{d}{dt} \left(\|R\|_{L^2}^2 + \|S\|_{L^2}^2\right) 
+ \fr{3D}{2} \left(\|\na R\|_{L^2}^2 + \|\na S\|_{L^2}^2 \right) 
\\&\quad\quad\le\fr{1}{2} \|\na U\|_{L^2}^2 + C\left(\|\sigma\|_{L^2}^2 + \|\na \sigma\|_{L^2}^2 + \|\na \rho\|_{L^2}^2  \right) \left(\|U\|_{L^2}^2 + \|R\|_{L^2}^2 + \|S\|_{L^2}^2 \right).
\end{aligned} 
\ee

\smallskip
\noindent\textbf{Step 2. $L^2$ evolution of $U$ and $\na \Phi$.}  
    We take the $L^2$ inner product of the equations \eqref{twospecieserg1} and \eqref{twospecieserg2} obeyed by $U$ and $R$ with $U$ and $\Psi$ respectively. We add them and obtain 
    \be 
\beg{aligned}
&\fr{1}{2} \fr{d}{dt} \left(\|U\|_{L^2}^2 + \|\na \Psi\|_{L^2}^2 \right) + \|\na U\|_{L^2}^2 + \lambda \|P_n U\|_{L^2}^2 + D\|R\|_{L^2}^2
\\&\quad\quad= -(U \cdot \na u, U)_{L^2} - (R\na \Phi, U)_{L^2} - (\tilde{\rho} \na \Psi, U)_{L^2}
- (U \cdot \na \rho, \Psi)_{L^2} 
\\&\quad\quad\quad\quad- (\tilde{u} \cdot \na R, \Psi)_{L^2}
+ D (\na \cdot (S \na \Phi), \Psi)_{L^2}
+ D(\na \cdot (\tilde{\sigma} \na \Psi, \Psi)_{L^2}.
\end{aligned}
    \ee In view of the divergence-free condition obeyed by $U$, we integrate by parts and deduce the relation 
\be 
\beg{aligned}
&-(\tilde{\rho} \na \Psi, U)_{L^2} - (\tilde{u} \cdot \na R, \Psi)_{L^2}
= (R \na \Psi, U)_{L^2} - (\rho \na \Psi, U)_{L^2} + (U \cdot \na R, \Psi)_{L^2} - (u \cdot \na R, \Psi)_{L^2}
\\&= -(U \cdot \na R, \Psi)_{L^2} - (\rho \na \Psi, U)_{L^2} + (U \cdot \na R, \Psi)_{L^2} - (u \cdot \na R, \Psi)_{L^2}
= - (\rho \na \Psi, U)_{L^2} - (u \cdot \na R, \Psi)_{L^2}.
\end{aligned}
\ee Due to the positivity of the modified concentrations, we have 
\be 
D(\na \cdot (\tilde{\sigma} \na \Psi, \Psi)_{L^2}
= - D(\tilde{\sigma} \na \Psi, \na \Psi)_{L^2} \le 0.
\ee Consequently, we infer that the evolution inequality
\be 
\beg{aligned}
&\fr{1}{2} \fr{d}{dt} \left(\|U\|_{L^2}^2 + \|\na \Psi\|_{L^2}^2 \right) + \|\na U\|_{L^2}^2 + \lambda \|P_n U\|_{L^2}^2 + D\|R\|_{L^2}^2
\\&\quad\quad\le -(U \cdot \na u, U)_{L^2} - (R\na \Phi, U)_{L^2} - (U \cdot \na \rho, \Psi)_{L^2} 
+ D (\na \cdot (S \na \Phi), \Psi)_{L^2} 
- (\rho \na \Psi, U)_{L^2} - (u \cdot \na R, \Psi)_{L^2}
\end{aligned}
\ee holds. We estimate the nonlinearities of the system and obtain
\be 
\beg{aligned}
&\fr{1}{2} \fr{d}{dt} \left(\|U\|_{L^2}^2 + \|\na \Psi\|_{L^2}^2 \right) + \|\na U\|_{L^2}^2 + \lambda \|P_n U\|_{L^2}^2 + D\|R\|_{L^2}^2
\\&\quad\quad\le \|\na u\|_{L^2} \|U\|_{L^4}^2
+ \|R\|_{L^2} \|\na \Phi\|_{L^{\infty}} \|U\|_{L^2} 
+ D\|\na \Psi\|_{L^4} \|\na \Phi\|_{L^4} \|S\|_{L^2}
+ \|u\|_{L^4} \|\na \Psi\|_{L^{4}} \|R\|_{L^2},
\end{aligned}
\ee which yields 
\be 
\beg{aligned}
& \fr{d}{dt} \left(\|U\|_{L^2}^2 + \|\na \Psi\|_{L^2}^2 \right) + \fr{3}{2}\|\na U\|_{L^2}^2 + \lambda \|P_n U\|_{L^2}^2 + D\|R\|_{L^2}^2
\le  C\left(\|\na u\|_{L^2}^2 + \|\na \rho\|_{L^2}^2 \right)\left(\|U\|_{L^2}^2 + \|R\|_{L^2}^2 + {\|S\|_{L^2}^2}\right)
\end{aligned}
\ee after interpolating and employing elliptic estimates. By the generalized Poincar\'e inequality, we have
\be 
\beg{aligned}
&\|\na U\|_{L^2}^2 + \lambda \|P_n U\|_{L^2}^2
\ge \lambda_{n+1}\|U\|_{L^2}^2,
\end{aligned}
\ee provided that $\lambda \ge \lambda_{n+1}$. Thus, we deduce the evolution inequality
\be \label{twospecieserg21}
\beg{aligned}
& \fr{d}{dt} \left(\|U\|_{L^2}^2 + \|\na \Psi\|_{L^2}^2 \right) + \fr{1}{2}\|\na U\|_{L^2}^2 + \lambda_{n+1} \|U\|_{L^2}^2 + D\|R\|_{L^2}^2
\\&\quad\quad\le  C\left(\|\na u\|_{L^2}^2 + \|\na \rho\|_{L^2}^2 \right)\left(\|U\|_{L^2}^2 + \|R\|_{L^2}^2 +{\|S\|_{L^2}^2} \right).
\end{aligned}
\ee 

\smallskip
\noindent\textbf{Step 3. $L^2$ evolution of $\mathcal{Q}$.} Adding \eqref{twospecieserg20} and \eqref{twospecieserg21} and using the Poincar\'e inequality, we obtain
\be 
\fr{d}{dt} \mathcal{Q} + \min \left\{cD, \lambda_{n+1} 
\right\} \mathcal{Q}
\le C\left(\|\na u\|_{L^2}^2 + \|\na \rho\|_{L^2}^2 + \|\sigma\|_{L^2}^2 + \|\na \sigma\|_{L^2}^2 \right) \mathcal{Q},
\ee which gives \eqref{twospecieserg23}.
\end{proof}

\medskip
\section*{Acknowledgment}
R.H. was partially supported by a grant from the Simons Foundation (MP-TSM-00002783). Q.L. was partially supported by the AMS-Simons Travel Grant.

\medskip
\appendix

\section{Elliptic Estimates} \label{Marcin}

In this appendix, we recall the definition of weak Lebesgue spaces, state the Marcinkiewicz interpolation theorem, and use it to prove a new elliptic estimate for solutions to periodic Poisson equations.  

Let $(X, \mu)$ be a measure space. For $0 < p < \infty$, we denote by $L^{p, \infty}(X, \mu)$ the set of all $\mu$-measurable functions $f$ such that 
\be 
\|f\|_{L^{p,\infty}} = \sup \left\{\lambda d_f(\lambda)^{1/p} : \lambda > 0\right\} < \infty,
\ee where
\be 
d_f(\lambda) = \mu (\left\{x \in X: |f(x)| > \lambda  \right\}). 
\ee 
We note that $L^{p, \infty}$ is a quasinormed linear space for $0 < p <\infty$, that is 
\begin{enumerate}
    \item $\|kf\|_{L^{p,\infty}} = |k| \|f\|_{p, \infty}$ for any complex nonzero constant $k$;
    \item $\|f_1 + f_2\|_{L^{p,\infty}} \le C_p \left(\|f_1\|_{L^{p,\infty}} + \|f_2\|_{L^{p,\infty}} \right)$;
    \item If $\|f\|_{L^{p,\infty}} = 0$, then $f = 0$ $\mu$-a.e.
\end{enumerate} We refer the reader to \cite{grafakos2008classical} for a detailed exposition of weak Lebesgue spaces. 

The Marcinkiewicz interpolation theorem states the following: 

\begin{Thm} \cite{grafakos2008classical} Let $(X, \mu)$ and $(Y, \nu)$ be two measure spaces. Let $T$ be a linear operator defined on the set of all simple functions on $X$ and taking values in the set of measurable functions on $Y$. Let $0 < p_0 \ne p_1 \le \infty$ and $0 < q_0 \ne q_1 \le 
\infty$. If $T$ maps $L^{p_0}$ to $L^{q_0, \infty}$ and $L^{p_1}$ to $L^{q_1, \infty}$, and for some $\theta \in (0,1)$ we have 
\be \la{pc}
\fr{1}{p} = \fr{1-\theta}{p_0} + \fr{\theta}{p_1},
\ee 
\be \la{qc}
\fr{1}{q} = \fr{1-\theta}{q_0} + \fr{\theta}{q_1},
\ee and $p \le q$, then 
\be 
\|Tf\|_{L^q} \le C \|f\|_{L^p},
\ee for all functions $f$ in the domain of $T$. By density, $T$ has a unique extension as a bounded operator from $L^{p}(X, \mu)$ to $L^q(Y, \nu)$.
\end{Thm}

Using the Marcinkiewicz interpolation theorem and employing Fourier series techniques, we now study the regularity of solutions to Poisson equations equipped with periodic boundary conditions:

\beg{prop} \label{ellipticprop}
Let $\tilde{\rho} \in L^{2}(\TT^2)$ have a zero spatial average over $\TT^2$ and $\tilde{\Phi}$ be the solution of the Poisson equation 
\be \label{periodicpoisson}
- \Delta \tilde{\Phi} = \tilde{\rho}
\ee with periodic boundary conditions. Then it holds that
\be \label{torusmom}
\|\na \tilde{\Phi}\|_{L^4}
\le C\|\tilde \rho\|_{L^{\fr{4}{3}}}.
\ee 
\end{prop}

\begin{proof} We write the Fourier series of $\tilde \rho$ as 
\be 
\tilde{\rho}  = \sum\limits_{k \in \ZZ^2 \setminus \left\{0\right\}} \tilde{\rho}_k e^{ik \cdot x},
\ee where the Fourier coefficient $\tilde{\rho}_0$ vanishes due to the mean-free property satisfied by $\tilde \rho$.

\smallskip
\noindent{\bf{Step 1.}} We prove the existence of a positive universal constant $C>0$ such that the estimate 
\be 
\|\na \tilde \Phi\|_{L^4(\TT^2)} \le C\left\||k|^{-1} \tilde \rho_k\right\|_{\ell^{\fr{4}{3}}(\ZZ^2 \setminus \left\{0\right\})} 
\ee holds. Indeed, the solution $\tilde{\Phi}$ to \eqref{periodicpoisson} is given by 
\be 
\tilde \Phi = \l^{-2} \tilde \rho,
\ee where $\l:= \sqrt{-\Delta}$ is the square root of the 2D periodic Laplacian subject to periodic boundary conditions. 
Thus, its gradient is controlled in $L^4(\TT^2)$ by 
\be 
\|\na \tilde \Phi\|_{L^4(\TT^2)}
\le  C\|\l^{-1} \tilde \rho\|_{L^4(\TT^2)},
\ee due to the boundedness of the periodic Riesz transform $\na \l^{-1}$ on $L^4(\TT^2)$. We fix $\psi \in L^{4/3}(\TT^2)$ and write its Fourier series as
\be 
\psi = \sum\limits_{j \in \ZZ^2} \psi_j e^{ij \cdot x}.
\ee Thus, we have 
\be 
(\l^{-1} \tilde  \rho, \psi)_{L^2}
= (2\pi)^2 \sum\limits_{k \in \ZZ^2 \setminus \left\{0\right\}} |k|^{-1} \tilde{\rho}_k \psi_{-k} 
\ee by Parseval's identity. A direct application of H\"older's inequality yields 
\be 
|(\l^{-1} \tilde  \rho, \psi)_{L^2}|
\le C\|\psi_k\|_{\ell^4(\ZZ^2)} \left\||k|^{-1} \tilde \rho_k \right\|_{\ell^{\fr{4}{3}}(\ZZ^2 \setminus \left\{0\right\})}.
\ee In view of the Hausdorff-Young inequality, we have
\be 
\|\psi_k\|_{\ell^4(\ZZ^2)}
\le C\|\psi\|_{L^{\fr{4}{3}}(\TT^2)},
\ee and thus, 
\be 
|(\l^{-1} \tilde  \rho, \psi)_{L^2}|
\le C\|\psi\|_{L^{\fr{4}{3}}(\TT^2)} \left\||k|^{-1} \tilde \rho_k \right\|_{\ell^{\fr{4}{3}}(\ZZ^2 \setminus \left\{0\right\})}.
\ee Taking the supremum over all functions $\psi \in L^{\fr{4}{3}}(\TT^2)$ with $\|\psi\|_{L^{\fr{4}{3}}(\TT^2)} \le 1$ gives 
\be 
\|\l^{-1}\tilde \rho\|_{L^4(\TT^2)} \le C  \left\||k|^{-1} \tilde \rho_k \right\|_{\ell^{\fr{4}{3}}(\ZZ^2 \setminus \left\{0\right\})},
\ee completing the proof of Step 1. 

\smallskip
\noindent{\bf{Step 2.}} We prove the existence of a positive universal constant $C$ such that 
\be 
\left\||k|^{-1} \tilde{\rho}_k \right\|_{\ell^{\fr{4}{3}}(\ZZ^2 \setminus \left\{0\right\})}
\le C\|\tilde \rho\|_{L^{\fr{4}{3}}(\TT^2)}
\ee holds. We let $\mu$ be the counting measure on $\ZZ^2$ and consider the operator
\be 
(\mathcal{P} h)(k) = \fr{1}{|k|} h_k,
\ee where $h$ is a mean-free function having the Fourier series representation
\be \label{fourrep}
h = \sum\limits_{j \in \ZZ^2 \setminus \left\{0\right\}} h_j e^{ij \cdot x},
\ee and $k \in \ZZ^2 \setminus \left\{0\right\}$.

On the one hand, we have
\be 
\mu (\left\{|\mathcal{P}h|> \lambda \right\})
= \mu (\left\{|k| <\lambda^{-1} |h_k|  \right\})
\le \mu (\left\{|k| \le \lambda^{-1} \|h\|_{L^1(\TT^2)} \right\}) 
\le C\lambda^{-2} \|h\|_{L^1(\TT^2)}^2,
\ee for any $\lambda > 0$ and periodic mean-free function $h \in L^1(\TT^2)$ with representation \eqref{fourrep}. Hence, we deduce that 
\be 
\sup\limits_{\lambda > 0} \lambda \mu (\left\{|\mathcal{P}h| > \lambda \right\})^{\fr{1}{2}} \le C\|h\|_{L^1(\TT^2)},
\ee and so the operator $\mathcal{P}$ is bounded from $L^1(\TT^2)$ into the weak Lebesgue space $L^{2, \infty}(\ZZ^2 {\setminus \{0\}} , \mu)$. On the other hand, we estimate 
\be \la{marcin1}
\mu (\left\{|\mathcal{P}h| > \lambda \right\}) = \mu (\left\{|k|^{-1}|h_k| > \lambda \right\})
\le \mu \left(\left\{\fr{|k|^{-2}}{2} >  \fr{\lambda}{2} \right\}\right) + \mu \left(\left\{\fr{|h_k|^2}{2} > \fr{\lambda}{2} \right\}\right),
\ee using the algebraic inequality $ab \le \fr{a^2}{2} + \fr{b^2}{2}$ and the countable subadditivity of the measure $\mu$. We note that 
\be 
\mu (\left\{\fr{|k|^{-2}}{2} >  \fr{\lambda}{2} \right\})
= \mu (\left\{|k| \le \fr{1}{\sqrt{\lambda}} \right\})
\le C\lambda^{-1},
\ee where the last inequality is an upper bound for the measure of the ball in $\ZZ^2$ with radius $\lambda^{-\fr{1}{2}}$. Moreover, applications of the Chebyshev's inequality  and Parseval's identity give rise to the estimate
\be \la{marcin2}
\mu (\left\{\fr{|h_k|^2}{2} > \fr{\lambda}{2} \right\})
\le \fr{1}{\lambda} \sum\limits_{k \in \ZZ^2 \setminus \left\{0\right\}} |h_k|^2
= C\lambda^{-1} \|h\|_{L^2(\TT^2)}^2,
\ee for any $\lambda >0$ and $h \in L^2(\TT^2)$ with series representation \eqref{fourrep}. Putting \eqref{marcin1}--\eqref{marcin2} together, we deduce that
\be 
\sup\limits_{\lambda >0} \lambda \mu (\left\{|\mathcal{P}h| > \lambda \right\}) \le {C( 1 + \|h\|_{L^2}^2)},
\ee for any $h \in L^2(\TT^2)$ represented by \eqref{fourrep}. This latter inequality implies that the linear operator $\mathcal{P}$ is bounded from $L^2(\TT^2)$ into the weak Lebesgue space $L^{1, \infty}(\ZZ^2{\setminus \{0\}}, \mu)$. Appealing to the Marcinkiewicz interpolation theorem, we deduce that for any mean-free function $h \in L^{\fr{4}{3}}(\TT^2)$, 
\be 
\| \mathcal{P} h \|_{\ell^{\fr{4}{3}}(\ZZ^2 {\setminus \{0\}})} \le C\|h\|_{L^{\fr{4}{3}}(\TT^2)},
\ee where $C$ is a positive universal constant independent of $h$. In particular, 
\be 
\| |k|^{-1} \tilde{\rho}_k\|_{\ell^{\fr{4}{3}}(\ZZ^2 {\setminus \{0\}})} \le C\|\tilde \rho\|_{L^{\fr{4}{3}}(\TT^2)},
\ee
 finishing the proof of Proposition \ref{ellipticprop}.
\end{proof}

\section{Exponential Ergodicity} \label{ExpErgFramework}

In this appendix, we reformulate the generalized coupling framework established in \cite{butkovsky2020generalized} for the reader's convenience.

Let $(E, \rho)$ be a Polish space. Let $(P_t)$ be a Feller Markov kernel satisfying the following conditions:
\begin{enumerate}
    \item There exists a measurable function $S: E \rightarrow [0, \infty]$ and a premetric $q$ on $E$ such that for any $v, \tilde{v} \in E$, there exists a couple of progressively measurable random processes $X^{v,\tilde{v}} = (X_t^{v,\tilde{v}})_{t \ge 0}, Y^{v,\tilde{v}}= (Y_t^{v,\tilde{v}})_{t \ge 0}$ satisfying the dissipativity bound
    \be 
q(X_t^{v, \tilde{v}}, Y_t^{v, \tilde{v}}) 
\le q(X_0, Y_0) e^{- \zeta t + \kappa \int_{0}^{t} S(X_s^{v, \tilde{v}})} ds,
    \ee for $t \ge 0$ and some constants $\zeta > 0$, $\kappa  \ge 0.$
    \item There exists a lower semicontinuous function $U: E \rightarrow [0, \infty)$ such that the energy estimate
    \be 
U(X_t^{v, \tilde{v}}) + \mu \int_{0}^{t} S(X_s^{v, \tilde{v}}) ds
\le U(X_0^{v, \tilde{v}}) + bt + M_t,
    \ee holds for any $t \ge 0$, where
    \begin{enumerate}
        \item $\mu > 0, b \ge 0$ are some constants such that $\zeta > \frac{\kappa b}{\mu}$;
        \item $M$ is a continuous local martingale with $M_0 = 0$ and so that its quadratic variation $\left<M\right>_t$ obeys
        \be 
d\left<M\right>_t \le b_1 S(X_t^{v, \tilde{v}}) dt + b_2 dt,
        \ee for $t \ge 0$ where $b_1, b_2 \ge 0.$
    \end{enumerate}
    \item Let $W$ be an $m$-dimensional Brownian motion, $m \ge 1$. There exists a constant $c>0$ such that for every $t  \ge 0$, $v, \tilde{v} \in E$, there exists a measurable function $\Phi = \Phi^{t, v, \tilde{v}}: \mathcal{C}[0,t] \rightarrow E$ and progressively measurable processes $\beta^{v, \tilde{v}}, \xi^{v, \tilde{v}}: \Omega \times [0,t] \rightarrow \RR^m$ such that
    \begin{enumerate}
        \item $d\xi_s^{v, \tilde{v}} = dW_s + \beta_s^{v, \tilde{v}} ds, s \in [0,t]$;
        \item $Law (\Phi(W_{[0,t]})) = P_t(\tilde{v}, 0)$ and $\Phi(\xi_{[0,t]}) = Y_t^{v, \tilde{v}}$;
        \item For each $s \in [0,t], |\beta_s|^2 \le cq(X_s^{v, \tilde{v}}, Y_s^{v, \tilde{v}})$.
    \end{enumerate} Here $f_{[0,t]} = \left\{f(s): s \in [0,t] \right\}$.
    \item There exists a measurable function $V: E \rightarrow \R_+$ such that for some $\gamma > 0, K > 0$, 
    \be 
\mathbb{E} V(X_t) \le V(v) - \gamma \mathbb{E} \int_{0}^{t} V(X_s) ds + Kt,
    \ee for $t \ge 0$ and $v \in E$. 
    \item For any $M > 0$, the function $U(\cdot)$ and $q(\cdot, \cdot)$ are bounded on the level sets $\left\{V \le M \right\}$ and $\left\{V \le M \right\} \times \left\{V \le M \right\}$ respectively. 
\end{enumerate}
If the Markov semigroup $P$ has an invariant measure $\pi$, then it is unique and 
\be 
W_{q(v,\tilde{v})^{\delta} \wedge q(\tilde{v}, v)^{\delta} \wedge 1} (P_t(v, \cdot), \pi)
\le C(1 + V(v)) e^{-rt},
\ee for $t \ge 0$, $v \in E$, some $C, r >0$, and an arbitrary $\delta > 0$. Here, 
\be \label{wasserstein}
W_d(\mu, \nu) = \inf\limits_{\lambda \in C(\mu, \nu)} \int_{E \times E} d(x,y) \lambda (dx, dy), \mu, \nu \in \mathcal{P}(E),
\ee where
\begin{enumerate}
    \item $\mathcal{P}(E)$ is the set of all Borel probability measures on $E$;
    \item $d = q(v, \tilde{v})^{\delta} \wedge q(\tilde{v}, v)^{\delta} \wedge 1$;
    \item $C(\mu,\nu)$ is the set of all couplings between $\mu$ and $\nu$, that is probability measures on $E \times E$ with marginals $\mu$ and $\nu$.
\end{enumerate}

  


\bibliographystyle{plain}
\bibliography{reference}

\begin{thebibliography}{10}

\bibitem{abdo2022space}
Elie Abdo and Mihaela Ignatova.
\newblock On the space analyticity of the nernst--planck--navier--stokes
  system.
\newblock {\em Journal of Mathematical Fluid Mechanics}, 24(2):51, 2022.

\bibitem{abdolong}
Elie Abdo and Mihaela Ignatova.
\newblock Long time dynamics of nernst-planck-navier-stokes systems.
\newblock {\em submitted}, 2023.

\bibitem{alkhad2022electrochemical}
Mohammad~A. Alkhadra, Xiao Su, Matthew~E. Suss, Huanhuan Tian, Eric~N. Guyes,
  Amit~N. Shocron, Kameron~M. Conforti, J.~Pedro de~Souza, Nayeong Kim, Michele
  Tedesco, Khoiruddin Khoiruddin, I~Gede Wenten, Juan~G. Santiago, T.~Alan
  Hatton, and Martin~Z. Bazant.
\newblock Electrochemical methods for water purification, ion separations, and
  energy conversion.
\newblock {\em Chemical reviews}, 122(16):13547--13635, 2022.

\bibitem{biler1994debye}
Piotr Biler, Waldemar Hebisch, and Tadeusz Nadzieja.
\newblock The debye system: existence and large time behavior of solutions.
\newblock {\em Nonlinear Analysis: Theory, Methods \& Applications},
  23(9):1189--1209, 1994.

\bibitem{bricmont2002exponential}
Jean Bricmont, Antti Kupiainen, and Rapha{\"e}l Lefevere.
\newblock Exponential mixing of the 2d stochastic navier-stokes dynamics.
\newblock {\em Communications in mathematical physics}, 230:87--132, 2002.

\bibitem{butkovsky2020generalized}
Oleg Butkovsky, Alexei Kulik, and Michael Scheutzow.
\newblock Generalized couplings and ergodic rates for spdes and other markov
  models.
\newblock {\em Annals of applied probability}, 30(1):1--39, 2020.

\bibitem{cole1965electrodiffusion}
Kenneth~S Cole.
\newblock Electrodiffusion models for the membrane of squid giant axon.
\newblock {\em Physiological Reviews}, 45(2):340--379, 1965.

\bibitem{constantin1988navier}
Peter Constantin and Ciprian Foias.
\newblock {\em Navier-stokes equations}.
\newblock University of Chicago Press, 1988.

\bibitem{constantin2019nernst}
Peter Constantin and Mihaela Ignatova.
\newblock On the nernst--planck--navier--stokes system.
\newblock {\em Archive for Rational Mechanics and Analysis}, 232:1379--1428,
  2019.

\bibitem{constantin2021nernst}
Peter Constantin, Mihaela Ignatova, and Fizay-Noah Lee.
\newblock Nernst--planck--navier--stokes systems far from equilibrium.
\newblock {\em Archive for Rational Mechanics and Analysis}, 240:1147--1168,
  2021.

\bibitem{constantin2022nernst}
Peter Constantin, Mihaela Ignatova, and Fizay-Noah Lee.
\newblock Nernst-planck-navier-stokes systems near equilibrium.
\newblock {\em Pure and Applied Functional Analysis}, 7:175--196, 2022.

\bibitem{da2014stochastic}
Giuseppe Da~Prato and Jerzy Zabczyk.
\newblock {\em Stochastic equations in infinite dimensions}.
\newblock Cambridge university press, 2014.

\bibitem{davidson2016dynamical}
Scott~M Davidson, Matthias Wessling, and Ali Mani.
\newblock On the dynamical regimes of pattern-accelerated electroconvection.
\newblock {\em Scientific reports}, 6(1):22505, 2016.

\bibitem{gajewski1986basic}
Herbert Gajewski and Konrad Gr{\"o}ger.
\newblock On the basic equations for carrier transport in semiconductors.
\newblock {\em Journal of mathematical analysis and applications},
  113(1):12--35, 1986.

\bibitem{gao2014high}
Jun Gao, Wei Guo, Dan Feng, Huanting Wang, Dongyuan Zhao, and Lei Jiang.
\newblock High-performance ionic diode membrane for salinity gradient power
  generation.
\newblock {\em Journal of the American Chemical Society}, 136(35):12265--12272,
  2014.

\bibitem{glatt2017unique}
Nathan Glatt-Holtz, Jonathan~C Mattingly, and Geordie Richards.
\newblock On unique ergodicity in nonlinear stochastic partial differential
  equations.
\newblock {\em Journal of Statistical Physics}, 166:618--649, 2017.

\bibitem{goldman1989electrodiffusion}
David~E Goldman.
\newblock Electrodiffusion in membranes.
\newblock In {\em Membrane Transport: People and Ideas}, pages 251--259.
  Springer, 1989.

\bibitem{grafakos2008classical}
Loukas Grafakos et~al.
\newblock {\em Classical fourier analysis}, volume~2.
\newblock Springer, 2008.

\bibitem{hairer2002exponential}
Martin Hairer.
\newblock Exponential mixing properties of stochastic pdes through asymptotic
  coupling.
\newblock {\em Probability theory and related fields}, 124(3):345--380, 2002.

\bibitem{hairer2011theory}
Martin Hairer and Jonathan Mattingly.
\newblock {A Theory of Hypoellipticity and Unique Ergodicity for Semilinear
  Stochastic PDEs}.
\newblock {\em Electronic Journal of Probability}, 16:658 -- 738, 2011.

\bibitem{hairer2011asymptotic}
Martin Hairer, Jonathan~C Mattingly, and Michael Scheutzow.
\newblock Asymptotic coupling and a general form of harris’ theorem with
  applications to stochastic delay equations.
\newblock {\em Probability theory and related fields}, 149:223--259, 2011.

\bibitem{jasielec2021electrodiffusion}
Jerzy~J Jasielec.
\newblock Electrodiffusion phenomena in neuroscience and the
  nernst--planck--poisson equations.
\newblock {\em Electrochem}, 2(2):197--215, 2021.

\bibitem{jungnickel2004coupled}
Christian Jungnickel, David Smith, and Stephen Fityus.
\newblock Coupled multi-ion electrodiffusion analysis for clay soils.
\newblock {\em Canadian geotechnical journal}, 41(2):287--298, 2004.

\bibitem{koch2004biophysics}
Christof Koch.
\newblock {\em Biophysics of computation: information processing in single
  neurons}.
\newblock Oxford university press, 2004.

\bibitem{kuksin2001coupling}
Sergei Kuksin and Armen Shirikyan.
\newblock A coupling approach to randomly forced nonlinear pde's. i.
\newblock {\em Communications in Mathematical Physics}, 221:351--366, 2001.

\bibitem{kuksin2002coupling}
Sergei Kuksin and Armen Shirikyan.
\newblock Coupling approach to white-forced nonlinear pdes.
\newblock {\em Journal de math{\'e}matiques pures et appliqu{\'e}es},
  81(6):567--602, 2002.

\bibitem{lee2016membrane}
Anna Lee, Jeffrey~W Elam, and Seth~B Darling.
\newblock Membrane materials for water purification: design, development, and
  application.
\newblock {\em Environmental Science: Water Research \& Technology},
  2(1):17--42, 2016.

\bibitem{lee2018diffusiophoretic}
Hyomin Lee, Junsuk Kim, Jina Yang, Sang~Woo Seo, and Sung~Jae Kim.
\newblock Diffusiophoretic exclusion of colloidal particles for continuous
  water purification.
\newblock {\em Lab on a Chip}, 18(12):1713--1724, 2018.

\bibitem{lopreore2008computational}
Courtney~L Lopreore, Thomas~M Bartol, Jay~S Coggan, Daniel~X Keller, Gina~E
  Sosinsky, Mark~H Ellisman, and Terrence~J Sejnowski.
\newblock Computational modeling of three-dimensional electrodiffusion in
  biological systems: application to the node of ranvier.
\newblock {\em Biophysical journal}, 95(6):2624--2635, 2008.

\bibitem{mattingly2003recent}
Jonathan Mattingly.
\newblock On recent progress for the stochastic navier stokes equations.
\newblock {\em Journ{\'e}es Equations aux d{\'e}riv{\'e}es partielles}, pages
  1--52, 2003.

\bibitem{mock1983analysis}
Michael~Stephen Mock.
\newblock Analysis of mathematical models of semiconductor devices.
\newblock {\em (No Title)}, 1983.

\bibitem{mori2009numerical}
Yoichiro Mori and Charles Peskin.
\newblock A numerical method for cellular electrophysiology based on the
  electrodiffusion equations with internal boundary conditions at membranes.
\newblock {\em Communications in Applied Mathematics and Computational
  Science}, 4(1):85--134, 2009.

\bibitem{nicholson2000diffusion}
Charles Nicholson, Kevin~C Chen, Sabina Hrab{\v{e}}tov{\'a}, and Lian Tao.
\newblock Diffusion of molecules in brain extracellular space: theory and
  experiment.
\newblock {\em Progress in brain research}, 125:129--154, 2000.

\bibitem{pods2013electrodiffusion}
Jurgis Pods, Johannes Sch{\"o}nke, and Peter Bastian.
\newblock Electrodiffusion models of neurons and extracellular space using the
  poisson-nernst-planck equations—numerical simulation of the intra-and
  extracellular potential for an axon model.
\newblock {\em Biophysical journal}, 105(1):242--254, 2013.

\bibitem{qian1989electro}
Ning Qian and TJ~Sejnowski.
\newblock An electro-diffusion model for computing membrane potentials and
  ionic concentrations in branching dendrites, spines and axons.
\newblock {\em Biological Cybernetics}, 62(1):1--15, 1989.

\bibitem{rubinstein2000electro}
Isaak Rubinstein and Boris Zaltzman.
\newblock Electro-osmotically induced convection at a permselective membrane.
\newblock {\em Physical Review E}, 62(2):2238, 2000.

\bibitem{rubinstein2008direct}
Shmuel~M Rubinstein, Gor Manukyan, Adrian Staicu, Issac Rubinstein, Boris
  Zaltzman, Rob~GH Lammertink, Frieder Mugele, and Matthias Wessling.
\newblock Direct observation of a nonequilibrium electro-osmotic instability.
\newblock {\em Physical review letters}, 101(23):236101, 2008.

\bibitem{ryham2009existence}
Rolf~J Ryham.
\newblock Existence, uniqueness, regularity and long-term behavior for
  dissipative systems modeling electrohydrodynamics.
\newblock {\em arXiv preprint arXiv:0910.4973}, 2009.

\bibitem{savtchenko2017electrodiffusion}
Leonid~P Savtchenko, Mu~Ming Poo, and Dmitri~A Rusakov.
\newblock Electrodiffusion phenomena in neuroscience: a neglected companion.
\newblock {\em Nature Reviews Neuroscience}, 18(10):598--612, 2017.

\bibitem{schilling2014brownian}
Ren{\'e}~L Schilling and Lothar Partzsch.
\newblock {\em Brownian motion: an introduction to stochastic processes}.
\newblock Walter de Gruyter GmbH \& Co KG, 2014.

\bibitem{schmuck2009analysis}
Markus Schmuck.
\newblock Analysis of the navier--stokes--nernst--planck--poisson system.
\newblock {\em Mathematical Models and Methods in Applied Sciences},
  19(06):993--1014, 2009.

\bibitem{tan2016computational}
Jinwang Tan and Emily~M Ryan.
\newblock Computational study of electro-convection effects on dendrite growth
  in batteries.
\newblock {\em Journal of Power Sources}, 323:67--77, 2016.

\bibitem{weinan2001gibbsian}
E~Weinan, Jonathan~C Mattingly, and Yakov Sinai.
\newblock Gibbsian dynamics and ergodicity for the stochastically forced
  navier-stokes equation.
\newblock {\em Comm. Math. Phys}, 224(1):83--106, 2001.

\bibitem{yang2019review}
Zi~Yang, Yi~Zhou, Zhiyuan Feng, Xiaobo Rui, Tong Zhang, and Zhien Zhang.
\newblock A review on reverse osmosis and nanofiltration membranes for water
  purification.
\newblock {\em Polymers}, 11(8):1252, 2019.

\bibitem{zaltzman2007electro}
Boris Zaltzman and Isaak Rubinstein.
\newblock Electro-osmotic slip and electroconvective instability.
\newblock {\em Journal of Fluid Mechanics}, 579:173--226, 2007.

\bibitem{zhu2020ion}
Haitao Zhu, Bo~Yang, Congjie Gao, and Yaqin Wu.
\newblock Ion transfer modeling based on nernst--planck theory for saline water
  desalination during electrodialysis process.
\newblock {\em Asia-Pacific Journal of Chemical Engineering}, 15(2):e2410,
  2020.

\end{thebibliography}

\end{document}